\documentclass{amsart}

\usepackage{color}

\usepackage{bbm} 
\usepackage{graphicx, epsfig}
\usepackage{epstopdf}
\usepackage{amsmath, amssymb, latexsym, euscript, amscd}
\usepackage{url}
\usepackage[all]{xy}
\usepackage{psfrag}
\usepackage{mathrsfs}
\usepackage{caption, wrapfig, multirow, tabularx, mathrsfs,verbatim}
\usepackage{subcaption}
\usepackage{tikz,tikz-cd}
\usepackage{todonotes} 

\usepackage{cleveref}

\usetikzlibrary{decorations.pathreplacing,decorations.markings}

\usetikzlibrary{backgrounds}

\tikzset{
  on each segment/.style={
    decorate,
    decoration={
      show path construction,
      moveto code={},
      lineto code={
        \path [#1]
        (\tikzinputsegmentfirst) -- (\tikzinputsegmentlast);
      },
      curveto code={
        \path [#1] (\tikzinputsegmentfirst)
        .. controls
        (\tikzinputsegmentsupporta) and (\tikzinputsegmentsupportb)
        ..
        (\tikzinputsegmentlast);
      },
      closepath code={
        \path [#1]
        (\tikzinputsegmentfirst) -- (\tikzinputsegmentlast);
      },
    },
  },
  mid arrow/.style={postaction={decorate,decoration={
        markings,
        mark=at position .25 with {\arrow[#1]{stealth}}
      }}},
}

\tikzset{
  on each segment/.style={
    decorate,
    decoration={
      show path construction,
      moveto code={},
      lineto code={
        \path [#1]
        (\tikzinputsegmentfirst) -- (\tikzinputsegmentlast);
      },
      curveto code={
        \path [#1] (\tikzinputsegmentfirst)
        .. controls
        (\tikzinputsegmentsupporta) and (\tikzinputsegmentsupportb)
        ..
        (\tikzinputsegmentlast);
      },
      closepath code={
        \path [#1]
        (\tikzinputsegmentfirst) -- (\tikzinputsegmentlast);
      },
    },
  },
  mid arrow1/.style={postaction={decorate,decoration={
        markings,
        mark=at position .5 with {\arrow[#1]{stealth}}
      }}},
}

\tikzset{
  on each segment/.style={
    decorate,
    decoration={
      show path construction,
      moveto code={},
      lineto code={
        \path [#1]
        (\tikzinputsegmentfirst) -- (\tikzinputsegmentlast);
      },
      curveto code={
        \path [#1] (\tikzinputsegmentfirst)
        .. controls
        (\tikzinputsegmentsupporta) and (\tikzinputsegmentsupportb)
        ..
        (\tikzinputsegmentlast);
      },
      closepath code={
        \path [#1]
        (\tikzinputsegmentfirst) -- (\tikzinputsegmentlast);
      },
    },
  },
  mid arrow2/.style={postaction={decorate,decoration={
        markings,
        mark=at position .75 with {\arrow[#1]{stealth}}
      }}},
}

\usepackage{pgfplots}
\usetikzlibrary{backgrounds}

\newcounter{notes}%
\newcommand{\marginnote}[1]{
\refstepcounter{notes}  
\nolinebreak
$\hspace{-5pt}{}^{\text{\tiny \rm \arabic{notes}}}$
\marginpar{\tiny \arabic{notes}) #1}}

\definecolor{darkgreen}{rgb}{0.0, 0.5, 0.0}

\newtheorem{theorem}[equation]{Theorem}
\numberwithin{equation}{section}
\newtheorem{lemma}[equation]{Lemma}
\newtheorem{corollary}[equation]{Corollary}
\newtheorem{definition}[equation]{Definition} 
\newtheorem{proposition}[equation]{Proposition}

\newtheorem{remarks}[equation]{Remarks}

\def\smallskip{\vspace{.15cm}}
\def\medskip{\vspace{.3cm}}
\def\text{\mbox}
\def\RR{{\mathbb R}}
\def\RRP{{\mathbb R}_+}
\def\CC{{\mathbb C}}
\def\AA{{\mathbb A}}
\def\EE{{\mathbb E}}
\def\ZZ{{\mathbb Z}}
\def\PP{{\mathbb P}}
\def\HH{{\mathbb H}}

\def\P{{\mathbb P}}
\def\R{{\mathbb R}}
\def\RPn{\operatorname{\mathbb{R}P}^n}
\def\RP2{\operatorname{\mathbb{R}P}^2}
\def\RP3{\operatorname{\mathbb{R}P}^3}
\def\RP{\operatorname{\mathbb{R}P}}

\def\Fr{{\operatorname{Fr\,}}}

\def\interior{\operatorname{int}}
\def\SL{\operatorname{SL}}

\def\PSL{\operatorname{PSL}}

\def\PO{\operatorname{PO}}
\def\PGL{\operatorname{PGL}}
\def\GL{\operatorname{GL}}

\def\Aff{\operatorname{Aff}}
\def\Hom{\operatorname{Hom}}

\def\cl{\operatorname{cl}}

\def\Aff{\operatorname{Aff}}
\def\tr{\operatorname{tr}}

\def\Isom{\operatorname{Isom}}
\def\Hex{\operatorname{\mathbb Hex}}
\def\Par{\operatorname{\mathbb Par}}
\def\Diag{\operatorname{Diag}}

\def\flow{\Phi}

\def\C2{\operatorname{C^2}}

\def\Ccal{\mathcal C}
\def\Bcal{\mathcal B}

\def\Dcal{\mathcal D}
\def\Hcal{\mathcal H}
\def\Lcal{\mathcal L}

\def\Wcal{\mathcal W}
\def\Image{\operatorname{Im}}
\def\opensimplex{{\red\mathring{\vartriangle}}}

\def\lat{\operatorname{Lat}}
\def\clat{\operatorname{\mathcal Mod}}
\def\mlat{\operatorname{Lat}_m}
\def\cmlat{\operatorname{\mathcal T}}
\def\SS{\mathcal E}

\def\UT{\operatorname{UT}}
\def\dvol{\operatorname{dvol}}

\def\secondfund{\operatorname{II}}
\def\bdy{\partial}
\def\vol{\operatorname{vol}}
\def\infimum{\operatorname{infimum}}

\def\Log{\operatorname{\ell og}}
\def\deriv{\operatorname{D}}
\def\param{\operatorname{\Theta}}
\def\eT{T}
\def\eTalg{\mathfrak t}

\def\psicoef{\psi}
\def\ppsi{\psi}
\def\psicoefi{\psicoef} 

\def\ur{{\bf u}}
\def\rank{{\bf r}}
\def\type{{\bf t}}
\def\minusone{-{\bf I}}
\def\horogeom{\operatorname{Horo}}

\newcommand{\Abs}[1]{\left|\left|#1\right|\right|}
\newcommand{\abs}[1]{\left|#1\right|}

\definecolor{back}{RGB}{255,255,255}
\definecolor{fore}{RGB}{0,0,0}
\definecolor{title}{RGB}{255,0,90}

\definecolor{green}{rgb}{0.0, 0.5, 0.0}
\definecolor{purple}{rgb}{0.5, 0.0, 0.5}
\definecolor{bluegreen}{rgb}{0.0,0.5, 0.5}
\definecolor{orange}{rgb}{1,0.5, 0.1}
\definecolor{redgreen}{rgb}{0.5, 0.5, 0.0}

\def\blue{\color{fore}}
\def\red{\color{fore}}

\def\green{\color{green}}

\def\green{\color{green}}

\def\g2{{\green 2}}
\def\core{\operatorname{core}}

\newcommand{\bv}{\left[\begin{array}{c}}
\newcommand{\ev}{\end{array}\right]}
\newcommand{\bbmat}{\begin{bmatrix}} 
\newcommand{\ebmat}{\end{bmatrix}}
\newcommand{\bmat}{\begin{matrix}} 
\newcommand{\emat}{\end{matrix}}
\newcommand{\bpmat}{\begin{pmatrix}} 
\newcommand{\epmat}{\end{pmatrix}}

 \pgfplotsset{compat=1.12}

\title[Generalized Cusps]{Generalized Cusps in Real Projective Manifolds: Classification}
\date{\today}
\author{Samuel A. Ballas, Daryl Cooper,  and Arielle Leitner}

\begin{document}
\maketitle

\begin{abstract}
A generalized cusp $C$
is diffeomorphic to $[0,\infty)$ times a closed Euclidean manifold. Geometrically
$C$ is the quotient of a properly convex domain {\red in $\RPn$} by a lattice, $\Gamma$, in one of a family
of affine groups $G(\ppsi)$, parameterized by a point $\ppsi$ in the (dual closed)
 Weyl chamber for $\SL(n+1,\RR)$,
and $\Gamma$ determines the cusp up to equivalence. These affine 
 groups correspond to certain  fibered geometries, each of which is
a  bundle over an open simplex 
 with fiber
 a horoball in hyperbolic space, and
 the lattices are classified by certain Bieberbach
groups plus some auxiliary data.
The cusp has finite Busemann measure if and only if $G(\ppsi)$ contains unipotent elements.
There is a natural underlying Euclidean structure on $C$ unrelated to the Hilbert metric. 
\end{abstract}

A {\em generalized cusp} is a  properly convex projective manifold  $C=\Omega/\Gamma$ where
$\Omega\subset\RP^n$  is a properly convex set 
and $\Gamma\subset\PGL(n+1,\RR)$
is a virtually abelian discrete group that preserves $\Omega$. We also require that $\partial C$
is compact and strictly convex (contains no line segment) and that there is a diffeomorphism 
$h:[0,\infty)\times\partial C\longrightarrow C$. {\blue See Definition \ref{gencusp}}(a).  

An example is a cusp in a hyperbolic manifold that is the quotient of a {\red closed} horoball. 
It follows from  \cite{CLT1} that every
generalized cusp in a  {\em strictly} convex manifold of finite volume is equivalent to a {\em standard cusp}, 
i.e. a cusp  in a hyperbolic
manifold.  
A generalized cusp is {\em homogeneous} if $\PGL(\Omega)$ (the group of projective transformations that preserves $\Omega$) 
acts transitively on $\bdy\Omega$. It was shown in
\cite{CLT2} that every generalized cusp is equivalent to a homogeneous one and,
 that if the holonomy of a generalized cusp contains no hyperbolic elements, then it is equivalent to a standard cusp. Furthermore, by \cite{CLT2} it follows that generalized cusps often occur as ends of properly convex manifolds obtained by deforming finite volume hyperbolic manifolds.
 
 {\red Here is an outline of the main new results of this paper. Given  $\psi\in\Hom(\RR^{n},\RR)$
 with $\ppsi(e_1)\ge\ppsi(e_2)\ge\cdots \ge \ppsi(e_n)\ge 0$, there is a properly
 convex domain $\Omega(\psi)\subset\RPn$: see Definition \ref{psidomain}. For $\psi\ne0$ the
  {\emph {cusp Lie group}}
 $G(\psi)=\PGL(\Omega(\psi))$, and for $\psi=0$  it is the subgroup of non-hyperbolic elements.
  In each case $G(\psi)$ acts transitively on $\bdy\Omega(\ppsi)$.
 A {\bf $\psi$-cusp} is the quotient of $\Omega(\psi)$ by a lattice in $G(\psi)$. 
  Two generalized cusps $C$ and $C'$
are {\em equivalent} if there is a generalized cusp $C''$ and projective embeddings,  that are also homotopy equivalences,
of $C''$ into both $C$ and $C'$, and they are all diffeomorphic. 
\begin{theorem}[Uniformization]
\label{genislambda} 
Every  generalized cusp is equivalent to  a $\ppsi$-cusp.
 \end{theorem}

 The geometry of a $\psi$-cusp depends on the {\emph type} $\type=\type_{\ppsi}$, which is the number of $i$ with $\psi(e_i)\neq 0$, and the \emph{unipotent rank} 
 $\ur(\ppsi)=\max(n-\type-1,0)$ {\blue is the dimension of the unipotent subgroup of $G(\ppsi)$.}  The {\em ideal boundary} of $\Omega:=\Omega(\psi)$ is
 $\bdy_{\infty}\Omega:=\cl(\Omega)\setminus\Omega\cong\Delta^{\blue\min(n-1,\type)}$. 
 There is a unique
 supporting hyperplane $\RP^{n-1}_{\infty}$ to $\Omega$ that contains $\bdy_{\infty}\Omega$ 
 so $\AA(\Omega):=\RPn\setminus \RP^{n-1}_{\infty}$
 is the unique affine
 patch in which $\Omega$ is properly embedded. Hence $\Omega$ has a well defined affine structure,
 and $\psi$-cusps inherit a unique affine structure that is a stiffening of the projective structure.
 The {\em (non-ideal or manifold) boundary of $\Omega$}  is a smooth, strictly-convex 
 hypersurface $\bdy\Omega:=\Omega\setminus\interior(\Omega)$ that is properly embedded  in 
 $\AA(\Omega)$. 
 Since $\Omega$ is 
 convex the {\em frontier} $\Fr(\Omega):=\bdy(\cl\Omega)\cong S^{n-1}$ and
   $\Fr(\Omega)=\bdy\Omega\sqcup\bdy_{\infty}\Omega$.

Types $0$ and $n$ are familiar.  For type $0$ then $\Fr\Omega(0)$ is a round sphere and $\bdy_{\infty}\Omega(0)$
is a single point. Thus
$\Omega(0)$ may be projectively identified with a closed horoball
 in the projective model of hyperbolic space $\HH^n$, and also with $\cl\HH^n\setminus\{\infty\}$. Then
 $\bdy\Omega=\bdy_{\infty}\HH^n\setminus\{\infty\}$ and $\bdy_{\infty}\Omega=\{\infty\}$.
  Moreover $G(0)$ is isomorphic to the subgroup of $\Isom(\HH^n)$ consisting of parabolics and elliptics
  that fix $\infty$. Whence these generalized cusps are \emph{standard}.  At the other extreme, when $\type=n$, there is an $n$-simplex 
  $\Delta^n\subset\RP^n$ and  $\Omega:=\Omega(\ppsi)\subset\interior(\Delta^n)$ 
  and 
  $\bdy\Omega$ is a properly embedded, convex smooth hypersurface that separates 
  $\interior(\Delta)$ into two components, one of which is $\interior(\Omega)$.
  Then $\bdy_{\infty}\Omega=\Delta^{n-1}$
   is a face of $\Delta^n$. Moreover
$G(\ppsi)\subset\PGL(\Delta^n)$ and thus
contains a finite index subgroup that is \emph{diagonalizable} over the reals. 

When $0<\type<n$, there is an affine
projection $\Omega:=\Omega(\ppsi)\to\interior(\Delta^{\type})$ 
with fibers that are  projectively equivalent to horoballs in  $\HH^{\ur+1}$.
In this case $\bdy_{\infty}\Omega\cong\Delta^{\type}$.
In fact one can regard a generalized cusp as a kind of fiber product
 of  a diagonalizable cusp of dimension $\type$ and a standard cusp of dimension $1+\ur$,
 and also
as a deformation of a standard cusp, where the boundary at infinity is expanded out into  a simplex.
In particular this
results in a {\em flat simplex} $\Delta^{\type}$ in the ideal boundary of any domain covering a manifold  
that contains  generalized cusps of type $\type>0$. 
In the sense of Klein geometries,
 $(G(\ppsi),\bdy\Omega(\ppsi))$ is
 a subgeometry of  Euclidean geometry. 
The orbits of $G(\ppsi)$ form a codimension-1 foliation and the leaves are called {\em{horospheres}}.
There is a 1-parameter group called the {\em{radial flow}} that centralizes $G(\psi)$ and the orbits are
orthogonal to the horospheres. These two foliations give a {\em natural} product structure on a generalized cusp.
}

{\red The following is more easily understood
 after first reading Section \ref{surfaces} about surfaces, then  Section \ref{3mfd} about 3-manifolds.}
The next goal is to classify cusps up to equivalence. For this it is useful to introduce {\em marked cusps}  and {\em marked} lattices
(see Section \ref{sec:classify} for the definition and more discussion).
A rank-2 cusp in a hyperbolic 3-manifold is determined by a  {\em cusp shape}, which is  a Euclidean torus
defined up to similarity.  This shape is usually described by a  complex number $x+iy$ with $y>0$, that uniquely determines a {\em marked} cusp.
Unmarked cusps are described by the modular surface $\HH^2/\PSL(2,\ZZ)$.

More generally, a maximal-rank cusp in a hyperbolic $n$-manifold is determined by  a lattice in $\Isom(\EE^{n-1})$
up to conjugacy and rescaling. We extend this result by showing when $\psi\ne0$ that a generalized cusp of dimension $n$
with holonomy in $G(\ppsi)$ 
 is determined by a pair
$([\Gamma],A\cdot O(\ppsi))$ consisting
of the conjugacy class of a lattice $\Gamma\subset\Isom(\EE^{n-1})$, 
 and an {\em anisotropy parameter}   which we now describe.

The second fundamental form on   $\bdy\Omega\subset\AA(\Omega)$
is {\red conformally equivalent to} a  Euclidean metric 
that is preserved by the action of $G(\ppsi)$. 
This 
identifies $G(\ppsi)$ with a subgroup of $\Isom(\EE^{n-1})$, and $G(\ppsi)=T(\ppsi)\rtimes O(\ppsi)$ 
is the semi-direct product of the \emph{translation subgroup}, $T(\ppsi)\cong\RR^{n-1}$, and a 
closed subgroup $O(\ppsi)\subset O(n-1)$ 
that fixes some point $p$ in $\bdy\Omega$,  { see Theorem \ref{Gppsistructure}.}
The Euclidean structure identifies $\Gamma$ with a lattice in $\Isom(\EE^{n-1})$. This lattice
is unique up to conjugation by an element of $O(\ppsi)$.
The {\em anisotropy parameter} is a left coset $A\cdot O(\ppsi)$ in $O(n)$ that determines the 
$O(\ppsi)$-conjugacy class.
The group $O(\ppsi)$ is computed in Proposition \ref{olambda}.


Given a Lie group $G$, the set of $G$-conjugacy classes of marked lattices in $G$ is denoted $\cmlat(G)$.
Define $\cmlat(\Isom(\EE^{n-1}),\ppsi)\subset \cmlat(\Isom(\EE^{n-1}))$ to be the subset of conjugacy classes of 
{\em marked} Euclidean lattices 
with rotational part of the holonomy (up to conjugacy) in $O(\ppsi)$.
The classification of  generalized cusps (up to equivalence) is completed by:

\begin{theorem}[Classification]\label{Classification Theorem}
\label{Classification Theorem} \
 \
 \begin{enumerate}
\item[(i)] If $\Gamma$ and $\Gamma'$ are lattices in $G(\ppsi)$ TFAE 
\begin{enumerate}
\item $\Omega(\ppsi)/\Gamma$ and $\Omega(\ppsi)/\Gamma'$ are  equivalent generalized cusps
\item $\Gamma$ and $\Gamma'$ are conjugate in $\PGL(n+1,\RR)$
\item $\Gamma$ and $\Gamma'$ are conjugate in $\PGL(\Omega(\ppsi))$ 
\end{enumerate} 
\item[(ii)] A lattice in $G(\ppsi)$ is conjugate in $\PGL(n+1,\RR)$ into $G(\ppsi')$ iff  $G(\psi)$ is conjugate to $G(\psi')$. 
\item[(iii)]$G(\ppsi)$ is conjugate in $\PGL(n+1,\RR)$  to $G(\ppsi')$ iff $\ppsi'=t\cdot\ppsi$ for some $t>0$.
\item[(iv)]  $\PGL(\Omega(\ppsi))=G(\ppsi)$ when $\ppsi\ne0$
\item[(v)] When $\psi\ne0$ the map $\param:\cmlat(\Isom(\EE^{n-1}),\ppsi)\times (O(n-1)/O(\ppsi))\longrightarrow\cmlat(G(\ppsi))$ defined in (\ref{thetadef}) is a bijection. 
\end{enumerate}
\end{theorem} 

One might view (ii) in the context of super-rigidity: {\em an embedding of a lattice determines an embedding of the Lie group that contains it}. 
Throughout this paper we repeatedly stumble over two exceptional cases. 
A generalized cusp with $\ppsi=0$ is projectively equivalent to a cusp in a hyperbolic manifold. This is the only case when $\PGL(\Omega(\ppsi))$
is strictly larger than $G(\ppsi)$, and occurs because there are elements of $\PGL(\Omega(0))\subset\Isom(\HH^n)$ that permute horospheres. These elements are
hyperbolic  isometries of $\HH^n$ that fix $\red\infty$. This accounts for the fact that the equivalence class of a cusp in a hyperbolic manifold 
 is determined by the {\em similarity class} 
 ($\PGL(\Omega(\ppsi))$-conjugacy class) of the
lattice, rather than the {\em $G(\psi)$-conjugacy class}, as in every other case.
 The other exceptional case is the diagonalizable case $\type=n$, and in this case the radial flow is hyperbolic
instead of parabolic. Fortunately both these exceptional cases are easy to understand, but tend to require proofs that consider
various cases.
 
Let $\Ccal^n$ denote the set of equivalence classes of generalized cusps of dimension $n$.  Let
$\clat^n$ denote the (disjoint) union over all $\ppsi$  with $\ppsi(e_1)=1$ of conjugacy classes of (unmarked) lattice in $G(\ppsi)$,
union lattices in $G(0)\cong\Isom(\EE^{n-1})$ up to conjugacy and scaling. 
Every non-standard
 generalized cusp is equivalent
to one given by a lattice in $G(\ppsi)$ with $\psi(e_1)=1$, that is unique up to conjugacy in $G(\ppsi)$ giving:

\begin{corollary}[Cusps classified by lattices]
\label{cor:cusplattice} There is a bijection $F:\clat^n\longrightarrow\Ccal^n$ defined for $[\Gamma]\in\clat^n$ by $F([\Gamma])=[\Omega(\ppsi)/\Gamma]$
when $\Gamma$ is a lattice in $G(\ppsi)$.
\end{corollary}
\begin{corollary}[Standard parabolics]	
\label{stdparabolics} Suppose $M=\Omega/\Gamma$ 
is a properly convex $n$-manifold such that every end of $M$ is a generalized cusp.
If $[A]\in\Gamma$ and {\red $d_{\Omega}$ is the Hilbert metric on $\Omega$, and if  $\inf\{d_{\Omega}(x,[A]x )|x\in \Omega\}=0$},
then $[A]$ is the holonomy of  an element of $\pi_1C$ for some generalized cusp $C\subset  M$,
 and $[A]$ is conjugate  in $\PGL(n+1,\RR)$ {\red to a parabolic in} $\PO(n,1)$.
\end{corollary}

{ Generalized cusps are modeled on the geometries $(G(\ppsi),\Omega(\ppsi))$, and these
are all  
isomorphic to subgeometries of Euclidean geometry, see Corollary \ref{lambdageomEuclidean}.  In fact 
 there is a natural
Euclidean metric:

\begin{theorem}[Underlying Euclidean structure]\label{euc_str}
{\red There is a  metric $\beta$ on $\Omega=\Omega(\ppsi)$
 that is preserved by $G(\ppsi)$ and by the radial flow,
and $(\Omega, \beta)$ is isometric to $\RR^{n-1}\times[0,\infty)$ with the usual Euclidean metric.  
The restriction of $\beta$ to $\bdy\Omega$ is {\red conformally equivalent to} the second fundamental form of $\bdy\Omega$ in $\AA(\Omega)$.}\end{theorem}

{\red Theorem \ref{euc_str}} implies a generalized cusp has an {\em underlying Euclidean
structure},
and also an underlying {\em hyperbolic structure}, see Theorem \ref{genishyp}.}
It is a well known that, if $C$ is a maximal rank cusp in a hyperbolic manifold $M$, then $C$ has finite hyperbolic volume. {\red For properly convex manifolds there is a natural notion of volume (see Section \ref{hilbmet} for details). 
}

\begin{theorem}[parabolic $\Leftrightarrow$ finite vol]
\label{volthm} 
 Suppose $C=\Omega/\Gamma$ is a generalized cusp
 in the interior a properly convex manifold $M$    and $\Gamma$ is conjugate into $G(\psi)$.
  Then $C$ has finite volume in $M$ (with respect to the Hausdorff measure induced by the Hilbert metric on $M$) iff   $\ur(\ppsi)>0$
 iff $G(\ppsi)$ contains a  parabolic element.
\end{theorem}

The original definition \cite{CLT2} of {\em generalized cusp} differs from the one in the introduction
by replacing the word {\em abelian} by {\em nilpotent}. {\red To avoid confusion, we have decided to call the generalized cusps of \cite{CLT1} \emph{g-cusps}. See Definition \ref{gencusp} for the precise definition.} The reason {\em nilpotent} was
used originally is the connection between cusps and the Margulis lemma.
A consequence of the analysis in this
paper is that these definitions are equivalent:

\begin{theorem}[nilpotent $\Rightarrow$ abelian] 
\label{vnilisvabelian}Suppose $C=\Omega/\Gamma$ is a properly convex manifold and $C\cong\bdy C\times[0,\infty)$
and $\bdy C$ is compact and strictly convex, and $\pi_1C$ is virtually nilpotent. Then $C$ is a generalized cusp
and $\pi_1C$ is virtually abelian.
\end{theorem}

Another aspect of the definition of {\em generalized cusp} is that $\bdy C$ is {\em compact}.
 In the theory of Kleinian groups,
rank-1 cusps are important. These are diffeomorphic to $A\times[0,\infty)$
 where $A$ is a (non-compact) annulus. For
hyperbolic manifolds of higher dimensions
there are more possibilities, however  the 
fundamental group of such a cusp  is always virtually abelian.
This is not the case for properly convex manifolds. In \cite{Coop} there is an example of a strictly convex
manifold with unipotent (parabolic) holonomy, and with fundamental group the integer Heisenberg group.
There might to be a nice theory of properly convex manifolds $C\cong\bdy C\times[0,\infty)$ with $\pi_1C$ virtually nilpotent
and $\bdy C$ strictly convex, but without requiring $\bdy C$ to be compact.

The definition of the term {\em generalized cusp} was the end result of a lot of experimentation with definitions, and was
modified as more was discovered about their nature. In retrospect it turns out they are all {\em deformations} of cusps in hyperbolic manifolds.
This theme will be developed in a subsequent paper. 

{\red Choi (\cite{Choi} \cite{Choibook}) has studied certain kinds of ends of projective manifolds, and generalized cusps in this paper correspond to some lens type ends and quasi-joined ends in his work. } 

{\red
\section*{acknowledgements}
	We would like to thank the  referees for several suggestions that improved the paper. Work partially supported by U.S. National Science Foundation grants DMS 1107452, 1107263, 1107367 {\em RNMS: GEometric structures And Representation varieties} (the GEAR Network). Ballas was partially supported by NSF grant DMS 1709097. Cooper was partially supported by NSF grants DMS 1065939, 1207068 and  1045292, and thanks Technische Universit\"{a}t Berlin for hospitality during completion of this work. Leitner was partially supported by the ISF-UGC joint research program framework grant No. 1469/14 and No. 577/15.  
 }

\section{The Geometry of $\ppsi$-Cusps}\label{cuspgeom}

{ We recall some definitions, see \cite{Marquis3} for more background. A subset $\Omega\subset\RP^n$
is {\em properly convex} if the intersection with every projective line is connected, and omits
at least 2 points. The {\em boundary} is used in the sense of manifolds: 
$\bdy\Omega=\Omega\setminus\interior(\Omega)\subset\Omega$ and is usually distinct from
the {\em frontier} which is $\Fr(\Omega):=\bdy(\cl\Omega)=\cl\Omega\setminus\interior(\Omega)$. A properly convex domain  
{\em has strictly convex boundary} if  $\bdy\Omega$ contains no line segment. 
An {\em affine patch}
is the complement of a projective hyperplane.
If there is a unique supporting hyperplane to $\Omega$ at $p\in\Fr(\Omega)$
then $p$ is a \emph{$C^1$ point}.  

A {\em geometry} is a pair $(X,G)$ where $G$ is a subgroup of the group of homeomorphisms of $X$ onto itself. 
In this section we describe a family of geometries parameterized by points in the {\em (closed dual) Weyl chamber} 
\begin{equation}\label{eq:weyl}
A=\{\psi\in\Hom(\RR^n,\RR): \psicoef_i:=\psi(e_i) \quad \psicoef_1\ge\psicoef_2\ge\cdots \ge\psicoef_{n}\ge 0\}
\end{equation}
For each $\ppsi\in A$, there is a closed convex subset $\Omega(\ppsi)\subset\RR^n$ and a Lie subgroup
$G(\ppsi)$ of $\Aff(\RR^n)$,  described by Corollary \ref{Gppsistructure}, that preserves $\Omega(\ppsi)$ and acts transitively on $\partial\Omega$. The pair $(\Omega(\ppsi),G(\ppsi))$ is
called {\em $\ppsi$-geometry}. It is isomorphic to a subgeometry of Euclidean geometry (Corollary \ref{lambdageomEuclidean}).

Given $\psi\in A$
the {\em type} { is   $\type=\type_{\ppsi}=|\{i:\psi(e_i)>0\}|$ and}
\begin{equation}\label{eq:Vppsi}V=V_\ppsi=\RRP^\type\times\RR^{n-\type}\end{equation}
 
{ \begin{definition}\label{psidomain} The {\em $\ppsi$-horofunction} $h_\ppsi:V_{\ppsi}\to\RR$ is defined by
\begin{equation}\label{eq:horo} h_{\ppsi}(x_1,\cdots,x_n)=\left\{
\begin{array}{lcl}  -x_{\type+1}-\sum_{i=1}^{\type}\psicoefi_i\log x_i+\frac{1}{2}\sum_{i=\type+2}^{n}x_i^2  & \textrm{if} & \type<n\\
-\left(\sum_{i=1}^n\psicoefi_i\right)^{-1}\sum_{i=1}^{n}\psicoefi_i\log x_i
& \textrm{if} & \type=n
\end{array}\right.
\end{equation}
Also  $\Omega(\ppsi)=h_\ppsi^{-1}((-\infty,0])\subset\RR^n$  is called a \emph{$\ppsi$-domain}
and ${\mathcal H}_t=h^{-1}_\ppsi(t)$
is called a  \emph{horosphere}.
\end{definition}}

\begin{center}
\begin{figure}[h]
\includegraphics[scale=1]{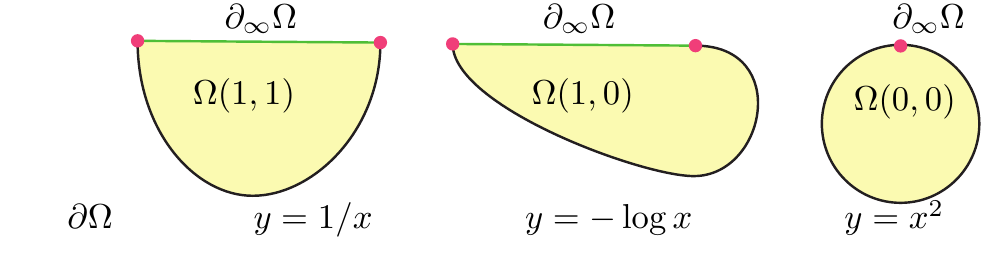}
 \caption{A projective view of some 2 dimensional {$\ppsi$-}domains} \label{dpic}
\end{figure}
\end{center}



\begin{proposition}\label{Omegaconvex}  { $\Omega(\ppsi)$ is a closed, { unbounded, convex} subset of $\RR^n$,
and a properly convex subset of $\RPn$.  The horospheres $\Hcal_t$ are smooth, 
strictly-convex, hypersurfaces  that foliate $V_{\ppsi}$,
  and
$\bdy\Omega(\ppsi)=\Hcal_0$.}\end{proposition}
\begin{proof} {   Since $h$ is a smooth submersion, 
 $\Hcal_t$ are 
  smooth hypersurfaces that foliate $V_{\ppsi}$
  and $\Omega=\Omega(\ppsi)$ is a closed submanifold of $\RR^n$ with boundary $\Hcal_0$.
   Moreover $\Omega$
 is unbounded because $h_{\psi}$ is a decreasing function of $x_s$
where $s=\max(\type,1)$.
  The second derivative of $x^2$, and of $-\log(x)$, are both positive on $V_{\ppsi}$,
so the second derivative $D^2h_{\ppsi}$ is positive semi-definite on $V_{\ppsi}$. For $\type<n$ it  has nullity 1, given by the $x_{\type+1}$ direction. When $\type=n$ it is positive definite. The tangent space
to $\Hcal_t$ is $T_*\Hcal_t=\ker D h_{\ppsi}$ which does not contain the $x_{\type+1}$ direction. Thus $D^2h_{\ppsi}$
restricted to $T_*\Hcal_t$ is positive definite hence
$\Hcal_t$
is strictly convex.
}

Suppose $\ell$ is a line segment with endpoints $a,b\in \Omega$. Set  $f=h|\ell$ then $f''\geq 0$
so $f$ attains its maximum at an endpoint. Thus $f\leq \max(f(a),f(b))$ and $f(a),f(b)\leq 0$ since $a,b\in\Omega$.
 Thus $f\leq0$  so $\ell\subset\Omega$ and $\Omega$ is convex. 


 Suppose $\ell$  is a complete affine line contained in
$\Omega$. Then $\ell$ is contained in $\RRP^\type\times\RR^{n-\type}$, so $x_i$ is constant along $\ell$
for $i\le \type$. Thus $\type<n$ and $h_{\ppsi}|\ell=C_1-t+C_2t^2,$ where $t$ is an affine coordinate on $\ell$. But $\ell\subset\Omega$ implies this function is nowhere positive, a contradiction. Hence
$\Omega$ contains no complete affine line, and is thus properly convex .
%
 %
\end{proof}
 {\begin{remarks}\label{positivelemma}  (a)  If $\forall i\ |\ppsi_{i}|\ge|\ppsi_{i+1}|$ then $h_{\ppsi}$ is convex iff either  $\forall i\ \ppsi_i\ge0$ or   $\forall i\ \ppsi_i<0$.\\
(b) It follows from Lemma \ref{uniquesupphyperplane} and the discussion in Section 3 of \cite{CLT1}
that the $\mathcal H_t$ are  horospheres in the sense of Busemann, and from
  (\ref{eq:horosphere}) that
 they are also  {\em algebraic horospheres} as defined in \cite{CLT1}. \end{remarks} }
\begin{definition}\label{psicuspdef} The {\em $\ppsi$ cusp Lie group} is  the { sub}group,
$G=G(\ppsi)\subset\PGL(n+1,\RR)$
that preserves each horosphere. A {\em $\ppsi$-cusp} is  $C=\Omega(\ppsi)/\Gamma$ where
 $\Gamma\subset G$ is a torsion-free lattice.
\end{definition}
The condition that $G(\ppsi)$ preserves each horosphere is equivalent to preserving the horofunction, {
thus $G(\ppsi)\subset\PGL(\Omega(\ppsi))$.
 It follows { from Propostion \ref{Omegaconvex}} that a $\ppsi$-cusp
is  a properly-convex manifold.}
{ The torsion free hypothesis on $\Gamma$ is strictly a matter of convenience.} If $\Gamma$ is a lattice in $G(\ppsi)$ that contains torsion then {$C$} is an orbifold. 

\subsection{The Radial Flow}
 The {\em unipotent rank} is 
$\ur=\max(n-1-\type,0)$ and the {\em rank} $\rank$ is defined by 
$ \rank+\ur=n-1$.  Then $\rank=\min(\type,n-1)$. { A more conceptual interpretation of $\rank$ and $\ur$ is given by Equation (\ref{rprime:eq}).}  It is convenient to use coordinates on  $V_{\ppsi}$ given by
\begin{equation}
\label{eq:coords}
(x,z,y)\in V_{\ppsi}=
\left\{\begin{array}{lcl} \RRP^{\rank}\times\RR\times \RR^{\ur} & \textrm{if} & \type<n\\
\RRP^{\rank}\times\RR_+\times \RR^{\ur} & \textrm{if} & \type=n
\end{array}\right.
\end{equation}
When {$\type=0$ the  $x$-coordinate is empty; and when $\type\ge n-1$ then $\ur=0$ so the $y$-coordinate} is empty.
The $z$-coordinate 
is called the {\em vertical direction}. This terminology is motivated by regarding
 the horospheres as graphs of functions, { see Equation \eqref{eq:Logpsif}.} 
 { The $y$-coordinate is called the {\em parabolic direction} and 
 the $x$-coordinate is called the {\em hyperbolic direction}, see Equation (\ref{eq:Tmatrix}).} 

\begin{definition}\label{basepointdef} The {\em basepoint} of $\Omega(\ppsi)$  is $b=b_{\ppsi}=e_1+\cdots+e_\type\in\RR^n$. 
 \end{definition}

Thus for $\type=0$ the basepoint is $b=0\in\RR^n$. The basepoint satisfies $h_{\ppsi}(b)=0$ so
$b\in\bdy\Omega$. When $\type<n$ then 
$b=(x_0,z_0,y_0)$
 where $x_0=(1,\cdots,1)\in \RRP^{\rank}$ and the remaining coordinates are $0$. When $\type=n$ then
 $b=(1,\cdots,1)$. In projective coordinates the basepoint is $[b_{\ppsi}+e_{n+1}]\in\RP^n$. 
Define $$U=U_{\ppsi}=\RRP^{\rank}\times \RR^{\ur}$$
{\em Radial projection} is the map $\pi=\pi_{\ppsi}:V_{\ppsi}\to U_{\ppsi}$ given by  
\begin{equation}
\label{eq:radialproj}
\pi(x,z,y)=\left\{\begin{array}{lcl} (x,y) & \textrm{if} & \type<n\\
 (x/z,y/z)& \textrm{if} &\type=n
 \end{array}\right.
 \end{equation}
\begin{definition} The {\em radial flow} { $\flow=\flow^{\psi}\subset\PGL(n+1,\RR)$ is
the $1$-parameter subgroup that acts on $V(\ppsi)$ by}
\begin{equation}
\label{eq:radialflow}
\flow_t(x,z,y)=\Phi((x,z,y),t)=\left\{\begin{array}{lcl}(x,z-t,y) & \textrm{if} &\type<n\\
e^{-t}(x,z,y) & \textrm{if} & \type=n
\end{array}\right.
\end{equation}\end{definition}
In the first case the radial flow is called \emph{parabolic} and in the second case it is \emph{hyperbolic}. 
This terminology agrees with that of \cite{CLT2}. The orbit of a point
is called a {\em flowline}.  Each flowline maps to one point under radial projection. When $\type<n$  flowlines are vertical lines, and when $\type=n$ they are open rays that limit on $0\in\RR^n$.  

The reason for the name {\em radial flow} 
is that this group acts on $\RP^n$
and there is a point $\alpha\in\RP^n$ called the {\em center of the radial flow} with the property that,
 if a point $\beta\in \RP^n$ is not fixed by the flow, then
the orbit  of $\beta$ is contained in the projective line containing $\alpha$ and $\beta$. Moreover $\Phi_t(\beta)\to\alpha$ as $t\to\infty$. { The center is
 \begin{equation}\label{eq:center}\alpha=[e_{\type+1}]\end{equation} If $\type<n$ then 
 $\alpha\in\RP_{\infty}^{n-1}$
  corresponds to  the $z$-axis, and  $\alpha=0\in\RR^n$ when $\type=n$.}

 Observe that the radial flow has the following equivariance property:
\begin{equation}\label{eq:flowequivariance}
h_{\ppsi}(\Phi_t(x))=h_{\ppsi}(x)+t 
\end{equation}
This equation would need to be modified without the first  factor in  the definition of $h_{\ppsi}$ (see \eqref{eq:horo}) when $\type=n$.
 It follows that the radial flow permutes the level sets of the horofunction  and hence permutes
the horospheres and
\begin{equation}\label{eq:horosphere}{\mathcal H}_{t}=\Phi_{-t}(\partial\Omega)\end{equation}

\begin{definition}\label{transversefoliations} A {\em product structure} on a manifold $M$ is a pair of transverse foliations
on $M$ determined by a diffeomorphism $P\times Q\to M$.  There is a 
diffeomorphism $f:\bdy\Omega(\ppsi)\times[0,\infty)\to \Omega(\ppsi)$  given by $f(x,t)=\Phi_{-t}(x)$.
This defines a product structure  on $\Omega(\ppsi)$,  
with a   foliation by  {\em horospheres}, and
a transverse foliation by {\em (half-)flowlines}. 
\end{definition}

If $C=\Omega(\ppsi)/\Gamma$ is a $\ppsi$-cusp, then $\Gamma$
preserves this product structure, so it covers a product structure on $C$. { The image in $C$ of a horosphere is called a {\em horomanifold}.} 
The set $\Omega$ is {\em backwards invariant} which means that 
$\Phi_t(\Omega)\subset\Omega$ for all $t\le 0$ and
$\Omega$ is the backwards orbit of $\partial\Omega$
$$\Omega=\bigcup_{t\le 0}\Phi_t(\partial\Omega)$$
 For $\type< n$ it is convenient to introduce $\ppsi^\type:\RR^\type\rightarrow\RR$ given by
\begin{equation}\label{ppsir}
\ppsi^\type(x)=\ppsi(x,0,\cdots,0)
\end{equation} 
Define  $\Log:\RR\to\RR$ by
\begin{equation}
\label{eq:Log}
\Log(x)=\left\{\begin{array}{lcr}
0 & \textrm{if} & x\le 0\\
\log(x) & \textrm{if} & x>0
\end{array}
\right.
\end{equation} and extend this to a map $\Log:\RR^r\to\RR^r$ by applying $\Log$ componentwise.
Then define $f=f_{\ppsi}:U\to\RR$ by 
\begin{equation}
\label{eq:Logpsif}
\begin{array}{rcl}
 f_{\psi}(x,y) & = &
\left\{
\begin{array}{lcl} -\ppsi^\type\circ\Log(x)+\|y\|^2/2& \textrm{if} & \type<n\\
\prod_{i=1}^{n-1}x_i^{-\psicoef_i/\psicoef_n} & \textrm{if} & \type=n
\end{array}\right.
\end{array}
\end{equation}
The map $F=F_{\ppsi}:U\to\bdy\Omega$ given  by 
\begin{equation}\label{eq:graph}F(x,y)=(x,f(x,y),y)\end{equation} 
is the inverse of the 
restriction of {\em vertical} projection $\pi|:\partial\Omega\to U$, so
 $\partial\Omega$
  is the graph $z=f(x,y)$ of $f$ and $\Omega=\{(x,z,y): z\ge f(x,y)\}$  is the supergraph of $f$.
{ From Definition \ref{psidomain} when} $\type<n$ the horofunction is  expressed more compactly as
\begin{equation}\label{eq:compacth}
h_{\ppsi}(x,z,y)=
 -z+f_{\ppsi}(x,y)
 \end{equation}
{  but for $\type=n$  this does not work.}

\subsection{The Ideal Boundary $\bdy_{\infty}\Omega$}
In what follows  $\ppsi$ { is omitted} from the notation. We describe the closure $\overline\Omega$ in $\RP^n$. 
Identify affine space $\RR^n$ with an affine patch in projective space $\RP^n$ by identifying $(x,z,y)$ in $\RR^n$ with  $[x:z:y:1]$ in $\RP^n$. Then
\begin{equation}
\label{eq:omegadef}
\Omega=\{[x:z:y:1]\ |\  z\ge f(x,y),\ \ x\in\RRP^{\rank}\}\subset\RP^n
\end{equation}
Observe that $\overline\Omega\cap\RR^n=\Omega$.
 The {\em points at infinity} are $\RP^{n-1}_{\infty}=\RP^n\setminus\RR^n$ and
\begin{equation}
\label{eq:bdyinfinity}
\overline\Omega=\Omega\sqcup\bdy_{\infty}\Omega\qquad{\rm with}\quad \bdy_{\infty}\Omega:=\overline\Omega\setminus\Omega\subset\RP^{n-1}_{\infty}
\end{equation}
The set $\bdy_{\infty}\Omega$ is called the {\em ideal boundary} or the {\em boundary at infinity} of $\Omega$.  See \cite{DKG} Definition 1.17. The {\em non-ideal boundary} or just {\em boundary} of $\Omega$ is $\partial\Omega=\RR^n\cap\bdy\overline\Omega$. Thus
$$\partial\overline\Omega=\bdy\Omega\sqcup\bdy_{\infty}\Omega.$$
\begin{lemma}\label{idealboundary}$\bdy_{\infty}\Omega(\ppsi)$ is the simplex of dimension $\rank$
$$\bdy_{\infty}\Omega(\ppsi)=\Delta^{\rank}:=\{[x_1:\cdots:x_{\rank+1}:{ 0}:\cdots:0]\  |\  x_i\ge 0\}$$
\end{lemma}
\begin{proof} From Equation (\ref{eq:omegadef}) $\bdy_{\infty}\Omega$ consists of all the points that are the limit of a sequence of 
points $[x:z:y:1]$ with $\|(x,z,y)\|\to\infty$ for which $z\ge f(x,y)$.

First assume $\type<n$, and so $\type=\rank$. We claim that $y/\|(x,z,y)\|\to 0$ along the sequence. If eventually $\psi^{\type}\Log(x)<\|y\|^2/4$   then by Equation (\ref{eq:Logpsif}) it follows that $z>\|y\|^2/4$ and, since $y/\|y\|^2\to 0$ as $\|y\|\to\infty$, it follows that $y/z\to 0$.
Otherwise we may take a subsequence so $\psi^{\type}\Log(x)\ge \|y\|^2/4\to+\infty$. Since $\psi_j>0$ for all $j\le \rank$,
 this means  for some $i\le \rank$ the coordinate 
$x_i$ of $x$ is positive
and 
larger than some fixed multiple of $\exp\|y\|$, hence $y/x_i\to0$. This proves the claim.
Hence
$\bdy_{\infty}\Omega\subset\Delta^{\rank}$.

From (\ref{eq:Logpsif}) we see that  $f(e^tx,0)<0$  for large $t$. Then by Equation (\ref{eq:omegadef}) 
$$\bdy_{\infty}\Omega\supset \{\lim_{t\to\infty}[e^t x:e^t z:0:1]\ |\  z \ge 0,\ \ x\in\RRP^r\}=\Delta^{\rank}$$
which proves the result for $\type<n$. 

When $\type=n$ then  $\rank=n-1$ and $\partial_\infty\Omega\subset \partial_\infty V_\ppsi=\Delta^{n-1}$. On the other hand if $v\in \interior\Delta^{n-1}$ then $v=\lim_{t\to\infty}[tx:tz:1]$, where $x\in \RRP^{n-1}$ and $z\in \RRP$. From the definition of $f(x)$ (see Equation \eqref{eq:Logpsif}), it is easy to check, { when $t$ is large}, that $tz>f(tx)$, hence $(tx,tz)\in\Omega$, and so $\interior \Delta^{n-1}\subset \partial_\infty\Omega$. Since $\partial_\infty\Omega$ is closed it follows that $\Delta^{n-1}=\partial_\infty\Omega$.
\end{proof}

\begin{lemma}\label{uniquesupphyperplane} { Every point in} the relative interior of $\bdy_{\infty}\Omega$ 
is a $C^1$-point, {  and
\begin{itemize}
\item[(a)] $\RP^{n-1}_{\infty}$ is the unique  hyperplane in $ \RPn$ that contains $\bdy_{\infty}\Omega$ and is disjoint from
$\Omega$. 
\item[(b)]  $\AA(\Omega):=\RPn\setminus\RP^{n-1}_{\infty}$ is the unique affine patch that contains $\Omega$
as a closed subset.
\item[(c)] $\PGL(\Omega)\subset\Aff(\RR^n)$.
\end{itemize}}
\end{lemma}
\begin{proof} { Clearly (a) implies (b) and (c).  For (a) the result follows from the following  picture
that we will establish. Near a point $q\in\interior(\bdy_{\infty}\Omega)$ the 
frontier  $\Fr\Omega$ looks like a (flat) open set in $\Delta^{\rank}$
product a hypersurface in $\RR^{n-\rank}$ that is $C^1$ close to an ellipsoid, and is thus $C^1$.

Now (a) is clear for $\rank=n-1$. It is also clear in the case
 $\type=0$ since $\cl\Omega$ is a round ball, and $q=\bdy_{\infty}\Omega$ is a single $C^1$ point.
 Thus we may suppose that $\ur\ge 1$ and $\rank=\type>0$. 
We use coordinates $(a,y,w)\in\RR^{\rank+1}\oplus\RR^{\ur}\oplus\RR\equiv\RR^{n+1}$ where 
$a=(x,z)\in\RR^{\rank}\oplus\RR$ in the coordinates above.
Thus the affine patch used above is $[a:y:1]$, and $\RP^{n-1}_{\infty}$ is
$[a:y:0]$.
Given $q\in\interior(\bdy_{\infty}\Omega)$ then $q=[a:0:0]$ with
 $a=(a_1,\cdots,a_{\rank+1})\in\RR_+^{\rank+1}$ by  Lemma \ref{idealboundary}. Let
 $$P=\PP(\RR\cdot a\oplus\R^{\ur}\oplus \RR)\subset\RPn$$ Then $P\cong\RP^{\ur+1}$ and $P\cap\bdy_{\infty}\Omega=q$.
Moreover $W:=P\cap\bdy\Omega\subset\RR^n$ is  the subset of $\bdy\Omega$ where  
the $(x,z)$-coordinate is $z\cdot a$ for some $z$.  We may scale $a$ so that $a_{\rank+1}=1$
then $z\cdot a=(za_1,\cdots,za_{\rank},z)$.
By Equation (\ref{psidomain}) $\bdy\Omega\cap\left(\RR\cdot a\oplus\RR^{\ur}\right)$ is given by
\begin{equation} { 0}=h(za_1,\cdots,za_{\rank},z,y)=-z-\sum_{i=1}^{\rank}\ppsi_i\log(za_i)+\|y\|^2/2
\end{equation}
This may be rewritten as
\begin{equation} z+\alpha\log z +\beta =\|y\|^2/2\qquad \alpha
=\sum\ppsi_i,\quad\beta = \sum\ppsi_i\log a_i\end{equation}
Near $q$ then $z$ is large and $W$ is $C^0$-close to the ellipsoid $z=\|y\|^2/2$ in $\RP^{\ur+1}$. We now show it is $C^1$-close
by changing to a different affine patch using
 $[y:z:1]=[yz^{-1}:1:z^{-1}]=[u:1:v]$. Then  $q$ is the point $(u,v)=(0,0)\in\RR^{\ur+1}$ and $u\in\RR^{\ur}$ and $v\in\RR$ satisfy
\begin{equation} v^{-1}+\alpha\log v^{-1}+\beta=\|v^{-1} u\|^2/2
\end{equation}
Which can be expressed as $f(v)=\|u\|^2/2$ where $f(v)=v-\alpha v^2\log v +\beta v^2$ for $v\ne0$. There is
a $C^1$ extension of $f$ given by $f(0)=0$, then $f'(0)=1$, so by the inverse function theorem
near $0\in\RR^{\ur}\oplus\RR$ we have $v=f^{-1}(\|u\|^2/2)$ is $C^1$ close to $v=\|u\|^2/2$ at $q$. {\blue In
particular $q$ is a $C^1$ point of $\cl W$.} It
is interesting that $f''(v)\to\infty$ as $v\to0$.

Let $H\cong\RP^{n-1}$ be a supporting hyperplane to $\Fr\Omega$ at $q$,
 then since $q$ is a $C^1$ point of $\cl(W)$ in $P$ it follows that $H$ contains $Q:=P\cap\RP_{\infty}^{n-1}$. Furthermore, $\dim Q=\dim P-1=\ur$.
Since $q\in\interior(\bdy_{\infty}\Omega)$ and $H$ is a supporting hyperplane, it follows that $H$  contains $\bdy_{\infty}\Omega$. By Lemma \ref{idealboundary}, $\dim \partial_\infty \Omega=\rank$ and hence $\dim(\bdy_{\infty}\Omega)+\dim Q=\rank +\ur=n-1=\dim H.$
Since $\bdy_{\infty}\Omega$ and $Q$ are transverse in $H$ it follows that $H$ is the unique hyperplane that 
contains $\bdy_{\infty}\Omega\cup Q$.}
\end{proof}

{ The next result implies that a generalized cusp has a natural affine structure that is a stiffening of the
projective structure. 

\begin{proposition}\label{characterizecenterRF}\label{invariantproduct}
 Let $\Omega=\Omega(\ppsi),\  \Phi=\Phi^{\psi},\ h=h_{\psi}$ and let $\blue\alpha$ be the center of $\Phi$. Let $\Omega',\Phi',h'$ and ${\blue\alpha}'$ be the corresponding objects for $\psi'$. Suppose that $P\in\PGL(n +1,\RR)$ and $P(\Omega)=\Omega'$
 then
 \begin{itemize}
 \item[(a)] $\type=\type'$
 \item[(b)] $P({\blue\alpha})={\blue\alpha}'$
\item[(c)] $P\in\Aff(\RR^n)$
\item[(d)] $\Phi'=P\cdot \Phi\cdot P^{-1}$
\item[(e)] $\exists\ {\blue c}>0\ \ \forall \ t\ \ \ \flow'_{{\blue c} t}=P\cdot\flow_t\cdot P^{-1}$
\item[(f)]  $h'\circ P={\blue c}\cdot h$
\item[(g)] $P$ sends the product structure of $\Omega$ to that of $\Omega'$
\end{itemize}
 \end{proposition}
\begin{proof} Clearly $P(\bdy_{\infty}\Omega)=\bdy_{\infty}\Omega'$. 
By Lemma \ref{idealboundary} $\bdy_{\infty}\Omega$ is a simplex
of dimension $\rank=\min(\type,n-1)$ and it follows that if $\type\le n-2$ then $\type=\type'$. It remains
to distinguish  $\type=n-1$ from $\type=n$ in a projectively invariant way.

{\bf Claim} If $\type\ge n-1$ then there is a unique minimal closed $n$-simplex $\Delta\subset\RPn$, such that $\Omega\subset\Delta$ and $\bdy_{\infty}\Omega$ is a face of $\Delta$,
 if and only if $\type=n$.
 
 In the case $\type=n$ then $\Delta$ is
  the closure in $\RPn$ of $\RR_+^n\subset\RR^n$.
Minimality and uniqueness follows from the fact that $\bdy\Omega$ is {\blue
asymptotic in $\RR_+^n$ to  
 $\bdy\Delta$ near $\RP^{n-1}_{\infty}$. Consider a ray $\gamma:[0,\infty)\to\RR_+^n$  
 then $h(\gamma(t))\to-\infty$ as $t\to\infty$. { Thus $\gamma(t)\in\Omega$ for $t$ large.} If $\Delta'\ne\Delta$, and $\Delta'$ contains $\Omega$, and 
 $\Delta'=p*\bdy_{\infty}\Omega$, then there is such a ray in $\bdy\Delta'$, unless $\Delta'\supset\Delta$.
 }  Thus any simplex
  that contains $\Omega$ also contains $\Delta$. 
  
  In the case $\type=n-1$
  the analysis below shows that there is a projective plane $H\cong\RP^2$
such that $\Omega\cap H$ looks like $\Omega(1,0)$ in Figure $1$. This implies no such $\Delta$ exists, which 
proves the claim and (a).
  
{\blue By Equation (\ref{eq:center})}, ${\blue\alpha}=[e_{\type+1}]$. When $\type=n$ then $\type'=n$ and uniqueness of $\Delta$
   implies $P(\Delta)=\Delta'$. Observe that
  ${\blue c}=[e_{n+1}]$ is the unique
   vertex of $\Delta$ that is not in $\bdy_{\infty}\Omega$. Thus $P({\blue\alpha})={\blue\alpha}'$ in this case.
 When $\type=0$ then
${\blue\alpha}=[e_1]=\bdy_{\infty}\Omega$ and ${\blue\alpha}'=\bdy_{\infty}\Omega'$ so (b) follows in this case.
 
 For (b) this leaves the case $0<\type<n$ then ${\blue\alpha}=[e_{\type+1}]\in\RP^{n-1}_{\infty}$. By Lemma \ref{idealboundary} the
 vertices of $\bdy_{\infty}\Omega$ are $[e_s]$ with $1\le s\le\type+1$.
Given $s\le \type$  let $H\cong\RP^2$ be any
 projective plane that  contains the vertices 
$[e_{\type+1}]$ and $[e_s]$ of $\Delta^\rank$, and also some point $[(v_1,\cdots,v_n,1)]\in\bdy\Omega$.
The intersection of $H$ with the affine patch $\RR^n$  is the affine subspace
$V=\langle e_{s}, e_{\type+1}\rangle + (v_1,\cdots,v_n)$.
{ Using Definition \ref{psidomain}, the} restriction of $h=h_{\ppsi}$ to $V$ is
$$f(X,Y):=h(Xe_{s}+Y e_{\type+1}+v) =
 -Y -\ppsi_s\log (X+v_s) +C
$$ { where $C$ is a constant independent of $X$ and $Y$ that depends on $v$.}
The curve $V\cap \partial \Omega$ is given by $f(X,Y)=0$. The  affine 
 change of coordinates   $(X,Y)=(x-v_s,\ppsi_s y+C)$  maps this curve to $y=-\log x$. It follows that, on the curve $H\cap\bdy\overline\Omega$, the point
  $[e_s]$ is $C^1$ and $[e_{\type+1}]$ is not $C^1$ (see the middle domain in Figure \ref{dpic}). 
  Thus the center ${\blue\alpha}=[e_{\type+1}]$ is a vertex of $\bdy_{\infty}\Omega$ that
   is distinguished (in a projectively invariant way)
  from every other vertex $[e_s]$ of $\bdy_{\infty}\Omega$. This completes the proof of (b).
  
  By Lemma \ref{uniquesupphyperplane}(c)  $P$ 
preserves $\RP^{n-1}_{\infty}$ proving (c).  The radial flow is characterized as the one-parameter subgroup  $\flow\subset \PGL$ that
fixes every point in the stationary hyperplane, preserves every line containing the center, and no non-trivial element fixes any other point.
This and (b) implies (d). Every automorphism  of $\RR$ is multiplication by some
 ${\blue c}\ne0$. Since $\Omega$ is backward invariant, ${\blue c}>0$  which proves  (e). By Equation (\ref{eq:flowequivariance}) $h'\circ P={\blue c}\cdot h$ which proves (f).  The level sets of $h$ gives the foliation by horospheres, so $P$ preserves this foliation. Similarly $P$ preserves the $\Phi$-orbits of points, which are the flowlines,
  giving (g). \end{proof}}

\subsection{{ The Structure of $G(\ppsi)$}}
  { Theorem \ref{Gppsistructure} gives a  decomposition of the group 
  $G(\ppsi)\cong T(\ppsi)\rtimes O(\ppsi)$ corresponding
  to the decomposition 
  $\Isom(\R^n)\cong\R^n\rtimes O(n)$ into translation and orthogonal subgroups. We begin by describing the translation subgroup $T(\ppsi)$.}

Recall the standard identification of the affine group $\Aff(\RR^n)$ with the subgroup 
$$\left\{\begin{pmatrix} A & v\\ 0 & 1\end{pmatrix}\ :\  A\in \GL(n,\RR),\ v\in\RR^n\right\}\subset \GL(n+1,\RR)$$
The affine action on $\R^n$ is realized by the embedding $\RR^n\to\RR^{n+1}$ given by $a\mapsto(a,1)$.

 Our next task is to define a subgroup of $G(\ppsi)$,
called the {\em translation subgroup} $T(\ppsi)\cong \RR^{n-1}$, that 
 acts simply transitively on
$\bdy\Omega(\ppsi)$. We first define the {\em enlarged translation group} $T_{\type}\cong\RR^n$ that  acts simply transitively on $V_{\psi}=\RR_+^\type\times\RR^{n-\type}$. Then  $T(\ppsi)=\ker\ppsi_*$ for a certain homomorphism 
$\ppsi_*:T_{\type}\to\RR$ derived from $\ppsi$. 
The enlarged translation group is  the direct sum of the translation group and the radial flow: $T_{\type}=T(\ppsi)\oplus\flow^{\ppsi}$.

The {\em enlarged translation group}  $\eT_{\type}$  has Lie algebra $\eTalg_{\type}$ that is the image
of the map $\Psi_{\type}:\RR^n\to\mathfrak{gl}(n+1,\RR)$ given by
\begin{equation}\label{enlargedtranslationgrp}
\Psi_{\type}(X,Z,Y):= \begin{pmatrix}
\Diag(X) & 0 \\
 0& \begin{pmatrix}
0 & Y^t & Z\\
0 &0 & Y\\
0 & 0 & 0
\end{pmatrix}
\end{pmatrix}
\end{equation}
Here $X\in\RR^\rank$ and $Z\in\RR$ and $Y\in\RR^{\ur}$, except
 when $\ur=0$ there is no $Y$, and when $\type=n$ there is no $Z$ and the bottom right block is $(0)$.
  It is easy to check that all Lie 
brackets in $\mathfrak t_{\type}$ are $0$ and so $\mathfrak t_{\type}$ is an abelian Lie subalgebra, and $T_{\type}\cong\RR^n$ as a  Lie group.
Define $m_{\type}(X,Z,Y)=\exp\Psi_{\type}(X,Z,Y)$ then $T_{\type}$ consists of all matrices
\begin{equation}\label{eq:mmatrix}
m_{\type}(X,Z,Y)=\begin{pmatrix}
\exp\Diag(X) & 0 \\
 0& \begin{pmatrix}
1 & Y^t & Z+\|Y\|^2/2\\
0 &I_{\ur} & Y\\
0 & 0 & 1
\end{pmatrix}
\end{pmatrix}
\end{equation}

\begin{definition}\label{Tdef} 
The {\em translation group}  
$T:=T(\ppsi)$ is the kernel of the homomorphism $\ppsi_*:T_{\type}\to\RR$  defined for $\type<n$ by $\ppsi_*(m_{\type}(X,Z,Y))=\ppsi^\type(X)+Z$,
and for $\type=n$ by $\ppsi_*(m_{n}(X))=(\sum\ppsi_i)^{-1}\ppsi(X)$.
\end{definition}

For $\type<n$,   the translation group $T(\ppsi)$ consists of the matrices $m_{\type}(X,Y,Z)$ given by Equation (\ref{eq:mmatrix}) for which $Z=-\ppsi^\type(X)$.
It will occasionally be convenient to write the translation group as the image of a linear map, instead of as the kernel of a linear map. 
For $\type<n$ the translation group $T(\ppsi)$ is the image of $m^*_{\ppsi}:\RR^\rank\times\RR^{\ur}\to\GL(n+1,\RR)$ given by
\begin{equation}
\label{eq:Tmatrix}m^*_{\ppsi}(X,Y)=\begin{pmatrix}
 \exp\Diag(X) & 0 \\
 0& \begin{pmatrix}
1 & Y^t & \|Y\|^2/2-\ppsi^\type(X)\\
0 & I_{\ur} & Y\\
0 & 0 & 1
\end{pmatrix}
\end{pmatrix}
\end{equation} 
and for $\type=n$ the translation group $T(\ppsi)$ is the image of $m^*_{\ppsi}:\ker\psi\to\GL(n+1,\RR)$ by
\begin{equation}\label{eq:Tmatrix2}
	m^*_{\ppsi}(X)=\begin{pmatrix}
 \exp\Diag(X) & 0 \\
 0& 1
\end{pmatrix} 
\end{equation}
It is worth pointing out that with this formalism the case $\type=n-1$ means $\ur=0$ and gives
\begin{equation}\label{eq:Tmatrix1}
m^*_{\ppsi}(X,Y) =\begin{pmatrix}
 \exp\Diag(X) & 0 \\
 0& \begin{pmatrix}
1 &-\ppsi^\type(X)\\
0 & 1
\end{pmatrix}
\end{pmatrix}
\end{equation}

\begin{lemma}\label{Tpluslemma} $\eT_{\type}$ acts simply transitively on $V_{\ppsi}=\RR_+^\type\times\RR^{n-\type}$ and  
\begin{enumerate}
\item[(a)]  $h_\ppsi\circ m_\type{(X,Z,Y)}=h_\ppsi-\ppsi_\ast\circ m_\type{(X,Z,Y)}$
\item[(b)]  $T(\ppsi)$ is the the subgroup of $T_{\type}$ that preserves $h_{\ppsi}$  
\item[(c)]  $\eT(\ppsi)$ preserves the foliation of $V_\ppsi$ by horospheres
\item[(d)] $\eT(\ppsi)$ preserves the transverse foliation by flowlines.  
\end{enumerate}
\end{lemma}
\begin{proof} It is clear the action is simply transitive, and that (a) implies both (b) and (c), and that (d) holds.  We first prove (a) in the case $\type<n$. From Equation (\ref{eq:compacth})
$$-h_{\ppsi}(x,z,y)=
 z+\ppsi^\type(\Log x)-\|y\|^2/2$$
and
 $$\begin{array}{l}m_{\type}(X,Z,Y)(x,z,y)^t\\
 =(\exp(X_1)x_1,\cdots,\exp(X_\rank)x_\rank,z+Y\cdot y+Z+\|Y\|^2/2,Y_1+y_1,\cdots, Y_{\ur}+y_{\ur})^t
 \end{array}$$
so
$$\begin{array}{l}
-h_{\ppsi}(m_{\type}(X,Z,Y)(x,z,y)^t)\\
= z+Y\cdot y+Z+\|Y\|^2/2+\ppsi^\type\Log(\exp(X_1)x_1,\cdots,\exp(X_\rank)x_\rank)-\|Y+y\|^2/2\\
=z+Z + \psi^\type(X+\Log x) - \|y\|^2/2\\
=(Z + \ppsi^\type(X)) +\left( z+\ppsi^\type(\Log x) - \|y\|^2/2\right)\\
=\ppsi_*(m{_\type}(X,Z,Y))-h_{\ppsi}(x,z,y)\\
\end{array}$$
A similar but simpler argument applies when $\type=n$, by omitting the $Y$ and $Z$ coordinates.
\end{proof}

\begin{lemma}\label{Tlemma} $T(\ppsi)\subset G(\ppsi)$ and  $T(\ppsi)$ acts simply transitively on $\bdy\Omega(\ppsi)$.
\end{lemma}
\begin{proof} By Equations (\ref{eq:Tmatrix}) and (\ref{eq:mmatrix})   $T(\ppsi)$ is the subgroup of $\eT_{\type}$ given by 
$Z=-\psi^\type(X)$. It follows from (Lemma \ref{Tpluslemma})(b) that $T(\ppsi)$
is the subgroup of $\eT_{\type}$ that preserves the horofunction, hence $T(\ppsi)\subset G(\ppsi)$. Simple transitivity on $\partial\Omega(\ppsi)$
also follows from Lemma \ref{Tpluslemma}.
\end{proof}

{ The following is from \cite{CLT1}. If $\Omega$ is open and properly convex and $A\in\PGL(\Omega)$ the {\em displacement distance} of $A$ is
\begin{equation}\label{displacementdistance}
\delta(A)=\inf\{d_{\Omega}(x,Ax)|x\in\Omega\}
\end{equation}
 where $d_{\Omega}$ is the Hilbert metric on $\Omega$. Then $A$ is called {\em hyperbolic} if $\delta(A)>0$, and {\em elliptic} if $A$ fixes a point in $\Omega$,
otherwise it is 
 called {\em parabolic} if $A$ does not fix any point in $\Omega$ and $\delta(A)=0$. Moreover $\delta(A)=0$
if and only if all eigenvalues of $A$ have the same modulus.  
A parabolic $A\in\PGL(n+1,\RR)$ is called {\em standard} if it is conjugate into 
$PO(n,1)\cong\Isom(\HH^n)$. This is equivalent to
there are $\lambda,t\ne0$ such that 
$\lambda A$ is conjugate in $\GL(n+1,\RR)$ into 
\begin{equation}\label{stdparabolic}\begin{pmatrix} 1 & t & t^2/2\\ 0 & 1 & t\\ 0 & 1 & 1
\end{pmatrix}
\oplus O(n-2)\end{equation}
Standard parabolics have a Jordan block of size $3$. 
It follows from { Equation (\ref{eq:Tmatrix}) } that:

\begin{lemma}\label{paraboliclemma} The  parabolic subgroup  $P(\ppsi)\subset T(\ppsi)$ 
 consists of all unipotent elements of $T(\ppsi)$. Moreover
 $P(\ppsi) = \{ m^*_{\ppsi}(0,Y):Y\in\RR^{\ur}\}$ and non-trivial elements are standard standard parabolics.
\end{lemma}}

 Let $T_1\subset T_2\subset T$ where $T_1$ is
 the subgroup of diagonalizable elements,
and $T_2$ is the subgroup of elements for which  every Jordan block has size at most $2$}. This description
is invariant under conjugacy, and
\begin{equation}\label{T1T2}
\begin{array}{rcccl}
T_1  & = &{ T_1(\ppsi)}& = &\{ m^*_{\ppsi}(X,0):X\in\ker\ppsi^{\rank}\}\\
T_2  & = &{ T_2(\ppsi)}& = &\{ m^*_{\ppsi}(X,0):X\in\RR^{\rank}\}
\end{array}
\end{equation}
Then $T(\ppsi)=P(\ppsi)\oplus T_2$, and { $\dim T_2= 1+\dim T_1$  if $0<\type<n$}. 
{ Non-trivial} elements of $T_2$ are  hyperbolic. { 
 A {\em weight} is a homomorphism $\lambda:T(\ppsi)\to \RR^{\times}$
such that $\det(A-\lambda I)=0$ for all $A\in T(\ppsi)$. Let $\Wcal$ be the set of such weights.
Here are conceptual descriptions of $\ur$, $\rank$ and $\type$:
\begin{equation}\label{rprime:eq}
\ur=\dim P(\ppsi)\qquad \rank=\dim T_2(\ppsi)\qquad\type=|\Wcal|-1\end{equation} 
Thus $\rank$ is the dimension of the subgroup of hyperbolics in the translation group, 
$\ur$ the dimension of the unipotent (parabolic) subgroup, $\dim T(\ppsi)=\ur+\rank=n-1$.}


\begin{definition}\label{Oppssidef}
$O(\ppsi)$ is the subgroup of $G(\ppsi)$ that fixes the basepoint $b_{\ppsi}$.
 \end{definition}
When $\psi=0$ (the case of a  cusp in $\HH^n$) then $O(\ppsi)\cong O(n-1)$ is the subgroup of 
$O(n)\subset\Aff(\RR^n)$ that fixes $e_1$. 
At the other extreme, when $\type=n$ and all the coordinates of $\psi$ are distinct, then $O(\ppsi)$ is trivial. The general case is:
\begin{proposition}\label{olambda} Suppose $\ppsi\in A$ has type $\type=\type(\psi)$.
Let $e_1,\cdots,e_{n+1}$ be the standard basis of $\RR^{n+1}$ and 
 $S(\ppsi)\subset \GL(\type,\RR)$  be the subgroup that permutes $\{e_1,\cdots,e_\type\}$ 
  and preserves the vector $\sum_{i=1}^\type\psicoef_ie_i.$
 Then $O(\ppsi)$ is equal to the subgroup $ O'(\ppsi) \subset \Aff(\RR^n)\subset\GL(n+1,\RR)$ given by
$$\begin{array}{cccccccc}
&\type<n-1 & & & \type=n-1 & & &\type=n\\
O'(\ppsi) = &
\begin{pmatrix}
S(\ppsi) & 0 & 0 & 0\\
 0  & 1 & 0 &0\\
0 & 0 & O(\ur) & 0\\
0 &0 & 0 & 1
\end{pmatrix} & &  & 
\begin{pmatrix}
S(\ppsi) & 0 &  0\\
 0  & 1 & 0\\
0 &0 &  1
\end{pmatrix} & &  &\begin{pmatrix}
S(\ppsi) & 0 \\
 0  & 1 
\end{pmatrix}
\end{array}
$$
\end{proposition}

\begin{proof} 
It is easy to check that $O'(\ppsi)$ fixes the basepoint and preserves the 
horofunction $h=h_{\ppsi}$ so $O'(\ppsi)\subset O(\ppsi)$.  For the converse, $\PGL(\Omega)\subset\Aff(\RR^n)$ so $O(\ppsi)\subset \Aff(\RR^n)$.
It is easy to check the result when $\type=n$, so assume $\type<n$.
From { Equation (\ref{eq:horo})} the horofunction  $h:\RRP^\type\times\RR^{n-\type}\to \RR$  is
$$h(x,z,y)=-\psi^\type(\Log(x)){-}z { +}\|y\|^2/2$$
If $\tau\in O(\ppsi)$ then $h=h\circ\tau$. Given a unit vector $u=(x,z,y)\in\RR^\rank\times\RR\times\R^{\ur}$ there is an affine line $\ell_u$ in $\RR^n$
 containing the basepoint that is the image of the map
 $\gamma_u(t)=b+t\cdot u$. The horofunction is only defined  on the subset of this line in $V_{\ppsi}$.
 This gives a  function $f=f_u:I_u\to\RR$ defined on some maximal interval
 $I_u\subset\RR$ by
$$f_u(t):=h\circ\gamma_u(t)={-}t z{ +} t^2\| y\|^2/2 -\sum_{i=1}^{\type} \psicoefi_i\log(1+tx_i)$$
here $x=(x_1,\cdots,x_{\type})$. We distinguish two classes of  line $\ell_u$ according to the behaviour of $f$. The function $f$ is defined on $I_u=\RR$
iff $x=0$, and it is defined on $[0,\infty)\subset I_u$ and grows logarithmically as $t\to\infty$  iff $z=y=0$ and each coordinate of $x$ is non-negative. Since $\tau$ is affine, it preserves the smallest affine subspace that contains all the lines of a given type. Since $\tau$ fixes the basepoint $b$, and preserves the type of lines, $\tau$ preserves the affine subspaces
$P=b+\langle e_{1},\cdots,e_\type\rangle$ and $Q=b+\langle e_{\type+1},\cdots,e_n\rangle$. Notice that $P=\langle e_{1},\cdots,e_\type\rangle$.

{ By Lemmas \ref{idealboundary} and \ref{characterizecenterRF}, $\tau$ preserves the simplex spanned by the ideal boundary $\bdy_{\infty}\Omega$ and the center $\alpha$ of the radial flow $\Phi$ of $\Omega$ (this simplex is exactly $\bdy_\infty \Omega$ unless $\type=n$ in which case it is larger). It follows that  $\tau$ permutes the vertices $\{[e_i]: 1\le i\le \type+1\}$ of this simplex.  }
On $P$ we have  $h(x_1e_1+\cdots x_\type e_\type)=-\sum\psicoefi_i\log x_i$.  Since $\tau|P$  preserves $h$, it follows that
must $\tau$ preserve $\psi|P$. Thus the first $\type$  columns of  $\tau$
 are as shown in $O'(\ppsi)$. 

The only $u$  for which $f_u$ is linear is when $u=\pm e_{\type+1}$.
Since $\tau$ fixes the basepoint and preserves $h$ 
it follows that $\tau$ maps the line $\ell_{e_{\type+1}}$ to itself by the identity. 
 This gives column $(\type+1)$ in $O'(\ppsi)$.
Finally  $f_u$ is a quadratic polynomial with a minimum of $0$ at the basepoint exactly when $x=0$ and $z=0$ so
 $u\in \langle e_{\type+2},\cdots,e_n\rangle$. On this subspace
$h(y_1e_{\type+2}+\cdots+y_{\ur}e_n)=\|y\|^2/2$. Since $\tau$ preserves this function, the columns $\type +2$ to $n$
 of $\tau$  in $O'(\ppsi)$ (those that contain $O(\ur)$) are as shown.
Since $\tau$ is affine and fixes the basepoint the last column is as shown in $O(\psi')$.
The result now follows.\end{proof}

A {\em morphism} between two geometries
$(G,X)$ and $(H,Y)$
is a homomorphism  $\rho:G\to H$, and an immersion
 $f:X\to Y$, such that  $$\forall g\in G, x\in X\quad f(g\cdot x)=\rho(g)\cdot(fx)$$
If $f$ and $\rho$ are both inclusions we say $(G,X)$ is a {\em subgeometry} of $(H,Y)$. 
{  \begin{theorem}\label{Gppsistructure}  $G(\ppsi)=T(\ppsi)\rtimes O(\ppsi)$ and:
\begin{itemize}
\item[(a)] $T(\ppsi)\cong\RR^{n-1}$ acts simply transitively on $\bdy\Omega(\ppsi).$
\item[(b)] $O(\ppsi)$ is the stabilizer of a point in $\bdy\Omega(\ppsi)$.
\item [(c)] $O(\ppsi)$ is a {maximal compact} subgroup of $G(\ppsi)$.
\item[(d)] $(G(\ppsi),\bdy\Omega)$ is isomorphic to a subgeometry of $(\Isom(\EE^{n-1}),\EE^{n-1})$
\item[(e)] $T(\ppsi)$ is the unique Lie subgroup of $G(\ppsi)$ isomorphic to $\RR^{n-1}$.
\item[(f)]  $T(\ppsi)$ is the subgroup of $G(\ppsi)$ of elements all of whose eigenvalues are positive. 
\end{itemize}
\end{theorem}

\begin{proof}   (a) and (b) follow from Lemma \ref{Tlemma} and Definition \ref{Oppssidef}. By Proposition \ref{olambda} $O(\lambda)$ is compact giving part of (c).
By (a) we may regard the orbit map $\tau:T(\ppsi)\to\bdy\Omega$ given by $\tau(g)=g(b_{\ppsi})$ as an identification,
then $O(\ppsi)$ acts smoothly on $T(\ppsi)$ fixing the identity. The derivative of this action acts linearly on the Lie algebra
of $T(\ppsi)$ as a compact group. Thus there is an inner product on the Lie algebra of $T(\ppsi)$ that is preserved by this action. Using left translation
gives a flat Riemannian metric on $T(\ppsi)$, which is therefore isometric to $\EE^{n-1}$. Then $\tau$
conjugates the action of $G(\ppsi)$ on $\bdy\Omega$ into a subgroup of $\Isom(\EE^{n-1})$. This proves (d). 

{\blue Clearly (d)
 implies the maximality claim in (c), as well as (e), and} also implies that $G(\ppsi)$ is an internal semidirect product as claimed.
{\blue A Euclidean isometry is conjugate by a translation to the composition
of an orthogonal element and a translation that commute. 
  It follows by (d) that  $g\in G(\ppsi)$  is conjugate to $a\cdot t$ with $t\in T(\ppsi)$ and $a\in O(\ppsi)$ and $a\cdot t=t\cdot a$.
By definition all eigenvalues of elements of $T(\ppsi)$ are positive. Since $a$ and $t$ commute the eigenvalues of $a\cdot t$
are products of eigenvalues of $a$ and of $t$. Thus, if all the eigenvalue of $g$ are positive, then all those of $a$ are positive.}
An  element of the orthogonal group with all eigenvalues positive is trivial which proves (f).
\end{proof}}

\begin{corollary}\label{parabolicsinGpsi} Every parabolic in $G(\ppsi)\subset\GL(n+1,\RR)$ is conjugate into $O(n,1)$.
\end{corollary}
\begin{proof}{ An element $g\in G(\ppsi)$  is parabolic iff all eigenvalues of $g$ have 
modulus $1$ and $g$ is not conjugate into  $O(n+1)$. Such $g$ is conjugate to $a\cdot t$ with $a\in O(\ppsi)$
and $1\ne t\in T(\ppsi)$ and $a\cdot t=t\cdot a$. Since the eignevalues of $t$ are all positive, they are all $1$, 
so $t\in P(\ppsi)$.
Thus $g$ is a standard parabolic.}
\end{proof}

From Equation (\ref{eq:radialflow}) the radial flow $\Phi\subset\GL(n+1,\RR)$ is given by
\begin{equation}\label{RFmatrix}
\begin{array}{ccc}  & \type<n & \type=n\\
\\
\Phi_s= &\exp\begin{pmatrix}0_{\type\times \type} & 0 & 0\\
0 & 0 & -s\\
0 & 0_{n-\type\times n-\type} & 0
\end{pmatrix}
&
\ \  \ \ \exp\begin{pmatrix}-s \operatorname{I}_{n\times n} & 0\\
0 &  0
\end{pmatrix}
\end{array}
\end{equation}
Observe that the one-parameter group $\Phi$ is a subgroup of $T_{\type}$ and $T_{\type}=T{(\psi)}\oplus\Phi$.
In particular  $T{=T(\ppsi)}$ commutes with the radial flow, so $T$ sends radial flows lines to radial flow lines.
Thus $T$ induces an action on  the space of flowlines in $V_{\ppsi}$, and radial projection identifies this space with $U_{\ppsi}$. The action
of $T$ on $U_{\ppsi}$ is affine, and given by
omitting row and column $\type+1$ to give
$$\begin{pmatrix}
 \exp\Diag(X) & 0 & 0 \\
 0& 
 I_{\ur} & Y\\
0 & 0 & 1
\end{pmatrix}\quad  Y\in\RR^{\ur}$$ 
with both $Y$ and $I_{\ur}$ interpreted as empty for $\ur= 0$. This happens when $\rank=n-1$. 
In the case $\type<n$ then $X\in\RR^\rank$ and when $\type=n$ then $X\in\ker\psi^\type$. From this it follows that:

\begin{lemma} Under radial projection $\pi:V_{\ppsi}\to U_{\ppsi}$
the action of $T(\ppsi)$ on $V_{\ppsi}$ is 
 semi-conjugate to a simply transitive affine action of $T(\ppsi)$ on $U_{\ppsi}$, 
{that} is  topologically conjugate to the action of $\RR^{n-1}$ on itself by translation.
\end{lemma}
\begin{proof}
The second conclusion follows by conjugating with the map $\R_+^{\rank}\times \RR^{\ur}\to\RR^{n-1}$ given by $(x,y)\mapsto(\Log x,y)$.
\end{proof}

\subsection{Domains preserved by $T(\ppsi)$}\label{extendedomainssec}
{  $\Omega=\Omega(\ppsi)$ is not the only properly convex domain preserved by $T=T(\ppsi)$. 
If $B\in\Aff(\RR^n)$ normalizes $T$, then $T$ also preserves $B(\Omega)$. However the cusp
$\Omega/\Gamma$ is affinely equivalent to $B(\Omega)/\Gamma$.
When $T$ is diagonal
there is a different class of examples given by gluing two copies of $\Omega$ along 
$\interior(\bdy_{\infty}\Omega)$, and then deleting one boundary component.}

\begin{definition}\label{signeddef} $\SS(\type)\subset\GL(n+1,\RR)$ is the the group of all diagonal matrices $\epsilon$ with $\epsilon_{i,i}=\pm1 $ for 
$i\le \type=\type(\ppsi)$ and $\epsilon_{i,i}=1$ for $i>\type$.  { Moreover $\SS(\type,\psi)$ is the subgroup of $\SS(\type)$ that normalizes $G(\ppsi)$
 }.
\end{definition}

\begin{lemma}\label{signdefnormal} { $\SS(\type)$ centralizes $T(\ppsi)$;
and  $\SS(\type,\psi)$ } consists of all $\epsilon\in \SS(\type)$ such that
 $\epsilon_{ k,k}=\epsilon_{j,j}$ whenever $\ppsi_{ k}=\ppsi_j$. Furthermore, $\SS(\type,\ppsi)$ also centralizes $G(\ppsi)$.
\end{lemma}
\begin{proof} { The first statement easily follows from the presentation of $T(\ppsi)$, see Equations (\ref{eq:Tmatrix2}) and (\ref{eq:Tmatrix}).}
{ By} Proposition \ref{olambda}, { we may} regard $S(\ppsi)$ and $O(\ur)$ as subgroups of $\Aff(\R^n)\subset\GL(n+1,\RR)$ 
acting on $\RR^n$.
It is easy to check that  $\SS(\type)$ centralizes   $O(\ur)$. 
An element $A\in S(\ppsi)$  permutes the $x_i$ coordinates for $1\le i\le \type$, and $\epsilon\in\SS(\type)$ assigns a sign
to each of these coordinates so that
\begin{equation}\label{signedperm}
(\epsilon A\epsilon)_{j,k}=\epsilon_{j,j}\epsilon_{k,k}A_{j,k}\end{equation} 
is a signed permutation. Thus $\epsilon\in\SS(\type,\ppsi)$ if and only if $\epsilon_{ k,k}=\epsilon_{j,j}$ whenever $\ppsi_{ k}=\ppsi_j$. Moreover in this case $\epsilon$ commutes with $A$.
\end{proof}

For $1\le i\le \type$ let $H_i\subset\RP^n$ be the hyperplane $x_i=0$.  
{ Then $X:=\RP^n\setminus(\RP_{\infty}^{n-1}\cup_i H_i)$
has $2^\type$ components, each affinely equivalent to $V_{\ppsi}$.} It is easy to check that:
\begin{lemma}\label{Xtransitive}
 $\SS(\type)$ acts { simply} transitively on the components of $X$, and $T_{\type}\oplus \SS(\type)$ acts { simply} transitively on $X$.
 \end{lemma}
It follows that the only projective hyperplanes that are preserved by $T(\ppsi)$ are $\RP_{\infty}^{n-1}$ 
and the hyperplanes $H_i$.
{ If $g\in  \SS(\type)\oplus \flow^{\ppsi}$ then  $g(\Omega)$ is called a {\em standard domain}.
Since $g$ normalizes $T(\ppsi)$,  this domain
 is preserved by $T(\ppsi)$.  Since 
 $\blue\Phi_t\left(\Omega(\ppsi)\right)\subset\Phi_s\left(\Omega(\ppsi)\right)$
 if  $t\le s$, it follows that standard domains intersect if and only if one contains the other.}

 A properly convex set $U\cong\bdy U\times[0,\infty)$ that is preserved by the action of $T(\ppsi)$ is called {\em reducible} if there is a projective
 hyperplane $H$ that is preserved by $T(\ppsi)$ and $H\cap U\ne\emptyset$, otherwise $U$ is {\em irreducible}. If such $H$
 exists then $H$ separates $U$ into two properly convex sets that are preserved by $T(\ppsi)$. It follows from the above
 that, if $U$ is irreducible, then $U$ is contained in some component of $X$.

\begin{lemma}\label{invariantdomains}
 If $U\subset\RPn$ is an irreducible  properly convex set
  that is preserved by $T(\ppsi)$, and $U\cong\bdy U\times[0,\infty)$, then 
 $U$ is a standard $\ppsi$-domain.
Moreover there is a unique $g\in \SS(\type)\oplus\Phi^{\psi}$ such that
$U=g(\Omega(\ppsi))$. 
\end{lemma}
\begin{proof} There is unique $\epsilon\in\SS(\type)$ such that $\epsilon(U)\subset V_{\ppsi}$.
If $x\in\bdy U$ then there is $h\in T_{\type}$ such that $h\circ \epsilon(x)\in\bdy\Omega(\psi)$. Since $T(\psi)$
acts simply transitively on $\bdy\Omega$, and is the subgroup of $T_{\type}\oplus\Phi^{\ppsi}$ that preserves $\bdy\Omega$,
 it follows there is a unique $h\in \Phi^{\psi}$ with this property, and $g=h\circ\epsilon$.
\end{proof}

 When $\type=\type(\ppsi)=n$
  let $\minusone\in \SS(\type)$ be the map
   that restricts to be the affine map of $\RR^n$ given by $x\mapsto-x$. In
 the above notation $\minusone_{i,i}=-1$ for all $1\le i\le n$ and $I_{n+1,n+1}=1$.
 { In what follows, let $\Omega\subset V:=V_{\psi}$ and $\Omega'\subset \minusone(V)$ be standard domains and
  observe that 
 $$\Delta^{n-1}\cong\bdy_{\infty}\Omega=\bdy_{\infty}\Omega'$$
Then
 $$W:=\Omega\sqcup\interior(\Omega')\sqcup \interior(\bdy_{\infty}\Omega)\subset\RPn$$
is called an {\em extended domain}.
 
 \begin{lemma}\label{extendeddomain}
 The extended domain $W$ is properly convex,  preserved by $G(\ppsi)$, and $W\cong \bdy W\times[0,\infty)$. 
 \end{lemma}
 \begin{proof} At each point $x\in\Fr W=(\bdy\Omega\sqcup\bdy\Omega')\sqcup\bdy_{\infty}\Omega$ there is a supporting hyperplane $H$. 
 If $x\in\bdy_{\infty}\Omega$ then $H$ is the projectivization of some coordinate hyperplane $x_i=0$. For 
 $x\in\bdy\Omega\sqcup\bdy\Omega'$  it is clear $H$ exists.
 Moreover $\cl(W)$ is disjoint from the projectivization of the affine hyperplane $\sum x_i=0$, so $W$ is properly convex.
  
  In what follows, closure is taken in $\RPn$. Observe that $W=A\cup A'$ where $A=\cl(\Omega)\setminus\bdy(\bdy_{\infty}\Omega)$ and $A'=\cl(\Omega')\setminus\cl(\bdy\Omega')$ so $A\cap A'=\interior(\bdy_{\infty}\Omega)$.
Then $A \cong\interior(\bdy_{\infty}\Omega)\times [0,1]$
 and  $A'\cong \interior(\bdy_{\infty}\Omega)\times [1,\infty)$ so $W\cong \interior(\bdy_{\infty}\Omega)\times[0,\infty)$.
 Since $G(\ppsi)$ preserves $A$ and $A'$ it preserves $W$.
 \end{proof} }
   
     \begin{proposition}\label{Ginvdomains} If $U\subset\RP^n$ is an open properly convex set that is 
  preserved by $T(\ppsi)$,  and $U\cong\bdy U\times[0,\infty)$, then either 
$U$ is a standard $\ppsi$-domain,
or else $\type=n$ and $U$ is an extended domain. \end{proposition}
\begin{proof} { As usual we drop $\ppsi$ from the notation}.
Since $T$  preserves { each component of $X$}, it preserves 
 each component of $U\cap  X$. The latter are
properly convex so by Lemma \ref{invariantdomains}, {the closure in $\RR^n$ of each of} 
these components is
a standard domain. It suffices to show that if there is more than one component, then $\type=n$ and there
are exactly two components. 

If there is more than one component then, since $\blue U$ is connected, the closure in $\RP^n$ of two distinct 
components must intersect. { We may assume one component $\Omega$ is contained in $V$
and  the other is $g\Omega$ for some $g\in \SS(\type)\oplus\Phi$.}
The intersection {$\Omega\cap g\Omega$} is contained in $\bdy_{\infty}\Omega\cong\Delta^{\rank}$ and
separates the open set  $\blue\interior(U)$. It follows that $\rank=n-1$ so $\type=n-1$ or $\type=n$, {\blue and that there are at most two components.}

We claim that if $\type=n-1$ then  $ U$ is not convex. 
This is because {using Definition \ref{psidomain}} the intersection of $\Omega$ with the 2 dimensional affine
subspace given by $x_i=1$ for $ i< n-1$  is $x_{ n}=-\ppsi_{ n-1}\log x_{ n-1}$
which looks like $y=-\log x$ shown in Figure \ref{dpic}. In this case it is clear that an extended domain is not convex at the right hand { endpoint of $\bdy\Omega(1,0)$}. 
If $\type=n$ then $g$ must preserve $\bdy_{\infty}\Omega$ which implies $ g\in\minusone\circ\Phi$ 
 completing the proof. 
\end{proof}

\begin{corollary}\label{holdeterminespsicusp}
 If $C$ is a generalized cusp with holonomy  $\Gamma\subset G(\ppsi)$ then $C$ is equivalent to a $\ppsi$-cusp.
\end{corollary}
\begin{proof}  We have $C=U'/\Gamma$ for some $U'\cong \bdy U'\times[0,\infty)$ that is preserved by $\Gamma$.
By { Theorem 6.3 from \cite{CLT2}} there is a $G(\ppsi)$-invariant subset $U\subset U'$ and $U/\Gamma$
is
equivalent to $C$, so $U\cong\bdy U\times[0,\infty)$. By Proposition \ref{Ginvdomains}, either $U$ is a standard $\ppsi$-domain
or else an extended domain. 
Otherwise, if $U$ is extended, then $U$ contains a  standard domain, $\Omega$, that is $G(\ppsi)$ invariant, and 
$C$ is equivalent to the $\ppsi$-cusp $\Omega/\Gamma$.
\end{proof}

If $C'$ is a generalized cusp 
that properly contains another generalized cusp $C$, 
and they have the
same boundary, then $\type=n$ and the holonomy is diagonalizable.
Equivalent cusps are not always projectively equivalent after
removing suitable collars of the boundary.  If $\type=n-1$, 
then $\bdy_{\infty}\Omega(\ppsi)\cong\Delta^{n-1}$, but there is 
no larger $G(\ppsi)$-invariant
domain
that contains $\bdy_{\infty}\Omega(\ppsi)$ in its interior.

\subsection{Hex geometry} { In this section $\opensimplex=\opensimplex^\rank$ denotes the interior of a simplex $\Delta=\Delta^{\rank}$}.
Let $\{v_0,\cdots,v_\rank\}\in\RR^{\rank+1}$ be a basis,
then $[v_i]$ are the vertices of an $r$-simplex $\Delta$. The identity component  $D^\rank\subset\PGL(\opensimplex)$ 
is the projectivization of the
 positive diagonal subgroup,
and $\PGL(\opensimplex)= D^\rank\rtimes S_{\rank+1}$  is an internal semidirect product, where $S_{\rank+1}$ 
is the group of coordinate permutations. 

\begin{definition}\label{Hexdef} The {\em $\rank$-dimensional  Hex geometry} 
 is { $\Hex^\rank=(\PGL(\opensimplex^\rank),\opensimplex^\rank)$}.\end{definition}

{ Let $\{u_i:0\le i\le\rank\}\subset\RR^{\rank+1}$
be a spanning set of unit vectors with $\sum u_i=0$. 
The map $[\sum x_iv_i]\mapsto \sum (\log |x_i|)u_i$}
is an isometry taking
$(\opensimplex,d_{\opensimplex})$ to a certain normed vector space $(\RR^\rank,\|\cdot\|)$. 
The name {\em Hex geometry}
 comes from the fact that when $\rank=2$, the unit ball is a regular hexagon.  
 It follows that
$(\Isom(\opensimplex),\opensimplex)$ is isomorphic to a subgeometry of Euclidean geometry. Moreover  $\PGL(\opensimplex)$ is an index-2 subgroup of $\Isom(\opensimplex,d_{\opensimplex})$. This is all due to de la Harpe \cite{delaharpe}.

Recall that $\psicoef_1\ge\psicoef_2\ge\cdots\ge\psicoef_\rank>0$ and $\psicoef_i=0$ for all $\rank<i\leq n$. 
Recall Proposition \ref{olambda} that $S(\ppsi)\subset\PGL(\opensimplex^\rank)$ is the group of coordinate permutations that preserve $\sum_{i=1}^{\rank}\psi_ie_i$. 
It is clear that $S(\ppsi)$ is isomorphic to a product of symmetric groups $\prod S_{k_j}$. There is one factor 
isomorphic to the symmetric group $S_k$
for each maximal consecutive
sequence $\psicoef_i=\psicoef_{i+i}=\cdots=\psicoef_{i+k-1}$ of non-zero coordinates in $\ppsi$.


\begin{definition}\label{hexlambdadef} The subgeometry $(D^\rank\rtimes S(\ppsi),\opensimplex^\rank)$ of
$\Hex^\rank$
is called $\Hex^\rank(\ppsi)$
\end{definition}

When $X$ is a metric space 
we { denote} {\em $(\Isom(X),X)$-geometry} { by $X$}. For example $\HH^n$ is hyperbolic geometry in dimension $n$. { The geometry $[0,\infty)$ has $G= \{ 1\}$.}
The {\em product geometry} of $(G,X)$ and $(H,Y)$ is $(G\times H,X\times Y)$ with the product action.
 {\em Horoball geometry} is the  
subgeometry $\horogeom^{\ur+1}=(\Bcal,G)$ of $\HH^{\ur+1}$ where $\Bcal\subset\HH^{{ \ur}+1}$
is a horoball, and $G\subset\Isom(\HH^{\ur+1})$ is the subgroup that preserves $\Bcal$. 
In the following theorem interpret both $\Hex^0(\ppsi)$ { and} $\EE^0$ as the trivial geometry on one point,
and $\horogeom^0$ as the trivial geometry  $[0,\infty)$.

\begin{theorem}\label{lambdageometryproduct} $(G(\ppsi),\Omega(\ppsi))$ is isomorphic to
 the product geometry  $\Hex^{\rank}(\ppsi)\times \horogeom^{\ur+1}$ and also to
 $\Hex^{\rank}(\ppsi)\times\EE^{\ur}\times[0,\infty)$.
 \end{theorem}
\begin{proof}  
In what follows most  functions and sets should be decorated with $\ppsi$. This is  often omitted for clarity.
First assume $\type<n$.
{ The diffeomorphism} $\theta:V_{\ppsi}\to V_{\ppsi}$ { is defined  by}
$$\theta(x,z,y)=(x,z+\psi^\type(\Log x),y)$$
 { By Equation (\ref{eq:graph})}  $\bdy\Omega$ is the graph $z=f(x,y)$ { and} it follows that
 $\theta(\bdy\Omega)$ is the graph of $z=f(x,y)+\psi^\type(\Log x)$. 
 Using Equation (\ref{eq:Logpsif}) this simplifies to $z=\|y\|^2/2$ when $\ur>0$ and to $z=0$ when $\ur=0$.
  In each case $\theta(\Omega)=\opensimplex^{\rank}\times B$ where
$${ \opensimplex^{\rank}:= (0,\infty)^{\rank}},\qquad B:= \{(z,y)\in\R\times\RR^\ur: z\ge  \|y\|^2/2\ \}$$
and $G^{\theta}:=\theta\circ G\circ\theta^{-1}$ acts on this set. 
This gives an isomorphism of geometries $(G,\Omega)\to(G^{\theta},{ \opensimplex^{\rank}}\times B)$.

The subgroup  $T^{\theta}:=\theta\circ T(\ppsi)\circ \theta^{-1}$ of $G^{\theta}$ acts on {$\theta(\Omega)$} by the affine transformations of $\RR^n$ 
$$T^{\theta}=\begin{pmatrix}
 \exp\Diag(X) & 0 \\
 0& \begin{pmatrix}
1 & Y^t & \|Y\|^2/2\\
0 & I_{\ur} & Y\\
0 & 0 & 1
\end{pmatrix}
\end{pmatrix}\qquad X\in\RR^\rank,\ Y\in\RR^{\ur}$$ 
 { By Proposition \ref{olambda} $(O(\ppsi))^{\theta}=O(\ppsi)$, which acts affinely on $\RR^{n+1}$}. By Corollary \ref{Gppsistructure},  $G(\ppsi)=T(\ppsi)\rtimes O(\ppsi)$
and it follows that the action of $G^{\theta}=G_{{ \opensimplex^{\rank}}}\times G_B$ is affine and splits into the  direct sum of actions on 
$\RR^{\rank}\oplus\RR^{\ur+1}$ given by 
$$G_{{ \opensimplex^{\rank}}}=D^\rank\rtimes S(\ppsi),
\qquad G_B:=\begin{pmatrix}
1 & Y^t & \|Y\|^2/2\\
0 & O({\ur}) & Y\\
0 & 0 & 1
\end{pmatrix}$$
Then $(B,G_B)\cong\horogeom^{\ur+1}$, which is obviously isomorphic to $\mathbbm{E}^\ur\times[0,\infty)$.

For $\type=n$ the {set} $\Omega(\ppsi)$ has a product structure coming from the horospheres, and the radial flow. The group $G(\ppsi)$ {acts} trivially on the radial flow factor,
 and projection along the radial flow gives a $G(\ppsi)$-equivariant diffeomorphism {from} each horosphere {to}
  $\partial_\infty\Omega(\ppsi)\cong \Delta^{n-1}$. 
\end{proof}

\begin{corollary}\label{lambdageomEuclidean} $(G(\ppsi),\Omega(\ppsi))$ is isomorphic to a subgeometry of Euclidean geometry.
\end{corollary}
\begin{proof} Each of the factors in Theorem \ref{lambdageometryproduct} is isomorphic to a subgeometry of Euclidean geometry.
\end{proof}

{ The next section gives a particular isomorphism}.

\section{Euclidean Structure}\label{euclideanstructure}
This section is devoted to showing that a generalized cusp has an
 underlying Euclidean structure with flat (totally geodesic) boundary. 
 This provides a natural map from a generalized cusp to a 
 standard cusp, modelled on $\HH^n$.
 A metric is first defined on
 $V_{\ppsi}\subset\RR^n$ in terms of a horofunction, and may be viewed as
 a kind of modified  {\em Hessian metric} \cite{Shima}.

\begin{theorem}\label{betathm} 
Let $h=h_{\ppsi}$ be  the horofunction on 
 $V=V_{\ppsi}$. Given $q\in V$ let $\Hcal$ be the horosphere containing $q$ 
and $\pi:V\rightarrow \Hcal$ be 
projection
along the radial flow.  Then
{ $\beta=(\deriv^2h\circ \deriv\pi)+ (\deriv h)^2$  is a quadratic form on} $T_qV$ that defines a Riemannian metric on $V$ and:
\begin{enumerate}
\item[(a)] There is an isometry $F:(V,\beta)\to (\RR^n,\|\cdot\|^2)$ where $\|\cdot\|^2$ is the standard Euclidean metric.
\item[(b)]  $F(\Omega(\ppsi))=\RR^{n-1}\times(-\infty,0]$.
\item[(c)]  The horofunction is  the $n$'th coordinate of $F$ i.e. $h(p)=F_n(p)$.
\item[(d)]  The action of $G(\ppsi)$ on $V$ is by isometries of this metric.
\item[(e)]   The radial flow $\Phi_{t}$ on $V$ is conjugated by $F$ to $x\mapsto x{ +}t\cdot e_n$
\item[(f)]  The radial flow acts on $V$ by isometries.
\item[(g)]  Radial flow lines are orthogonal to  horospheres.
\item[(h)] The action of $T(\ppsi)$ on $\bdy\Omega(\ppsi)$ is conjugated by $F$ to the group of translations of $x_n=0$. 
\end{enumerate}
\end{theorem}
\begin{proof} { In what follows derivatives are  at $q$,
so $\deriv$ means $\deriv_q$ and so on.}
Cleary $\beta$ is symmetric and  we first verify that it is also positive definite.
Given $q\in V$ let  
$\mathcal{H}\subset V$ be the horosphere   containing  $q$.
The radial flow line through $q$ is $f:\RR\to V$, given by
$f(t)=\Phi_t(q)$, and is transverse to $\Hcal$. 
Thus $T_qV=T_q\mathcal{H}\oplus \RR\cdot v$ where $v=f'(0)$ is tangent to the radial flow at $q$.
If $w\in T_qV$ then $w=a+t\cdot v$ for some $a\in T_q\mathcal{H}$ and $\deriv\pi(w)=a$. 

Observe that $T_q\mathcal{H}=\ker \deriv h$. From Equation (\ref{eq:flowequivariance}) 
 $\deriv h(v)=1$
 so $(\deriv h)^2(a+t\cdot v)=t^2$. Thus
 \begin{equation}\label{eq:beta}
 \beta(w)=(\deriv^2 h)(a) + t^2
 \end{equation} and it suffices to check that $\deriv^2h$ is positive definite on 
 $\ker\deriv h$.

 When $\rank=n$
\begin{equation} \deriv^2 h=\left(\sum_{i=1}^n\psicoef_i\right)^{-1}\sum_{i=1}^n \psicoef_ix_i^{-2}dx_i^2
\end{equation} 
and since all $\psicoef_i>0$, and  $x_i>0$ on $V$, it follows that $\deriv^2 h$ is positive definite on $T_qV$, 
and so  is positive definite on $\ker \deriv h$.

When $\rank<n$ 
\begin{equation}
\deriv h =
-dx_{\rank+1}-\sum_{i=1}^{\rank}\psicoefi_ix_i^{-1}dx_i + \sum_{i=\rank+2}^{n}x_idx_i
\end{equation}

\begin{equation}\label{d2h}
\deriv^2h =
\sum_{i=1}^{\rank}\psicoefi_ix_i^{-2}dx_i^2 + \sum_{i=\rank+2}^{n}dx_i^2
\end{equation}

In this case, by Equation (\ref{eq:radialflow}), the radial flow is vertical translation and $v=-\partial/\partial x_{\rank+1}$.
Thus $\deriv^2h$ is  positive semi-definite 
and vanishes only in the $v$-direction, hence it is positive definite on $\ker\deriv h$.

 Thus $\beta$ is  a Riemannian metric on $V$. Since $G(\psi)$ preserves $h_{\psi}$
  { and commutes with $\pi$},
 it acts by isometries of $\beta$ proving (d). The radial flow preserves $h$ up to adding a constant, and so preserves $\deriv h$ and $\deriv^2h$,
 and is therefore also an isometry of $\beta$ proving (f).
Hence the extended translation group $T_{\rank}$ acts by isometries of $\beta$. Since this action is simply transitive
 we may identify $T_{\rank}$ with $V$. Since $T_{\rank}\cong\RR^{n}$ as a Lie group, it follows that this metric is flat, so
 there is an isometry 
 $F:(V,\beta)\to(\RR^n,\|\cdot\|^2)$, proving (a).  
 We use $(u_1,\cdots,u_n)$ as the coordinates of a point in the codomain $\RR^n$.

 Each horosphere in $V$ is the orbit of a point under the subgroup 
 $\RR^{n-1}\cong T(\ppsi)\subset T_\rank$
 therefore the horospheres are identified with parallel  hyperplanes in $\RR^n$. 
 We can choose the isometry  $F$ 
 so that the horosphere $\bdy\Omega$ is sent to the subspace $u_n=0$, and so that $\Omega$
 is identified with the half space  $u_n\le 0$. 
 
Observe that $T_qV=T_q\Hcal \oplus \RR\cdot v$, and $\beta$ is the sum of two
 quadratic forms, each of which vanishes on one summand and is positive definite on the other.
 It follows the two summands are orthogonal with respect to $\beta$, which proves (g).
 
From (g) it follows that flow lines are lines parallel to the $u_n$ direction.
Along a flow line $\beta$ is $(\deriv h)^2$ so the distance between $x$ and $\Phi_t(x)$ 
is $|h(\Phi_t(x))-h(x)|=|t|$ by Equation (\ref{eq:flowequivariance}). Moroever since $\Omega$ is $u_n\le 0$
 the radial flow $\Phi_t$  is conjugated by $F$ to
$u\mapsto u+t\cdot e_n$. This proves (b),  (c), (e) and (h). 
 \end{proof}
 
 {  \begin{definition}  Set $I=[0,\infty)$.
  Suppose $(A,ds_A)$ is a Euclidean manifold and $(I,dt)$ is a complete Riemannian metric on $I$.
 The metric $ds^2=ds_A^2+dt^2$ on $A\times I$ is called a {\em product Euclidean structure}. 
 Given $c>0$ the metric $c\cdot ds_A^2+c^2 dt^2$ is called a {\em horoscaling} of  $ds^2$. A diffeomorphism
 $f:A\times I\to M$ is a {\em horoscaling} if the pullback of the metric on $M$ is a horoscaling of the metric on $A\times I$. A horofunction metric on $\Omega$ is a horoscaling of the metric $\beta|\Omega$ in Theorem \ref{betathm}.
  \end{definition}

Thus $\beta|\Omega$ is a product Euclidean structure, and the metric $\beta^c$ obtained by replacing  $h$  in the definition of $\beta$ by  $c\cdot h$
 is a horoscaling of $\beta$
      \begin{equation}\label{horofunctionmetric}
\beta^c=(\deriv^2(c\cdot h)|\ker \deriv c\cdot h)+ (\deriv c\cdot h)^2= c(\deriv^2h|\ker \deriv h)+ c^2(\deriv h)^2
\end{equation}

\begin{proposition}\label{horometricesinvariant} Suppose $\Omega=\Omega(\ppsi)$ has horofunction metric $\beta$ and $\Omega'=\Omega(\ppsi')$ has horofunction metric $\beta'$ and
 $P\in\PGL(\RR,n+1)$ and $P(\Omega)=\Omega'$, then $P^*\beta'$
is a horoscaling of  $\beta$. Thus the set of horofunction metrics on $\Omega$ is an invariant of the projective equivalence
class of $\Omega$.
\end{proposition}
\begin{proof}
 $h'\circ P=c\cdot h$ for some $c>0$
  by Proposition \ref{invariantproduct}(f), so $P^*\beta'$ is a horoscaling of $\beta$ by Equation (\ref{horofunctionmetric}). \end{proof}
 
 \begin{proposition}\label{invariantEuclmetric}  Let $\Omega=\Omega(\ppsi)$ and $\type=\type(\ppsi)$. If $\type>0$ then $\PGL(\Omega)=G(\ppsi)$.  If $\type=0$ then $\PGL(\Omega)$ acts by horoscalings,
 and $G(0)\vartriangleleft \PGL(\Omega)$ and $\PGL(\Omega)/G(0)\cong \RR$.
\end{proposition}
\begin{proof}
Proposition \ref{horometricesinvariant} implies  $\PGL(\Omega)$ acts by horoscalings.
Thus if $P\in\PGL(\Omega)$ then $P|\bdy\Omega$ 
 is  a Euclidean similarity of $\beta|\bdy\Omega$.
   If $P$ is not an isometry of $\beta$, after replacing $P$ by $P^{-1}$ if needed, we may assume $P$ is a contraction of $\bdy \Omega$.
Then there is a  point $x\in\bdy\Omega(\ppsi)$ fixed by $P$. 
 By Theorem \ref{Gppsistructure}(a) the group $T$ acts simply transitively on $\bdy\Omega$
and this
gives an identification  $T(\ppsi)\equiv\bdy\Omega(\ppsi)$ via $t\mapsto t(x)$. 
 Under this identification $x$ is identified with the $id\in T(\ppsi)$.

Let $U\subset\bdy\Omega(\ppsi)$ be the ball of $\beta$-radius $1$ with center $x$. 
 Then $P(U)\subset U$ is a ball of strictly smaller radius.  Under the identification, $U$ gives a neighborhood
 $V\subset T$ of the identity in $T$, and $P V P^{-1}\subset V$ is a strictly smaller neighborhood. 
 This inclusion is given by conjugacy thus $T$ is unipotent, so $\type=0$. Thus if $\type>0$ then $P$ is an isometry of $\beta$. 
 A horosphere
 in $\Omega(\ppsi)$ is  characterized as the set of points some fixed $\beta$-distance from $\bdy\Omega(\ppsi)$. Therefore $P$
 preserves each horosphere. Hence $P\in G(\ppsi)$.
 \end{proof}

 If $\type=n$ and $s>0$ then $T(\ppsi)=T(s^2\cdot\ppsi)$. By Equation (\ref{eq:Tmatrix}), if $\type<n$ then  \begin{equation}\label{Tconj}
T(s^2\cdot\ppsi)=P T(\ppsi) P^{-1}\qquad {\rm with}\quad P= \Diag(I_{\rank},s,I_{\ur},s^{-1})
\end{equation}
In both cases $\Omega(\ppsi)$ is projectively equivalent to $\Omega(s^2\cdot\ppsi)$. 

\begin{proposition}\label{transgpscale}\label{domaintransgp} If $\ppsi,\ppsi'\in A$ and $\Omega=\Omega(\ppsi),\ \Omega'=\Omega(\ppsi'),\  T=T(\ppsi)$ and $T'=T(\ppsi')$ 
then the following are equivalent
\begin{itemize}
\item[(a)] There is a projective transformation that sends $\Omega$ to $\Omega'$ 
\item[(b)] $T$ and $T'$ are conjugate in $\GL(n+1,\RR)$
\item[(c)]  $\ker\ppsi=\ker\ppsi'$
\item[(d)] $\ppsi'=c\cdot\ppsi$ for some $c>0$
\end{itemize}
\end{proposition}
\begin{proof}	$(d)\Leftrightarrow(c)$ is immediate because $\ppsi\ne0$ implies $\ppsi(e_1)>0$ and similarly for $\ppsi'$. Suppose $P\in\GL(n+1,\RR)$ and $P T P^{-1}=T'$.
Since $T$ acts transitively on
 $\bdy\Omega$ it follows that $P(\bdy\Omega)$ is a $T'$-orbit.  By Lemma \ref{Xtransitive}
 there is a projective transformation $R$ taking this orbit to $\bdy\Omega'$. After replacing $P$
 by $R\circ P$ we may assume $P(\bdy\Omega)=\bdy\Omega'$.
  But $\Omega$ is the convex hull of $\bdy\Omega$, so $P(\Omega)=\Omega'$. Thus $(b)\Rightarrow (a).$
 Conversely if $P(\Omega)=\Omega'$ then $S:=P T P^{-1}$ is contained in $G(\ppsi')$
 by Proposition
 \ref{horometricesinvariant}.
Since $S\cong\RR^{n-1}$,  
Theorem \ref{Gppsistructure}(f) implies $S=T'$ proving $(a)\Rightarrow(b)$.

$(c)\Rightarrow (b)$ by Equation (\ref{Tconj}).  This leaves $(b)\Rightarrow(c)$. If $\rho:G\to\GL(V)$ then a subspace $0\ne U\subset V$ is a {\em weight space}
 with (real) weight $\lambda:G\to\RR^{\times}$ if
$U=\{v\in V:\ \forall\ g\in G\ \ \rho(g)\cdot v=\lambda(g)\cdot v\}$. By Equations (\ref{eq:Tmatrix}) and (\ref{eq:Tmatrix2}), the weight spaces of $m^*_{\ppsi}$ are $\langle e_i\rangle$, and 
the points $[e_i]\in\RPn$ are the fixed points of $T(\ppsi)$, for $1\le i\le\type+1$. 

  If $P T P^{-1}=T'$ then  $P$ sends the fixed
points of $T$ to those of $T'$. By the first paragraph we may assume $P(\Omega)=\Omega'$.
The center of the radial flow for both $\Omega$ and $\Omega'$ is $\alpha=[e_{\type+1}]$
and Proposition \ref{characterizecenterRF}(b) implies $P(\alpha)=\alpha$.
 Thus $P$ preserves the $T$-invariant 
  subspace $W= \langle e_1,\cdots,e_{\type}\rangle$.
The definition of $T_1$ is conjugacy invariant so
$PT_1 P^{-1}=T_1'$. 

{\blue 
Restricting the action of $T_2$ to $W$ gives a subgroup $T_2^W\subset\GL(W)$ of 
the positive diagonal group, 
and $T_1^W$ is the subgroup of $T_2^W$ given by $\ker\psi$.
The conjugacy $R=(P|W)\in GL(W)$ from $T_1^W$  to $(T_1')^W$ permutes the fixed points
 $\{\langle e_1\rangle,\cdots,\langle e_{\type}\rangle\}$ of $T_1^W$.
Since $T_1^W$ is diagonal, conjugating by a diagonal matrix centralizes $T_1^W$,
so we may assume $R$ is  a permutation matrix. Then
$RT_1^WR^{-1}=(T_1')^W$
and $R$ permutes the coordinates of $\psi$. But the coordinates of both $\psi$ and $\psi'$
are in decreasing order, so $R$ must leave $\psi$ unchanged, thus $RT_1^WR^{-1}=T_1^W$.
 It follows that $T_1^W=(T_1')^W$, which implies (c).}
\end{proof}
}

\subsection{Normalizing the metric}

Given that $\Omega(\ppsi)$ comes equipped with a family of Euclidean (flat)
 metrics, it is natural to ask if there is any intrinsic way of distinguishing different metrics. 
 When $\ppsi=0$ then the interior of $\Omega(0)$ can be identified with
 $\HH^n$ and for each $c>0$ there is a (hyperbolic) 
 element $\gamma\in \PGL(\Omega)\subset \Isom(\HH^n)$  that rescales the horofunction:
 $h_0\circ\gamma=c \cdot h_0$. As a result,  there is no projectively invariant way
 to assign a distinguished metric to $\Omega(0)$.
This corresponds to the familiar fact that the complement of a point in the sphere at infinity
 for $\HH^n$ only
has an invariant Euclidean {\em similarity structure} 
rather than a Euclidean metric. But when $\ppsi\ne 0$ the story is different.

If $(X,d)$ is a metric space and $f:X\to X$ is an isometry
the {\em displacement distance} of  $f$ is $\delta_{ d}(f)=\inf\{d(x,fx):x\in X\}$. 
If $\ppsi\ne0$ define  the subset $J(\ppsi)\subset T_2(\ppsi)$ to consists of all $ A\in T_2(\ppsi)$ such that
  the largest eigenvalue of $A$ is { $\exp(1)$}. This set is non-empty and compact. A horofunction metric, $\beta$, on $\Omega(\ppsi)$ is {\em normalized} if $\sup\{\delta_{\beta}(A):A\in J(\ppsi)\}=1$. This metric is Euclidean by Theorem \ref{betathm}.
  If $C=\Omega(\ppsi)/\Gamma$
  is a generalized cusp the {\em normalized horofunction metric} on $C$ is the metric covered
  by $\beta$.

\begin{corollary}\label{normalizedmetric} If $\ppsi\ne0$ then there is a unique normalized horofunction metric on $\Omega(\ppsi)$.
If $P\in\PGL(n+1,\RR)$ and $P(\Omega(\ppsi))=\Omega(\ppsi')$ then 
$P$ is an isometry between the normalized horofunction metrics.

There is a  unique normalized horofunction metric on a $\ppsi$-cusp $C=\Omega(\ppsi)/\Gamma$
{and $C$ is isometric to 
 the  Euclidian 
 manifold  $\bdy C\times[0,\infty)$  with a product metric, and $\partial C$ is a compact Euclidean manifold.} 

If $C$ and $C'$ are generalized cusps {of type $\type >0$}, and $P:C\to C'$ is a projective diffeomorphism, then $P$ is an
isometry between these metrics.
\end{corollary}
{ \begin{proof} Unicity is clear. Suppose $\Omega=\Omega(\ppsi)$ (resp. $\Omega'=\Omega(\ppsi')$)
 has normalized horofunction metric
$\beta$ (resp. $\beta'$).
If $P(\Omega)=\Omega'$ then by Proposition \ref{horometricesinvariant}
 $P^*\beta'$ is a horoscaling of $\beta$. The normalization condition implies these metrics
are equal because
conjugation does not change eigenvalues. By Theorem \ref{betathm} $\beta|\Omega$ is
isometric to a Euclidean half-space, and $G(\ppsi)$ preserves this metric,
so $\beta$ covers a Euclidean metric on $C$ with $\bdy C$ flat. The last conclusion follows from
the second conclusion.
\end{proof}
}

 \subsection{The Second Fundamental Form}\label{secfundform}
 \begin{definition} 
 Suppose $S\subset\RR^n$ is a transversally oriented, smooth hypersurface. The {\em second fundamental form $\secondfund$ on $S$} is the quadratic
 form defined on each tangent space $\secondfund_q:T_qS\rightarrow \RR$ by 
 $$\secondfund_q(\gamma'(0))=\langle\gamma''(0), n_q\rangle$$
 where $\gamma:(-\epsilon,\epsilon)\to S$ is a smooth curve in $S$ with $\gamma(0)=q$ and $n_q$ is a unit normal vector to $S$ at $q$ in the direction
 given by the transverse orientation.
 \end{definition}
 
It is routine to verify that $\secondfund$ is well defined.
The sign of $\secondfund$ depends on a choice of normal orientation. If $S$ is a convex hypersurface and $q\in S$
 and $n_q$ points to the convex side, then $\secondfund_q$ is
 positive definite and  defines a Riemannian metric on $S$ see \cite{Sch}. 
 {There is a cotangent vector $\eta_q\in T_q^*\RR^n$ defined by $\eta_q(v)=\langle v,n_p\rangle$ and $$\secondfund_q(\gamma'(0))=\eta_q(\gamma''(0))$$
 Observe that $\ker\eta_q=T_qS\subset T_q\RR^n$. We refer to $\eta_q$ as the {\em inward unit cotangent vector} for $S$ at $q$. 
 
 Now apply this to the horosphere $S=\Hcal$ in Theorem \ref{betathm}.
 Suppose $v=\gamma'(0)\in T_q\Hcal$. Then $(\deriv_q\pi) v=v$
and $v\in\ker \deriv_q h$, hence  the definition of $\beta$ in Theorem \ref{betathm} implies  $\beta(v)=\deriv_q^2h(v)$. 
Also $\deriv_q h=\lambda(q)\eta_q$ for some $\lambda(q)>0$.
  Using
tangent plane coordinates at $q$ gives
\begin{equation}\label{betaissecondfund}
\beta|T_q\Hcal=\lambda(q)\secondfund_q\
\end{equation}  
In particular since $\bdy\Omega$ is a horosphere:
\begin{proposition}\label{secondfundformmetric}
The restriction of the horofunction metric to $\bdy\Omega$  is conformally equivalent to the second fundamental form of $\bdy\Omega$ in $\AA(\Omega)=\RR^n$.
\end{proposition}
}

  The following elementary fact does not seem to be well known, cf Proposition 1.1 in \cite{Calabi}.
  It  { implies that $G(\ppsi)$ acts conformally on $\bdy\Omega(\ppsi)$. Since the action of $T(\ppsi)$
  on $\bdy\Omega$ is simply transitive
  the second fundamental form is conformally equivalent to a flat metric. 
  }

 \begin{proposition}\label{conformallemma} Suppose $S \subset\RR^n$ is a smooth strictly convex hypersurface and $\tau:\RR^n\rightarrow\RR^n$ is an affine isomorphism
 and $\tau(S)=S'$.
 Then $\tau:(S,\secondfund)\to(S',\secondfund')$ is a conformal map. 
 
 Suppose $\eta_p$ and $\eta'_q$ are the inward unit cotangent vectors to $S$ at $p$, and to $S'$ at $q=\tau(p)$ respectively.
 Then $\tau^*\eta'_q=\alpha\cdot \eta_p$ for some $\alpha=\alpha(p)>0$ and
 $\tau_p^*\secondfund'_q=\alpha\cdot\secondfund_p$.
  \end{proposition}
 \begin{proof} Given $p\in S$ set $q=\tau(p)$. We must show that $\tau^*(\secondfund_q')= \alpha(p)\secondfund_p$ for some $\alpha(p)>0$. 
  Let $H\subset\RR^n $ be the hyperplane tangent to $S$ at $p$.  Translate $H$ infinitesimally so that it intersects $S$ in
 an infinitesimal ellipsoid centered on $p$. This gives a foliation of an infinitesimal neighborhood of $p$ in $S$ by ellipsoids
 which we may identify with the levels sets of  $\secondfund_p$ in $T_p$. Since affine maps send parallel hyperplanes to
 parallel hyperplanes, the foliation near $p$ is sent to the foliation near $q$. If two quadratic forms have the same level sets then  one is a scalar multiple of the other.
 
 More formally,  
 suppose $\gamma:(-\epsilon,\epsilon)\to S$
 is smooth with $\gamma(0)=p$. Then 
 $$(\tau^*(\secondfund_q))(\gamma'(0)) =\secondfund_q((\tau\circ\gamma)'(0))=\eta_q((\tau\circ\gamma)''(0))$$
 Since $\tau$ is an affine map $$(\tau\circ\gamma)''(0)=(d\tau)(\gamma''(0))$$
 Since $d\tau(T_pS)=T_qS$ it follows that $\eta_q\circ d\tau=\alpha\cdot\eta_p$ for some $\alpha=\alpha(p)$.
 Thus $$\eta_q(d\tau(\gamma''(0)))=\alpha\cdot \eta_p(\gamma''(0))=\alpha\cdot\secondfund_p(\gamma'(0))$$
\end{proof}

 \begin{proof}[Proof of Theorem \ref{euc_str}]  The metric $\beta$ with the required properties
  is given by Theorem \ref{betathm} {in the affine patch $\AA(\Omega)$, and the restriction of $\beta$ to 
  { $\bdy\Omega$}} is { conformally equivalent to}
  the second fundamental form, {by Proposition \ref{secondfundformmetric}}.
  \end{proof}

\section{Generalized Cusps are $\ppsi$-cusps}

As mentioned in the introduction, the idea of a {\em cusp} in a projective manifold has evolved in a series of papers.
Recall that if $\Omega$ is properly convex then $[A]\in\PGL(\Omega)$ is {\em parabolic} if  all
the eigenvalues of $A$ have the same modulus and there is no fixed point in $\interior(\Omega)$.
A definition of the term {\em cusp} in a properly convex manifold was first given in Definition 5.2 of \cite{CLT1}.
There, the holonomy of a cusp $C$ consists of parabolics.
 The
definition used there was dictated by the requirement to establish a thick-thin decomposition for
strictly convex manifolds, of possibly infinite volume.  In that
paper the {\em rank} of  $C$ is defined, and {\em maximal rank} is equivalent to $\bdy C$ being compact {\red (see Proposition 5.5 of \cite{CLT1})}.
In this paper we only consider cusps of maximal rank, so we have omitted the term {\em maximal rank} from
statements.

A definition of the term {\em generalized cusp} was first given in \cite{CLT2} Definition 6.1.  
It differs from the definition in the introduction, by using the term {\em nilpotent} in place of {\em abelian}.
 Theorem \ref{vnilisvabelian} at the end of this section shows that  these definitions are equivalent.  


\begin{definition}\label{gencusp}  A {\em g-cusp} (called a generalized cusp in \cite{CLT2})  is a properly convex manifold
 $C=\Omega/\Gamma$ homeomorphic
to $\partial C\times[0,\infty)$ with $\partial C$ a connected closed manifold and $\pi_1C$ virtually {\em nilpotent}
 such that $\partial\Omega$ contains no line segment.
The group $\Gamma$ is called a {\em g-cusp group}. In addition:
\begin{itemize}
{\blue \item[(a)] If $\pi_1C$ is  virtually abelian, then $C$ is a {\em generalized cusp}. }
\item[(b)] If  $\PGL(\Omega)$ acts transitively on $\bdy\Omega$, then $C$  is {\em homogeneous}.
\item[(c)] A {\em cusp} is a generalized cusp with parabolic holonomy.
\item [(d)] A {\em standard cusp} is a cusp that is projectively equivalent to a cusp in a hyperbolic manifold. 
\end{itemize}\end{definition}

{\red Next, we restate some previous results from \cite{CLT1} and \cite{CLT2}, with respect the the terminology in Definition \ref{gencusp}}

\begin{theorem}[Theorem 0.5 in \cite{CLT1}]
\label{maximalcuspsstandard}
Every
cusp in a properly convex  real projective manifold is standard. 
\end{theorem}

Observe that a finite cover of a $g$-cusp is also a $g$-cusp.
 
  \begin{theorem}[{\red Theorem 6.3} in \cite{CLT2}]\label{homogcusp} Every g-cusp 
is equivalent to a homogeneous g-cusp. 
 \end{theorem}
 
\begin{definition} $\UT(n)\subset\GL(n,\RR)$ is the subgroup of upper-triangular matrices with
positive diagonal entries.  
\end{definition}

\begin{definition} An {\em e-group} is a subgroup  $G\subset \GL(n,\RR)$ such that every eigenvalue of
every element of $G$ is positive. If $\Gamma\subset\GL(n,\RR)$ is discrete, a {\em virtual e-hull} for $\Gamma$
is a connected e-group $G\subset\GL(n,\RR)$ such that $|\Gamma:G\cap\Gamma|<\infty$ and 
$(G\cap\Gamma)\backslash G$ is compact.
\end{definition}

Observe that $UT(n)$ is an e-group.
Definitions 6.1 and 6.10, and Proposition 6.12 in \cite{CLT2} imply:
 \begin{proposition}\label{gencuspisVFG} 
Suppose  $P=  \Omega/\Gamma$  is a g-cusp of dimension $n$. 
 Then $\Gamma$ contains a finite index subgroup, $\Gamma_1$,
 that
 is a lattice in the connected {\red nilpotent} group
  $T(\Gamma)=\exp\langle\log (\Gamma_1)\rangle$. Moreover $T(\Gamma)$ is  conjugate
  in $\GL(n+1,\RR)$ into $\UT(n+1)$.
 \end{proposition}
In \cite{CLT2} $\Gamma_1=\core(\Gamma,n)$. The {\em Zariski closure} of $\Gamma$ generally has larger dimension
than $T(\Gamma)$.

\begin{theorem}[Theorem 6.18 \cite{CLT2}] 
\label{uniquehull} If $\Omega/\Gamma$ is a generalized cusp then $T(\Gamma)$ is the unique
virtual e-hull of $\Gamma$. \end{theorem}
{\red\begin{theorem}[Theorem 9.1 \cite{CLT1}]\label{ellipsoid}  Suppose that $\Omega$ is open and
strictly convex of dimension $n$ and $W\subset\SL(\Omega)$
is a nilpotent group that fixes $p\in\Fr(\Omega)$ and acts simply transitively of $\Fr(\Omega)\setminus\{p\}$. 
Then $\Fr(\Omega)$ is an ellipsoid, and $W$ is conjugate to the subgroup
of a parabolic group in $O(n,1)$ that has all eigenvalues $1$. \end{theorem}}

 \begin{definition}[cf. Definition 6.17 in \cite{CLT2}]\label{translationgrpdef} A {\em translation group}
is a connected nilpotent subgroup
 $T\subset\GL(n+1,\RR)$
that is the virtual e-hull  of a g-cusp. \end{definition}

\begin{definition}[page 189  \cite{CLT1}]\label{spacedirections} Given a $1$-dimensional subspace $U\subset V$
 {\red set $p=\PP(U)$} and define
$\Dcal_p:\PP(V)\setminus\{p\}\to \PP(V/U)$ by $\Dcal_p([x])=[x+U]$. 
 The {\em space of directions} of the subset $\Omega\subset\PP(V)$ at $p$ is $\Dcal_p(\Omega\setminus p)$.\end{definition}

\begin{theorem}\label{main} If $G\subset\GL(n+1,\RR)$ is a translation group then  
there exists $\ppsi\in A$ such that {\red $G$ and $T(\ppsi)$ have images
 in $\PGL(n+1,\RR)$ that are conjugate}.  
\end{theorem}
\begin{proof}
{\red Here is an outline. The group $G$ preserves a properly convex domain $\Omega$. The  idea is to build a bundle structure $\Omega\to\opensimplex$ where $\opensimplex$ is the interior of a simplex, and
 the fibers are standard horoballs. At several key places we use the fact that 
$\bdy\Omega$ contains no line segment.

 We can reduce to the case that $G$ is upper-triangular and
nilpotent. 
Then $G=\Image(\rho)$ is block upper-triangular and each block is of the form $\lambda_i\rho_i$ where
$\lambda_i$ is a weight and $\rho_i$ is unipotent.  Define $t+1$ to be the number of blocks.
The diagonal case is easy so we assume $G$ is not diagonal, then $t\le n-1$.

There is a projection 
$\pi:\Omega\to \opensimplex^t$ onto the interior of a simplex of dimension $t$. 
This is obtained by writing  $V\equiv\RR^{n+1}=W\oplus U$ where $W\cong\RR^{t+1}$ is the 
subspace spanned by the union over blocks of the last vector
in each block, and $U$ is spanned by the remaining basis vectors. Then $U$ is $\rho$-invariant (because $G$ is upper triangular)
and so $G$ acts diagonally by the weights on $V/U\cong W$. The vertices of $\opensimplex$ are given by the last vector in each block, so $\opensimplex$ is preserved by $G$.
Then $\pi(\bdy\Omega)$ is the orbit of a point because $G$ acts transitively on $\bdy\Omega$.
It follows that $\pi(\bdy\Omega)=\opensimplex\subset\PP(W)$ otherwise $\dim\pi(\bdy\Omega)<\dim\opensimplex<n$
so $\dim\pi(\bdy\Omega)<\dim\bdy\Omega$ which implies $\bdy\Omega$ contains a line segment and contradicts $\bdy\Omega$ is strictly convex. 

The key fact is that the fiber $\Omega_q:=\pi^{-1}(q)\subset\Omega$ over a point $q\in \opensimplex$ is a standard horoball. 
This is because the subgroup of $G$
that preserves $\Omega_q$ acts transitively on $\bdy\Omega_q$ and is unipotent, thus $\Omega_q$ is
projectively equivalent to a horoball by {\blue Theorem \ref{ellipsoid}}.
  
This implies there is at most one block of dimension bigger than $1$. We can arrange this block is
in the bottom right corner and has trivial weight $1$.
At this point we almost have $G$, it remains to find the coupling term $\sum\psi_i\log x_i$. 

There is a
 short-exact sequence $1\to K\to G\to H\to 1$.
 Here $K\cong\RR^{n-t-1}$ is a standard parabolic group acting on $\Omega_q$ and $H\cong\RR^t$ is the diagonal 
 group acting on $W$. There is a splitting $\sigma:H\to G$ that maps into the 
normalizer (in unipotent upper-triangular matrices), $N$, of $K$. Now $N$ is generated by $K$ and a 1-parameter group, $\Phi$, which turns out to be
the radial flow. This is enough structure to {\blue pin everything down. Here are the details.}

}

By {\red Proposition \ref{gencuspisVFG}},  we may assume $G$ is upper-triangular.
By {\red Propositions 6.23 and} 6.24 in \cite{CLT2}  
$G$ preserves a properly convex domain $\Omega\subset\RR^{n}$ with $S=\bdy\Omega$ strictly convex.
 Moreover $G$ acts simply transitively on $S$. 
Let $\{e_i\ |\ 1\le i\le n+1\}$ be the standard basis of $\RR^{n+1}$. Since $G$ is nilpotent, we may further assume {\red (cf proof of 9.2 in \cite{CLT1})} 
there is a decomposition $V:=\RR^{n+1}=V_1\oplus\cdots \oplus V_{{\red t}+1}$ into $G$-invariant subspaces
such that $V_i$ has ordered basis {\red $\Bcal_i=\{ e_k\ |\ m_{i-1}< k\le  m_i\}$ where $n_i:=\dim V_i=m_i-m_{i-1}$.}
By reordering the standard basis we may assume {\red $n_{i}\le n_{i+1}$}.

 Let $\UT_1(V_i)$ be the group of  unipotent, upper-triangular matrices
of size $n_i$. Then there are  distinct weights $\lambda_i:G\to \RR$  and 
homomorphisms $\rho_i:G\rightarrow \UT_1(V_i)$
so that $G$ is the image of the inclusion map $\rho:G\rightarrow \GL(V)$ given by  \begin{equation}\label{Gmatmainproof}
\rho=\bpmat
\lambda_1\rho_1 & 0 &\cdots  &  & 0\\
 0& \lambda_2\rho_2 & 0 &  \cdots & 0\\
  \vdots & \vdots & \ddots & \vdots & 0\\
0 & 0 & \cdots &  & \lambda_{{\red t}+1}\rho_{{\red t}+1}
\epmat. 
\end{equation}

{\red By scaling we may assume that $\lambda_{t+1}\equiv 1$ and hence that $G\subset\Aff(\RR^n)$}. 
Next, let $U_i=\langle\Bcal_i\setminus\{ w_i\}\rangle$, 
 then $V_i=U_i\oplus \RR\cdot w_i$. 
The subspace $U=\bigoplus U_i$ is preserved by $G$, and
there is a linear projection $\pi:V\rightarrow V/U$.
Define a subspace  $W=\langle w_1,\cdots, w_{{\red t}+1}\rangle\subset\RR^{n+1}$, 
so $\Wcal=\{w_1,\cdots,w_{{\red t}+1}\}$ is an ordered basis of $W$. There is projection $\pi:V \to V/U$
and an isomorphism $W\rightarrow V/U$ defined by $w_i\mapsto w_i+U$, {\red and
$\pi_*:{\mathbb P}(V)\setminus{\mathbb P}(U)\longrightarrow{\mathbb P}(V/U)$ is the induced projection.}

Since $G$ preserves $U$, it acts on $V/U$, and thus on $W$. We denote this action by $\rho_W:G\rightarrow\GL(W)$.
Using the basis $\Wcal$, this action is diagonal and, recalling that $\lambda_{{\red t}+1}\equiv 1$ 
\begin{equation}\label{Hmat}
\rho_W=\bpmat
\lambda_1& 0 &\cdots  &  & 0\\
 0& \lambda_2 & 0 &  \cdots & 0\\
  \vdots &\vdots &\ddots &  & \vdots\\
  0 & \cdots & & \lambda_{\red t} & 0\\
0 & 0 & \cdots &  & 1
\epmat
\end{equation}

There are ${\red t}+1$ projective hyperplanes $P_i$ in ${\mathbb P}(W)\cong\RP^{r}$
each of which contains all but one of the points $[w_i]$. {\red Each of these hyperplanes is preserved by $G$}. The complement of these hyperplanes consists of $2^{{\red t}}$
open simplices.

Since $G$ acts transitively on $S$ it also acts transitively (via $\rho_W$) on $\pi(S)\subset\P(W)$.
Choose $q=[x]\in S$, 
then $\pi q$ is in one of these open simplices{\red, say} $\opensimplex$. Otherwise, since $S$ is preserved by $G$ it follows that
 $\pi S$ is contained in some hyperplane 
$P_i\subset{\mathbb P}(W)$. But this
implies $S \subset {\red \pi_*}^{-1}(P_i)$ which is a hyperplane in ${\mathbb P}(V)$. This contradicts that $S$ is a strictly convex hypersurface in $V$.

{\bf Claim 1} Either $G$ is diagonal, or else $H:=\rho_W(G)$ acts simply transitively on $\opensimplex$. 

The fiber 
${\red \pi_*}^{-1}({\red \pi_*} q)\subset{\mathbb P}(V)\red\setminus\PP(U)$ that contains $q$
is the affine subspace $U_q:=[x+U]$. 
If $S$ is transverse to $U_q$ then $\pi S$ contains an open subset of $\opensimplex$, so $\dim H=\dim\opensimplex$.
But $\pi S$ is the $H$-orbit of a point, and $H$ is the projectivization of a diagonal subgroup, so  $H$
acts transitively on $\opensimplex$. 

Thus we may assume $S$ is not transverse to $U_q$. If a strictly convex hypersurface is
not transverse to a hyperplane, then it is to  tangent to it at one point, so
 $U_q\cap S=q$. Since $G$ acts transitively on $S$ this condition holds
at every $q\in S$.  This implies $\pi|S$ is injective so $\dim\opensimplex\ge\dim S$ thus ${\red t}\ge n-1$. 

If  ${\red t}=n$ then $G$ is diagonal as claimed. Otherwise ${\red t}=n-1$. Since $\pi|S$ is injective, and $\dim S=n-1=\dim\opensimplex$, it follows that $\pi(S)$
contains an open subset of $\opensimplex$.  As before this implies $H$ acts transitively,  which proves claim 1.

In the case $G$ is diagonal since $\dim G=\dim S = n-1$ it follows that $G$ is the kernel of some 
homomorphism $\psi:D\to\RR$ 
where $D$ is the diagonal subgroup of $UT(n+1)\cap\Aff(\RR^n)$. It follows from Remark \ref{positivelemma} that $\psi$ or $-\psi$ is positive. This proves the theorem when ${\red t}=n$.
  
  Henceforth we assume ${\red t}<n$ so $H$ acts simply transitively
on $\opensimplex$. Thus $\dim H={\red t}$ and from Equation (\ref{Hmat}) it follows that  $H\subset\GL({\red t}+1,\RR)$ consists 
of all positive diagonal matrices with $1$
in the bottom right corner.

The projection,  $\pi_{\red *}$, restricts to a $G$-equivariant surjection $\pi_{\Omega}:\Omega\longrightarrow \opensimplex$
and $K=\ker({\red \rho_W})\subset G$ 
acts trivially {\red via $\rho_W$} on $\opensimplex$, and is unipotent. 
Each fiber $\Omega_q:=U_q\cap\Omega=\pi_{\Omega}^{-1}(q)$  is a properly convex set which is
preserved by $K$.
Since $G$ acts simply transitively on $ \bdy \Omega$, it follows that
  $K$ acts simply transitively on $\bdy \Omega_q=U_q\cap \bdy \Omega$ for every $q$.
  Simple transitivity implies that the action of $K$ on $U_q$ is faithful. 
  
 {\red Since the action of $K$ on $\opensimplex$ is trivial,   $[k(x)+U]=[x+U]$ for all $k\in K$. 
    Thus} the subspace $U^+=U\oplus\RR\cdot x\red\subset V$ is preserved by $K$ and 
  $U_q=[U+x]\subset\PP(U^+)\subset\PP(V)$. The action of $K$ on $U^+$ is the restriction
  of the action on $V$, and is therefore unipotent. Moreover $U^+=\red \left(\bigoplus U_i\right)\oplus \RR\cdot x$
  so the action $K$ on $U^+$ is given by $K'=\rho'(K)$  where
  \begin{equation}
\rho':=\rho|U^+=\bpmat
\rho_1|U_1 & 0 &\cdots  &  & & *\\
 0& \rho_2|U_2 & 0 &  \cdots & 0 & *\\
  \vdots & \vdots & \ddots & \vdots & 0 & *\\
0 & 0 & \cdots &  & \rho_{{\red t}+1}|U_{{\red t}+1} & *\\
0 & 0 &  \cdots & & & 1
\epmat
\end{equation}
  
  {\red The notation $\rho|U^+$ means the {\em restriction of the action} of $\rho$ to the subspace $U^+\subset V$,  etc.}
  The properly convex set $\Omega_q=\Omega\cap U_q\subset\PP(U^+)$ is preserved by $K'$. Moreover
  $K'$ is  nilpotent, upper-triangular, and acts simply transitively on $\bdy\Omega_q$.  
  The hyperplane $\PP(U)\subset\PP(U^+)$
  is preserved by $K'$, and
  the point $s\in\bdy_{\infty}\Omega_q=\cl(\Omega_q)\cap\PP(U)$ is fixed by $K'$.
  Also $\Dcal_s\Omega_q=\Dcal_s(\bdy\Omega_q)$, hence $(\Dcal_s\Omega_q)/K'=(\bdy\Omega_q)/K'$ 
  is a single point, and thus 
  compact. It now follows from Theorem 5.7 in \cite{CLT1}
  that $s$ is a round point of $\Omega_q$ {\red (recall from \cite{CLT1} that a point is round if it is $C^1$ and strictly convex in the boundary).}   Hence $\cl(\Omega_q)=\Omega_q\sqcup\{s\}$, and $\Omega_q$ is strictly convex. 

It follows from Theorem \ref{ellipsoid} 
   that $\Omega_q$ is an 
  ellipsoid, and $K'$ is conjugate to  the  parabolic subgroup  
  \begin{equation}\label{Pmatrixmainproof}
   P=\exp\bpmat 0 & y_1 &\cdots & y_u &0\\
0 & \cdots & & & y_1\\
0 & \cdots & & & \vdots\\
0 & \cdots & & & y_u\\
0 & \cdots & & & 0\\
\epmat\subset GL(U^+)
\end{equation}
Hence, $u+2=\dim \red U^+$. If $u=0$ this is the identity matrix, and $K\red \cong K'$ is the trivial group. When $u>0$ then
 $P$ acts affinely on $\RR^{u+1}$. {\red The orbit under $P$ of the origin
is the  paraboloid $y_0=(1/2)(y_1^2+\cdots+y_u^2)$ which is the boundary (minus one point)
 of the parabolic model of $\HH^{u+1}$, and $P$ is
the group of parabolics. In particular if $A$ and $B$ are invariant subspace of $U^+$ then $A\cap B \ne 0$.
  It follows that $\dim U_i>0$ for at most one $i$. Since $\dim V_i\le \dim V_{i+1}$ and $\dim U_i=\dim V_i-1$
then $U_i=0$, and $\dim V_i=1$ for all $i\le {\red t}$. Let $\pi_{t+1}:V\to V_{t+1}$ be the projection given by the direct sum
decomposition. Then   $U^+=U_{t+1}\oplus \RR\cdot x$ and $\pi_{t+1}|:U^+\to V_{t+1}$ is an equivariant isomorphism. 
After a change of basis for $V_{t+1}$}
\begin{equation}
K=\bpmat I_{\red t} & 0\\
0 & P
\epmat\subset\GL(n+1,\RR)
\end{equation}
This formula also holds when $u=0$ since $K$ is then trivial. 
We thus have a short exact sequence 
\begin{equation}\label{groupexact}
\begin{tikzcd}
	1\arrow{r} & K\arrow{r}{incl} & G\arrow{r}{{\red \rho_W}} & H\arrow{r}\arrow[bend left=33]{l}{\sigma} & 1
\end{tikzcd}
\end{equation}
{\red If $t=0$ then $G=K=P=T(\ppsi=0)$ and the result holds. Thus we may assume $t>0$.}
{\red Since $V_i=\RR\cdot e_i$ for $i\le t$ it follows that $V_{t+1}$ has basis $\{e_{t+1},\cdots,e_{n+1}\}.$} 
Since $H\cong {\mathbb R}^{\red t}$ there is a splitting $\sigma:H\to G$. {\red Since $G$ is block-diagonal
and referring to} Equation (\ref{Gmatmainproof})
it follows that  
\begin{equation}
\sigma=
\bpmat
\bpmat
\lambda_1'& 0 &\cdots  &  \\
 0& \lambda_2' & 0 &  \cdots \\
  \vdots &\vdots &\ddots &   \vdots\\
  0 & \cdots & & \lambda_{\red t}' \\
\epmat & 0\\
 0& \psi
\epmat
\end{equation}
where $\ppsi:H\rightarrow \UT_1(V_{{\red t}+1})$, {\blue and $\lambda_i':H\to\RR$ satisfies 
$\lambda_i=\lambda_i'\circ\rho_W$}. Since $K$ is a normal subgroup of $G$ it follows that $\ppsi(H)$
is a subgroup of the normalizer, $N$, of $P$ in $\UT_1(V_{{\red t}+1})$.
Let $\Phi=\exp\RR\cdot a\subset \UT_1(V_{{\red t}+1})$ 
be the one-parameter group where $a$ is the elementary matrix
with $1$ in the top right corner. Then $\Phi$ centralizes $P$. {\red 
Thus $\Phi$ is the {\em radial flow} on $\PP(V_{t+1})$ with center  $\alpha=[e_{t+1}]$ and stationary hyperplane $H=\PP(\langle e_{t+1},\cdots,e_{n}\rangle)$. \blue It remains to show $\sigma$ can be chosen so that $\psi:H\to\Phi$.  Then $K=P(\psi)$ and
$\sigma(H)=T_2(\ppsi)$ thus $G=P(\psi)\oplus T_2(\psi)=T(\ppsi)$.} Let $\mathfrak n, \mathfrak p,\mathfrak k$ be the Lie algebras of $N,P,K$ respectively. The above provides an identification $\mathfrak k \equiv\mathfrak p$.

 {\bf Claim 2}: $N=\langle P,\Phi\rangle$, so   $\mathfrak n=\mathfrak p\oplus {\mathbb R}\cdot a$.  
 
The closures of the orbits of $P$ in $\PP(V_{{\red t}+1})$
consists of a fixed point $\red\alpha$, {\red lines in the} hyperplane $H$ containing $\red\alpha$, and a 
one-parameter family of horospheres, each tangent to $H$ at $\red\alpha$. 
Since $N$ normalizes $P$ it permutes $P$-orbits. 
Thus $N$ 
preserves the fixed set {\red and center} of  $\red\Phi$, and so normalizes the radial flow $\red\Phi$.
Since $N$ is unipotent, $N$ centralizes the radial flow.
The radial flow acts transitively on horospheres,
so if $n\in N$ there is $\red\phi=\exp(t a)\in\Phi$ such that ${\red p}=\phi\circ n$ preserves one horosphere.
But, since $N$ centralizes $\Phi$, this implies that $p$ preserves every horosphere.
Thus $p$ is an isometry of $\HH^{\red u+1}\subset\PP(V_{{\red t}+1})$ where $\dim V_{{\red t}+1}\red=u+2$.
Since $p$ is unipotent it follows that $p$ is parabolic, thus $p\in P$.
So $n= {\red\phi^{-1}}p$, which proves Claim 2.

Taking derivatives $\deriv\sigma:\mathfrak h\to\mathfrak g$.
If $f:\mathfrak{h}\to\mathfrak{k}$ is a homomorphism then exponentiating $f+\deriv\sigma$, gives a 
new section of Equation \eqref{groupexact}, and so without loss of 
generality we may assume that $\deriv\ppsi$ has image in $\R\cdot a$, then $\red \ppsi:H\to\Phi$.

The strictly convex hypersurface 
$\partial\Omega$ is a $G$-orbit. It follows from Remark \ref{positivelemma} that $\psi$ or $-\psi$ is positive. Without loss we may assume $\psi$ is positive, so $\ppsi\in A$. 
{\red The rescaling changes $G$ by central elements of $\GL(n+1,\RR)$,
 and as a result we have only shown that the original
  $G$ and $T(\psi)$ have the same image in $\PGL(n+1,\RR)$.}
  \end{proof}

\begin{proof}[Proof of Theorems \ref{genislambda} and \ref{vnilisvabelian}] Suppose $C'$ is {\red generalized cusp and hence} a $g$-cusp. By Theorem 
\ref{homogcusp},  $C'$ is equivalent to a homogeneous $g$-cusp $C=\Omega/\Gamma$. By Proposition \ref{gencuspisVFG},  and {\red Theorem \ref{uniquehull}},
$\Gamma$  contains a finite index subgroup $\Gamma_1$ that is conjugate to a subgroup of a {\red translation} group, $T$.  By
Theorem \ref{main}, after a conjugacy, $T=T(\ppsi)$ for some $\ppsi$. 
We may assume it is irreducible, then by Lemma \ref{invariantdomains}, $\Omega=g(\Omega(\ppsi))$ for some 
$g\in \SS({\psi})\oplus\Phi^{\psi}$. Thus a conjugate of $\Gamma$ preserves $\Omega(\ppsi)$,
so after conjugacy we may assume $\Gamma\subset\PGL(\Omega(\ppsi))$. 
If $\ppsi\ne0$ then $\PGL(\Omega(\ppsi))=G(\ppsi)$ and by Proposition \ref{invariantEuclmetric}, therefore $C$ is $\ppsi$-cusp.
If $\ppsi=0$
then $\Gamma_1\subset T(0)$ 
so $\Gamma_1\subset G(0)$.  Since $\Gamma/\Gamma_1$ is finite and $\PGL(\Omega(0))/G(0)\cong\RR$
it follows that $\Gamma_1\subset G(0)$ and again $C$ is a $\ppsi$-cusp. This proves Theorem \ref{genislambda}.
It follows $\Gamma$ is virtually abelian, which proves Theorem \ref{vnilisvabelian}.
\end{proof}

\begin{proof}[Proof of Corollary \ref{stdparabolics}] We identify $\pi_1M\equiv\Gamma$.
Since $\delta([A])=0$, for every $\epsilon>0$ there is a loop $\gamma$
in $M$ that has length less than $\epsilon$ and $[\gamma]$ is conjugate in $\pi_1M$ to $[A]$.
It follows that if $X\subset M$ is compact then $[A]$ is represented by a loop in $M\setminus X$.
Thus $[A]$ is represented by a loop in an end of $M$, and therefore in a generalized cusp $C\subset M$
with $C=\Omega(\ppsi)/\Gamma_C$.
Since $\delta([A])=0$, then we can conjugate so $\det A=\pm1$, and then $A\in G(\ppsi)$. 
The result
now follows from Corollary \ref{parabolicsinGpsi}.
\end{proof}

{\red Another consequence of Theorem \ref{genislambda} is that each generalized cusp is equipped with a canonical hyperbolic metric.}

\begin{theorem}[underlying hyperbolic structure]
\label{genishyp} Every generalized cusp $C$,
 with boundary a horomanifold, has a hyperbolic metric $\kappa(C)$ such that 
 $\bdy C$ is the quotient of a horosphere in $\HH^n$.
 If $C'$ is another such cusp, and if $P:C\to C'$ is a projective diffeomorphism,
then $P$ is an isometry from $\kappa(C)$ to $\kappa(C')$. 
\end{theorem}
\begin{proof} Suppose $C$ is a generalized cusp of dimension $n$ bounded by a horomanifold. {\red Then by Theorem \ref{genislambda}} $C=\Omega(\ppsi)/\Gamma$ is {\red equivalent to a} $\ppsi$-cusp. There is a unique horofunction metric, $\beta_C$,
 on $C$ such that the Euclidean volume of $\bdy C$ is $1$. This metric is $\kappa(C)$.
 If $C=\Omega/\Gamma$ and $C'=\Omega'/\Gamma'$ are generalized cusps, and $P:C\to C'$ is a projective diffeomorphism, then $P$ is covered by a projective isomorphism $\Omega\to\Omega'$, which
 is an
isometry between horofunction metrics. Thus $P$ is an isometry.

There is a unique hyperbolic cusp $H$ bounded by a horomanifold, with $\bdy H$ isometric to $(\bdy C,\kappa(C))$.
The restriction of the hyperbolic metric to $\bdy H$ equals the restriction of $\kappa(H)$ to $\bdy H$. Thus
there is an isometry $(C,\kappa(C))\to(H,\kappa(H))$, that identifies $C$ with a hyperbolic cusp.
\end{proof}

\
This raises several questions. For example, 
using this, one can assign a cusp shape, $z\in\CC$, to a generalized cusp
in a 3-manifold. If a hyperbolic 3-manifold with {\red one} cusp can be projectively deformed 
 can this shape change?

\section{Classification of $\ppsi$ cusps}\label{sec:classify}
  This section is devoted to the proofs of Theorem \ref{Classification Theorem} and Corollary \ref{cor:cusplattice}.
  But first we need:

\begin{lemma}\label{equivalentcusps} If $C=\Omega/\Gamma$ and $C'=\Omega'/\Gamma'$ are 
equivalent
generalized cusps of dimension $n$, then $\Gamma$ and $\Gamma'$ are conjugate subgroups
of $\PGL(n+1,\RR)$.
\end{lemma}
\begin{proof} The definition of equivalent cusps given in the introduction is not transitive, 
though it will follow from
 the classification that it is transitive. 
 In this proof we use the equivalence relation {\em generated} by
 the relation on pairs of cusps: {\em
Given $C$ and $C'$ there is a cusp $C''$, diffeomorphic to both of them,
and projective
embeddings {\red of $C''$}, that are also homotopy equivalences,  into both $C$ and $C'$.}
Thus it suffices to prove the lemma when there is a projective embedding of $C$ into $C'$ that is
also a homotopy equivalence.  
We may assume $C\subset C'$ and $\Omega\subset\Omega'$ (this amounts to performing a conjugacy).
Since the embeddings are homotopy equivalences
it follows that $\Gamma=\Gamma'$. 
 \end{proof}
  
  \begin{proof}[Proof of Theorem \ref{Classification Theorem}]  (i) It is clear that $c\Rightarrow a$. Also $a\Rightarrow b$ follows from Lemma \ref{equivalentcusps}.

For (i) $b\Rightarrow c$ and (ii).  Suppose $\Gamma\subset G(\ppsi)$ and $\Gamma'\subset G(\ppsi')$ are
 lattices and $P\in\PGL(n+1,\RR)$ with 
$P\Gamma P^{-1}=\Gamma'$.
 By Theorem \ref{uniquehull} $T(\ppsi)$ is the unique virtual e-hull of $\Gamma(\ppsi)$,
  thus $P T(\ppsi) P^{-1}=T(\ppsi')$. 

 Hence $U=P^{-1}(\Omega(\ppsi'))$ is a properly convex set that is preserved by $T(\ppsi)$.
 Moreover $U$ is irreducible, since this property is preserved by projective maps. By Lemma
  \ref{invariantdomains} there is $g\in \SS({\psi})\oplus\Phi^{\psi}$ such that $g(U)=\Omega(\ppsi)$. 
  Since $g$ centralizes $T(\ppsi)$ we may replace $P$ by $g\circ P$ and assume that $P(\Omega(\ppsi'))=\Omega(\ppsi)$.
  It follows that $P\cdot G(\ppsi)\cdot P^{-1}=G(\ppsi')$ proving one direction of (ii).  The converse of (ii) is obvious.
   If $G(\ppsi)=G(\ppsi')$ then $P$ preserves $\Omega(\ppsi)$ which proves  (i) $b\Rightarrow c$.

(iii)  By Theorem \ref{Gppsistructure}(f),  $T(\ppsi)$ is a characteristic subgroup of $G(\ppsi)$: it is the subgroup
of elements all of whose eigenvalues are positive. 
Thus if $P$ conjugates $G(\ppsi)$
to $G(\ppsi')$ then it conjugates $T(\ppsi)$ to $T(\ppsi')$. By Proposition \ref{transgpscale}
this happens if and only if $\ppsi=t\cdot\ppsi$ for some $t>0$.
 
(iv)  follows from Proposition \ref{invariantEuclmetric}. (v) is done below.
 \end{proof} 
  
   \begin{proof}[Proof of Corollary \ref{cor:cusplattice}] To show $F$ is surjective, suppose $C$ is a generalized cusp of dimension $n$. 
  By Theorem \ref{genislambda} there is an equivalent cusp $\Omega(\ppsi)/\Gamma\in[C]$  
  for some lattice $\Gamma\subset G(\psi)$. By Theorem \ref{Classification Theorem} (i)(c) we may assume $\ppsi(e_1)=1$.
  Then $F([\Gamma])=[C]$
therefore $F:\clat^n\rightarrow\Ccal^n$ is surjective.
  
  To show $F$ is injective, suppose $F([\Gamma_1])=F([\Gamma_2])$ for lattices $\Gamma_i\subset G(\ppsi_i)$.
  By Lemma \ref{equivalentcusps} 
   $\Gamma_1$ and $\Gamma_2$ are conjugate subgroups of  $\PGL(n+1,\RR)$. Then by Theorem \ref{Classification Theorem}(ii) 
  $G(\psi_1)$ and $G(\psi_2)$ are conjugate in  $\PGL(n+1,\RR)$, and by Theorem \ref{Classification Theorem}(iii) 
  this implies $\ppsi_1=t\cdot\ppsi_2$ for some $t>0$. 
  Since $\ppsi_1(e_1)=\ppsi_2(e_1)$ then $\ppsi_1=\ppsi_2$. By Theorem \ref{Classification Theorem}(i)(c) 
  it follows that
  $\Gamma_1$ and $\Gamma_2$ are conjugate subgroups of $G(\ppsi_1)$ so $[\Gamma_1]=[\Gamma_2]$ and $F$ is injective.
   \end{proof}

Corollary \ref{cor:cusplattice} reduces the classification of equivalence classes of generalized cusps
  to the classification of conjugacy classes of lattices in each of the groups $G(\ppsi) \subset \PGL(n+1, \RR)$. This classification corresponds to {\em moduli space} {\red of $G(\ppsi)$}. 
  There is a finer classification using the notion of {\em marking} that results in an analog of {\em Teichmuller space}. We will show that
  a {\em marked} generalized cusp is parameterized by a marked Euclidean cusp, together with a left coset $A\cdot O(\ppsi)\in O(n-1)/O(\ppsi)$ called the {\em anisotropy parameter}.
   The classification of unmarked cusps is more complicated to state. 
   
   One  complication is that in general there are finitely
  many distinct isomorphism types of lattice in $G(\ppsi)$. To make these subtleties clear requires several definitions.
  
   A discrete subgroup $H$ of a Lie group $G$ is a {\em lattice} if $G/H$ is compact. 
  The {\em set of lattices}
   in $G$ is denoted $\lat(G)$. The quotient of this set by the action of $G$ by conjugacy gives the set of {\em conjugacy classes
   of lattices} in $G$ denoted $\clat(G)=\lat(G)/G$. {\red This set is partitioned into isomorphism classes}.
   Given a lattice $H$ in $G$, an {\em $H$-lattice} is a lattice $H'$ in $G$ with $H\cong H'$;
   and the {\em set of $H$-lattices} is the subset $\lat(G,H)\subset\lat(G)$. 
   The set of conjugacy classes of $H$-lattice is $\clat(G,H)=\lat(G,H)/G$ and is a subset of $\clat(G)$.
   
   A {\em marking}
   of an $H$-lattice $H'$ in $G$ is an isomorphism $\theta:H\to H'$, and $\theta$ is also called a {\em marked $H$-lattice}.
      The set of all {marked $H$-lattices} in
    $G$ is  denoted by $\mlat(G,H)$. Thus a lattice is a group, but a marked lattice is a homomorphism, and $\mlat(G,H)$ is 
    the subset of the
   representation variety $\Hom(H,{\red G})$ 
   consisting of those injective homomorphisms with image a lattice {\red of $G$}.
     Let $\Hcal$ be a set of lattices in $G$ that contains one lattice in each isomorphism class.
  The set of marked lattices in $G$ is $\mlat(G)=\cup\mlat(G,H)$ where the union is over $H\in\Hcal$.
   
    Two marked $H$-lattices
$\theta_1,\theta_2:H\rightarrow G$ are {\em conjugate} if
 there is $g\in G$ with $\theta_2=g^{-1}\cdot \theta_1\cdot g$, 
   and the set of {conjugacy classes of marked $H$-lattices} is $\cmlat(G,H)=\mlat(G,H)/G$. 
   The set of conjugacy classes
   of marked lattices in $G$ is $\cmlat(G)=\mlat(G)/G$.
      
As an example,  a lattice in $G=\Isom(\EE^2)$  is 
  a 2-dimensional Bieberbach group (wallpaper group), and there are 17 isomorphism types for $H$. These are also
  the isomorphism classes of compact Euclidean 2-orbifolds.
There is a natural bijection between $\cmlat(\Isom(\EE^2),\ZZ^2)$ and  marked 
 Euclidean structures on a torus $T^2$.
 It is well known that a marked Euclidean torus of area $1$ is parameterized by a point  in the upper half plane $\HH^2$.
  Moreover
$$\begin{array}{rcl}
\cmlat(\Isom(\EE^2),\ZZ^2) & \cong & \RR^+\times\{x+iy\in\CC: y>0\}\equiv\RR^+\times\HH^2\\
 \clat(\Isom(\EE^2),\ZZ^2)& \cong & \RR^+\times\HH^2/\PSL(2,\ZZ)
 \end{array}$$
the $\RR^+$ factor records the area of the torus that is the quotient of $\EE^2$ by the action of the lattice.

 Before proceeding to the proof of Theorem \ref{Classification Theorem}(v) we give an example  for 3-manifolds. For a generic diagonalizable
 generalized cusp Lie group, such as $\ppsi=(3,2,1)$,  then $G(\ppsi)\cong\RR^2$ and
 $O(\ppsi)$ is trivial.
 A $\ZZ^2$-lattice in $G(\psi)$ is  a subgroup $H=\ZZ u\oplus\ZZ v\subset\RR^2$ given by
 a pair of linearly independent vectors $u,v\in\RR^2$. Using the $\ZZ^2$-marking given by $(1,0)\mapsto u$ and $(0,1)\mapsto v$
 shows that the $2\times 2$ matrix $M=(u^t,v^t)$ determines a unique marked lattice, 
 so $$\cmlat(G(\psi),\ZZ^2)\cong\GL(2,\RR)$$
  There is a natural map $\cmlat(G(\ppsi),\ZZ^2)\to\cmlat(\Isom(\EE^2),\ZZ^2)$ and two lattices $M,M'\in GL(2,\RR)$ 
  have the same image
 if and only if  there is $A\in O(2)$ with $AM=M'$. It follows that $$\cmlat(G(\ppsi),\ZZ^2)\cong O(2)\times \cmlat(\Isom(\EE^2),\ZZ^2)$$ 
 This illustrates Theorem \ref{Classification Theorem}(v): a marked lattice in $G(\psi)$ is parameterized by
  a marked Euclidean lattice and a left coset of $O(\ppsi)$. In this
 case  $O(\ppsi)$ is trivial, so the left coset is just an element of $O(2)$.
 
Now consider unmarked lattices. A change of marking is a change of basis in $\ZZ^2$, and
this changes the lattice $M$ to $A.M$ where $A\in \GL(2,\ZZ)$.
 Thus $$\clat(G(\ppsi),\ZZ^2)\cong \GL(2,\ZZ)\backslash\GL(2,\RR)$$
  The left action of $\GL(2,\ZZ)$ on $\GL(2,\RR)$ is 
 free. However  the action of $\GL(2,\ZZ)$ on $\cmlat(\Isom(\EE^2),\ZZ^2)$ is not free: a $\pi/2$ rotation fixes an unmarked square torus.
 Thus 
 $$\clat(G(\ppsi),\ZZ^2)\ncong O(2)\times \clat(\Isom(\EE^2),\ZZ^2)$$
 which means unmarked lattices in $G(\psi)$ are {\bf not}  parametrized by an
 unmarked lattices in $\Isom(\EE^2)$ together with an anisotropy parameter. 

\begin{proof}[Proof of \ref{Classification Theorem}(v)]  For this proof we will identify $G(\psi)$ with 
the subgroup $\RR^{n-1}\rtimes O(\psi)$ of $\Isom(\EE^{n-1})$.
Since $\Isom(\EE^{n-1})/G(\psi)\cong O(n-1)/O(\psi)$ is compact,
every lattice in $G(\psi)$ is also a lattice in $\Isom(\EE^{n-1})$.
Let
 $\cmlat(\Isom(\EE^{n-1}),\ppsi)\subset\cmlat(\Isom(\EE^{n-1}))$ be the subset of conjugacy classes
 of lattice with rotational part in $O(\ppsi)$.
 The map $\pi:\mlat(G(\ppsi))\to \cmlat(\Isom(\EE^{n-1}),\ppsi)$ is surjective.
Choose a {\red right} inverse
$$\sigma: \cmlat(\Isom(\EE^{n-1}),\ppsi)\to\mlat(G(\ppsi))$$ so $\pi\circ\sigma=id$, 
  and define $
\param:\cmlat(\Isom(\EE^{n-1}),\ppsi)\times(O(n-1)/O(\ppsi))\rightarrow\cmlat(G(\ppsi))
$ by
 \begin{equation}\label{thetadef}
\param([\theta],g.O(\ppsi))=[g^{-1}\cdot\sigma([\theta])\cdot g].
\end{equation}
{\red This map is well defined because the equivalence class in $\cmlat(G(\ppsi))$ of a lattice
is not changed by an $O(\ppsi)$-conjugacy.}
Then Theorem \ref{Classification Theorem}(v) is the assertion that  $\param$ is a bijection.
Set $\Lcal=\Image(\sigma)$
then $\Lcal$ is a set of marked lattices in $G(\ppsi)$  that contains one representative of 
each $\Isom(\EE^{n-1})$-conjugacy class.
There is a  map 
$$
\widetilde{\param}:\Lcal\times \Isom(\EE^{n-1})\rightarrow\cmlat(G(\ppsi))
$$ given by 
$\widetilde{\param}(\theta,g)=[g^{-1}\circ \theta\circ g]$ 
which is obviously surjective.
Observe that   $\widetilde{\param}(\theta_1,g_1)=\widetilde{\param}(\theta_2,g_2)$ if and only if
$$g_1^{-1}\circ\theta_1\circ g_1 =k^{-1}\circ(g_2^{-1}\circ \theta_2\circ g_2)\circ k$$ for some $k\in G(\ppsi)$. This is equivalent to
$$\theta_1=g^{\red-1}\circ \theta_2\circ g\quad \text{\rm with}\quad g= g_2\circ k\circ g_1^{-1}$$
Thus $\theta_1,\theta_2$ are conjugate. This implies the domain of {\red both} $\theta_1$ and of $\theta_2$ is the same
lattice $H\in\Hcal$. 
 Since  $\theta_1,\theta_2\in\Lcal$  it follows that $\theta_1=\theta_2=\theta$ and
\begin{equation}
\label{thetaeqtn}
\theta=g\circ\theta\circ g^{-1}
\end{equation}
Therefore $g$ centralizes the 
lattice $\Gamma=\theta(H)$. It follows that  $\widetilde{\param}(\theta_1,g_1)=\widetilde{\param}(\theta_1,g_2)$ if and only if
there is $\theta\in\Lcal$ and
$k\in G(\psi)$ such that $\theta_1=\theta_2=\theta$ and   $g=g_2\circ k\circ g_1^{-1}$ centralizes $\Gamma$.
Observe that if marked lattices are replaced by (unmarked) lattices we can only conclude at this point that $g$ {\em normalizes} $\Gamma$.

We can express $g\in\Isom(\EE^{n-1})$ uniquely as a pair $g=(A,v)\in O(n-1)\times\RR^{n-1}$ where $g(x)=Ax+v$, and $A$
is called the {\em rotational part} of $g$. Indeed, if $g_1(x)=A_1x+v_1$ and $g_2(x)=A_2x+v_2$ and $k(x)=Bx+v$ with $B\in O(\ppsi)$ then
\begin{equation}\label{Teqtn}
g(x)=g_2\circ k\circ g_1^{-1}(x) = A_2BA_1^{-1}x +(v_2-A_2BA_1^{-1}v_1+A_2v) 
\end{equation}

By Bieberbach's first theorem, \cite{Charlap}, the subset of the lattice $\Gamma$  consisting of pure translations is
a finite index subgroup, $\Gamma_t\subset \Gamma$
 that is also a lattice in $\RR^{n-1}$. Thus $\Gamma_t$ is centralized by $g$. This means the rotational part
of $g$ preserves an ordered basis
of $\RR^{n-1}$. 
An element of $O(n-1)$ that preserves an ordered basis of $\RR^{n-1}$ is trivial, 
hence the rotational part of $g$ is trivial, so $A_2BA_1^{-1}=I$,
and 
\begin{equation}\label{geqtn}
g(x)=x +(v_2-v_1+A_2v)
\end{equation}

It follows that $\widetilde{\param}(\theta,g_1)=\widetilde{\param}(\theta,g_2)$ if and only if 
$g_1=(A_1,v_1)$ and $g_2=({\red A_2=}A_1B^{-1},v_2)$  and there is $v\in\RR^n$ such that $g=(I,v_2-v_1+A_2v)$ centralizes $\theta$.
If we choose $v=A_2^{-1}(v_1-v_2)$ then $g=(I,0)$ centralizes $\theta$. It follows that
$\widetilde\param(\theta_1,(A_1,v_1))=\widetilde\param(\theta_2,(A_2,v_2))$ if and only if  $\theta_1=\theta_2$ and $A_2\in A_1 O(\ppsi)$. In other words, $\widetilde \Theta(\theta_1,g_1)=\widetilde \Theta(\theta_2,g_2)$ if and only if $\theta_1=\theta_2$ and $g_1G(\ppsi)=g_2G(\ppsi)$.  As a result $\widetilde \Theta$ induces a bijection 
$$
\Theta':\Lcal\times \Isom(\EE^{n-1})/G(\ppsi)\to \cmlat(G(\ppsi))
$$
Observe that $\Isom(\EE^{n-1})/G(\ppsi)\cong  O(n-1)/O(\ppsi)$.
By definition of $\Lcal$, there is a bijection  $$(\pi|\Lcal):\Lcal\rightarrow\cmlat(\Isom(\EE^{n-1}),\ppsi)$$
given by
 $\theta\mapsto[\theta]$. 
 Thus $\param'$ factors through the bijection $\param$ in  Equation (\ref{thetadef})
completing the proof. \end{proof}

\section{Hilbert Metric in a generalized cusp}\label{hilbmet}

 In this
section we describe  how the Hilbert metric of a {\red horomanifold changes  as it is pushed
 out into the cusp} by the radial flow. 
 {\red In the following discussion
 volume means {\em Hausdorff measure}.}
The horomanifolds shrink {\red as they flow into a cusp}, although not uniformly in all directions. Parabolic directions (which only exist when $\ur>0$)
 shrink exponentially with distance out into the cusp, but hyperbolic directions shrink
 towards a limiting positive value. Hence the volume of the cusp cross-section
  (horomanifold) goes to zero exponentially fast
 when $\ur>0$, and the cusp has finite volume. 
 When $\ur=0$ the cusp cross-section 
  converges geometrically to compact $(n-1)$-manifold, and in this case the cusp has infinite volume.

If $\Omega$ is an open properly-convex set  in $\RPn$  the {\em Hilbert metric} on $\Omega$ is defined as follows. 
Suppose $p,q\in\Omega$ lie on the line $\gamma:[a,b]\to\RPn$ 
given by  $\gamma(t)=[(t-a)\vec u+(b-t)\vec v]$  
with endpoints $[\vec u],[\vec v]\in\bdy\overline\Omega$ and interior in $\Omega$. 
If $p=[\gamma(x)]$ and $q=[\gamma(y)]$ then
\begin{equation}\label{eq:Hilbertmetric}
d_{\Omega}(p,q)=\frac{1}{2}\left|\log\left( \frac{\abs{b-x}\abs{y-a}}{\abs{b-y}\abs{x-a}}\right)\right|
\end{equation}
Since cross ratios are preserved by projective transformations   
this is independent of the choice of $\gamma$. 
This is a Finsler metric. For  vectors  tangent to this line, the {\red Hilbert norm is the} Finsler norm that
 is the pullback of the
 Riemannian metric on $(a,b)$ given  by 
 \begin{equation}\label{Finslernorm}
 \frac{1}{2}\left(\frac{1}{|x-a|}+\frac{1}{|b-x|}\right)|dx|	
 \end{equation}


If $p\ne q$ are two points in $\Omega(\ppsi)$
then $q-p\in\RR^n$ is called a {\em parabolic direction at $p$} if there is $A\in P(\ppsi)$ with $A(p)=q$.
{\red It follows that $p$ and $q$ lie in the same horosphere.}
The infinitesimal version of this is that a  {\em parabolic tangent vector} is a vector $v\in T_p\Omega$
that is tangent to the orbit of  point $p$ 
under the action of a 1-parameter subgroup of $P(\ppsi)$.
If $\ur=0$ there are no parabolic directions, and
if $\type=0$ then every vector tangent to a horosphere is a parabolic direction. In general the parabolic directions
correspond to the $y$-coordinates in $(x,z,y)$ coordinates.

{\red \begin{lemma}\label{parabolicdirectionsshrink} Let $\red\Phi$ be the radial
flow on $\Omega=\Omega(\ppsi)$ and $\type=\type(\ppsi)$ and $n=\dim\Omega$.
Suppose $p\ne q\in\bdy\Omega$ and {\red for $t>0$} define $p_t=\Phi_{-t}(p)$ and $q_t=\Phi_{-t}(q)$ and
 $f(t)=d_{\Omega}(p_t,q_t)$.
Then there is $\gamma>0$
\begin{itemize}
\item [(a)]If $\type<n$ then $d_{\Omega}(p_1,p_t)=|\log \sqrt{t}|$, and if $\type=n$ then $d_{\Omega}(p_1,p_t)=t/2.$
\item[(b)] $f(t)$ is  a decreasing function of $t$.
\item[(c)] If $q-p$ is not a parabolic direction at $p$ then  $\lim_{t\to \red\infty}f(t)=\gamma$.
\item[(d)] If $q-p$ is a parabolic direction at $p$ then
$\lim_{t\to\infty}f(t)\cdot  \exp(d_{\Omega}(p_1,p_t))= \gamma$.
\end{itemize}
\end{lemma}
\begin{proof}  (a)  follows from a simple computation using \eqref{eq:Hilbertmetric}.
First assume $\type<n$ so
the radial flow is $\Phi_t(x,z,y)=(x,z-t,y)$ and moves points in the $z$-direction which we call the {\em vertical} direction.
 In this case (b)  follows from \cite{Socie} and also Lemma 1.11 in \cite{CLT1}. 
Let $I_t\subset\RR^n$ be the intersection with $V_{\ppsi}$
 of the line containing $p(t)$ and $q(t)$.  Then $I_t=\Phi_{-t}(I_0)$ because $\Phi$ preserves $V_{\ppsi}$.
 
Observe that $I_t$ is a complete affine line
if and only if $q-p=(0,z,y)$ in $(x,z,y)$ coordinates, which is equivalent to $q-p$ is a parabolic direction.
{\red The subinterval $J_t=I_t\cap\Omega$ contains $p_t$ and $q_t$}.
Thus if $I_t$ is not a complete line, then 
 $$f(t)=d_{\Omega}(p_t,q_t)=d_{J_t}(p_t,q_t)\ge d_{I_t}(p_t,q_t)=d_{I_0}(p_0,q_0)>0$$ 
 Hence if $f(t)\to0$ then  $q-p$ is  parabolic. Since $f(t)$ is decreasing this proves (c).
 In fact it is easy to check that $f(t)\to d_{I_0}(p_0,q_0)$.

Now suppose $q-p$ is parabolic.
 Then $p(t)=p+t e_{r+1}$ and $q(t)=p(t)+q-p$.
 Let $P\subset\RR^n$ be the affine 2-plane  containing the two flow lines $p(t)$ and $q(t)$.
Since $q-p=(0,z,y)$ it follows that $x_i$ is constant on $P$ for $i\le r$.
Then Equation (\ref{eq:horo}) implies $h_{\ppsi}|P$ is quadratic, so $U:=P\cap\Omega$ 
is a convex set bounded by a parabola. The rays  $p(t)$ and $q(t)$ are vertical in $U$.  

The translation group $T(\ppsi)$ acts by isometries of the Hilbert metric and commutes with the radial flow.
We may apply an element of $T(\ppsi)$ so that $p$ and $q$ have the same $z$ coordinate.
Then we can choose a  $w$-coordinate axis for $P$ in the hyperplane $z=0$ so that
 so that $U=\{(w,z): z\ge w^2\}$. Then $p=(-a,a^2)$ and $q=(a,a^2)$ with $a>0$ and the radial flow
acts by $\Phi_t(w,z)=(w,z-t)$. Then $p_t=(-a,a^2+t)$ and $q_t=(a,a^2+t)$ and the endpoints
of $J_t$ are $(\pm\sqrt{a^2+t}, a^2+t)$. Let $K_t=(-\sqrt{a^2+t},\sqrt{a^2+t})\subset\RR$ then  $d_{\Omega}(p_t,q_t)=d_{K_t}(-a,a)$. For $t$ large by Equation (\ref{eq:Hilbertmetric})
 $$d_{K_t}(-a,a)
 =\frac{1}{2}\left|\log\left( \frac{\abs{a+\sqrt{a^2+t}}\cdot \abs{-a-\sqrt{a^2+t}}}{\abs{-a+\sqrt{a^2+t}}\cdot \abs{a-\sqrt{a^2+t}}}\right)\right|
 \approx2a\cdot t^{-1/2}
$$
Using (a) gives (d). 

When $\type=n$ there are no parabolic directions. In this case
$p(t)$ and $q(t)$ are rays in $\RR^n$ contained in lines through $0$. The closure
of $\RR_+^n$ in $\RP^n$ is an $n$-simplex $\Delta=0*\bdy_{\infty}\Omega$ that contains $\Omega$,
so $d_{\Omega}\ge d_{\Delta}$ and $f(t)\ge d_{\Delta}(p(t),q(t))$.
The rays $p(t)$ and $q(t)$ limit on distinct points  $p_{\infty},q_{\infty}$ in the interior of $\bdy_{\infty}\Delta$,
 and  $d_{\Delta}(p(t),q(t))$ is bounded below by
 the Hilbert distance in $\bdy_{\infty}\Delta$ between $p_{\infty}$ and $q_{\infty}$. This proves (c).
  
 Let $L_t\subset J_1$ be the image of $J_t$ under radial projection from $0$. Since $\Omega$ is convex
 it follows from studying a diagram that $L_t$ increases with $t$. However the images  $p(t)$ and $q(t)$ in $J_1$  are $p_1$ and $q_1$,
 thus
$$f(t)=d_{\Omega}(p_t,q_t)=d_{J_t}(p_t,q_t)=d_{L_t}(p_1,q_1)$$
decreases with $t$, which proves (b).
  \end{proof}

A geodesic $\lambda$ is {\em orthogonal} to a hypersurface $S\subset\Omega$ at the point $x\in\lambda\cap S$
 if for all $y\in\lambda$ and $z\in S\setminus\{x\}$ then $d_{\Omega}(y,x)< d_{\Omega}(y,z)$.

\begin{proposition}\label{RForthogonal} In $\Omega=\Omega(\ppsi)$
the flowlines of the radial flow $\Phi=\Phi^{\ppsi}$ are orthogonal to the  horospheres $\Hcal_r=\Phi_r(\bdy\Omega)$. 

Moreover
$d_{\Omega}(\Hcal_r,\Hcal_s)=(1/2)|\log(r/s)|$ if $\type<n$ and $d_{\Omega}(\Hcal_r,\Hcal_s)=(1/2)|r-s|$ if $\type=n$.
\end{proposition}
\begin{proof} Given  $x\in\interior(\Omega)$,
let $\Hcal_r$ be the horosphere, and $\lambda$ the radial flow line, each containing $x$. Let
 $p=\lambda\cap\bdy\Omega$ then $\Phi_r(p)=x$. The radial flow acts conformally on $\RR^n$: in the parabolic
case by translation, and in the hyperbolic case by homothety. Moreover the radial flow preserves $\lambda$
and permutes the horospheres. Let $P\subset\RR^n$ be the hyperplane tangent to $\bdy\Omega$ at $p$.
Then $\Phi_r(P)$ is parallel to $P$ and tangent to $\Hcal_r$ at $x$.
Let $U$
be the component of $\RR^n\setminus P$ that contains $\interior(\Omega).$ 
Then $U$ is a half-space and  the formula for the Hilbert metric applied to $U$ gives a semi-metric (distinct points can have zero distance) with $d_U\le d_{\Omega}$. The level sets
of $f(y):=d_U(x,y)$ are planes parallel to $P$. Moreover $d_U(x,y)=d_{\Omega}(x,y)$ when $y\in\lambda$. It follows
that if $s<0$ then $\Phi_s(x)$
is the unique point on $\Phi_s(\Hcal_r)$ that minimizes distance to $x$.  The formula follows from
Lemma \ref{parabolicdirectionsshrink} (a)
\end{proof}
}

Given  a metric space $(M,d)$  the $k$-dimensional Hausdorff measure is defined as follows.
If $B(x;r,M)$ is the ball of radius $r$ in $M$ center $x$ then $\nu(B(x,r))=c_k r^k$ where $c_k$ is the volume of the
ball of radius $1$ in $\RR^k$. If $\mathcal S$ is a set of balls
in $M$ then $\nu(\mathcal S)=\sum_{B\in \mathcal S}\nu(B)$. 
Given a subset $X\subset M$ and $\epsilon>0$ define  $\nu_{\epsilon}(X)=\infimum\nu(\mathcal S)$ where the
infimum is over all sets, $\mathcal S$, of balls with radius at most $\epsilon$ that cover $X$. Then define an
outer measure by $\nu(X)=\lim_{\epsilon\to 0}\nu_{\epsilon}(X)$. This gives a measure on $M$ in the usual way,
called $k$-dimensional Hausdorff measure, denoted $\vol_k$.  
If $\alpha$ is an arc in $M$ then $\vol_1(\alpha)$ is the length of the arc.
 We will  use $\vol_{n-1}$
 to measure the size of a horomanifold in a generalized cusp. 

If $M$ is an $n$-dimensional manifold with a Finsler metric then
the measure $\vol_n$  is given by a integrating  a certain $n$-form
called the {\em volume form}.
Suppose $p\in M$  and $B\subset T_pM$ is the unit ball in the given norm. The volume form on $T_pM$
is normalized so that the volume of $B$ is
  the Euclidean volume, $c_n$, of the unit $n$-dimensional Euclidean unit ball.
Thus if $\omega\ne 0$ is an $n$-form on $T_pM$ then
 the volume form $\dvol$ on $T_pM$ is 
$$\dvol=c_n\left(\int_B\omega\right)^{-1}\omega$$ 
This defines a Borel measure $\vol_M$ on $M$ given by 
$$\vol_M(X)=\int_X\dvol$$
For a Riemannian metric this is the usual volume form. For $X\subset M$ we refer to $\vol_n(X)$ 
as its {\em volume} 
written $\vol_n(X;M)=\vol(X)$. For a properly convex projective $n$-dimensional manifold $\vol_n$  is 
also called 
  \emph{Busemann measure}. {\red It is a result of Busemann \cite{Busemann} that Busemann measure equals Hausdorff measure.}

The next result describes how the volume of a subset of a {\red horosphere} shrinks
as it flows out into the end of the cusp using the radial flow. 
The asymptotic 
behavior depends only on the {\red unipotent} rank $\ur$ of the cusp. If $\ur>0$ the volume of the region shrinks 
exponentially with distance
as it flows
out, but if $\ur=0$ the volume stays bounded away from $0$.

{\red
\begin{proposition}\label{radialflowchangesvolume} Suppose $\Omega=\Omega(\ppsi)$
 has unipotent rank $\ur=\ur(\ppsi)$ .  
 Let $\vol_{n-1}$ denote Busemann
  measure on hypersurfaces. Set $\Hcal_t:=\Phi_{-t}(\bdy\Omega)$, let $L:=\Phi_{1-t}:\Hcal_{\green 1}\to \Hcal_t$, and let $\nu_t=L^{-1}_\ast \vol_{n-1}$. The measures $\vol_{n-1}$ and $\nu_t$ on $\Hcal_1$ are absolutely continuous and their Radon-Nikodym derivative, $\kappa(t)$, is constant. Furthermore, 
  there exists   $K>0$ such that for all $t\ge 0$
 	if $\ur=0$  then $\kappa(t)\ge K$, and
if $\ur>0$ then $\kappa(t)\le K\cdot \exp(-d_{\Omega}(\Hcal_1,\Hcal_t))$.
\end{proposition}
\begin{proof}   We may regard  the Hilbert norm for $\Omega$ restricted to $\Hcal_t$  as a normed vector space
 $(\Hcal_t,\|\cdot\|_t)$
because $T=T(\ppsi)$ acts simply transitively on $\Hcal_t$ by isometries of the Hilbert metric. {Then $L$} is $1$-Lipschitz by
Lemma \ref{parabolicdirectionsshrink}(b), and $L$ commutes with $T$ so $L$ is linear. Hence the measures are absolutely continuous and $
\kappa(t)\le 1$ is constant.

If $\ur>0$ then by Lemma \ref{parabolicdirectionsshrink}(d) there is $T>0$, $\gamma>0$, and a vector 
$v\ne 0$ such that  for all  $t\ge T$
$$\|L(v)\|_t\le 2\gamma\cdot t^{-1/2}\|v\|_1$$  Since $\|L(v)\|_t\le \|v\|_1$ for $t\le T$, this inequality holds for all $t>0$ with $2\gamma$ replaced by 
$K=\max(2\gamma, \sqrt{T})$.
The result  then follows for $\ur>0$
 using $t^{-1/2}=\exp(-d_{\Omega}(\Hcal_1,\Hcal_t))$ by Proposition \ref{RForthogonal}.

When $\ur=0$ by Lemma \ref{parabolicdirectionsshrink}(c) there is $\delta>0$, independent of $T$, so that the map $L^{-1}$ is $\delta^{-1}$-Lipschitz. 
Then for $X\subset\Hcal_1$ $$\vol_{n-1}(X)=\vol_{n-1}(L^{-1}(L X))\le (\delta^{-1})^{n-1}\vol_{n-1}(LX)$$ and the result follows with $K=\delta^{n-1}$.
\end{proof}
}

\begin{lemma}\label{balldistanceestimate} There is a decreasing function $\kappa:\RRP\to(1,\infty)$ such that 
$\lim_{x\to0}\kappa(x)=1$ with the following property.
Suppose $\Omega'\subset\Omega\subset\RPn$ are both open and properly-convex. Let $\|\cdot\|'$ and $\|\cdot\|$ be the
Hilbert norms on $\Omega'$ and $\Omega$.
Suppose $p\in\Omega'$ and $d_{\Omega}(p,\Omega\setminus\Omega ')>x$ then 
$\|\cdot\|_p\le \|\cdot\|'_p\le \kappa(x)\|\cdot\|_p$.
\end{lemma}
\begin{proof} Since the definition of the Hilbert metric only involves a line segment, it suffices to prove the result  in dimension $n=1$ with $\Omega=(-1,1)$ and $\Omega'=(-u,u)$ and $0<u<1$.  It is easy to do this.
\end{proof}

\begin{proof}[Proof of Theorem \ref{volthm}] {\red Let $C'$ be a smaller cusp  contained in a larger cusp $N\subset C$ so that $\bdy C'$ and $\bdy N$ are both horomanifolds.
By Lemma \ref{balldistanceestimate} it suffices to prove the theorem  for $C'\subset N$ instead
of $C\subset M$. {\blue  Suppose $V$ is a normed vector space and  $U\subset V$ is a codimension-1 subspace.
Suppose $A$ is a compact subset of $U$, and $v\in V$ satisfies $\|v\|=\min_{u\in U} \|u-v\|$
 then
 $A(v):=\{a+tv:a\in A,\ 0\le t\le 1\}\subset V$ is called a {\em  cylinder} and is diffeomorphic to $A\times I$.

Write $x\sim_K y$ to mean $K^{-1}x\le y\le Kx$. Given $n$ 
there is a constant $K>0$ such that  in a normed vector space $V$ of dimension $n$, if $A$ is a cylinder
 then  (see \cite[\S 5.5]{BurBurIvan})
$$\|v\|\vol_{n-1}(A)\sim_K \vol_n(A(v)),$$

There is a diffeomorphism $f:\bdy C'\times[0,\infty)\to C'$ given by $f(x,s)=\Phi_t(x)$ where $s=t/2$ if $\type=n$ and $s=|\log t|/2$ if $\type<n$.
By Lemma \ref{parabolicdirectionsshrink}(a) $f|x\times[0,\infty)$ has image a flow line parameterized at unit speed. Then $C_t=f(\bdy C',t)$ is compact so $\vol_{n-1}(C_t)<\infty$.

Using $f$ and volume with respect the  Busemann measure on $N$ it follows that
\begin{equation}\label{volumeintegral}
	\vol_n(C)\sim_K\int_1^\infty \vol_{n-1}(C_t)dt
\end{equation}
}

When $\ur>0$ it follows from Lemma \ref{parabolicdirectionsshrink} and Proposition \ref{radialflowchangesvolume} that $\vol_{n-1}(C_t)=O(e^{-t})$, and so the integral converges. 
On the other hand if $\ur=0$, Lemma \ref{parabolicdirectionsshrink} and Proposition \ref{radialflowchangesvolume} implies there is $A>0$ such that $\vol_{n-1}(C_t) > A$ for all $t$, and in this case the volume is infinite. }
\end{proof}

\section{Dimension 2}\label{surfaces}

In this section we describe 2-dimensional generalized cusps in a way that illuminates the higher 
dimensional cases, and can be read before
the rest of the paper.    

A {\em generalized cusp}, $C$, in a properly-convex surface, $M$,  is a
 convex submanifold  $C\cong S^1\times[0,\infty)$ of $M$ with $M\setminus C$
 connected and $\bdy C$ is a strictly convex curve in the interior of $M$.
  Thus $C=\Omega/\Gamma$, where $\Gamma$ is an infinite cyclic group generated
by some element $[A]\in \PGL(3,\R)$, and $\Omega$ is properly-convex, and homeomorphic to a closed disc with
one point deleted from the boundary, and $\bdy\Omega:=\Omega\setminus\interior(\Omega)$ is a strictly convex
curve that covers $\bdy C$. {\red Consideration of the Jordan normal form readily shows: }

\begin{theorem}\label{dim2_class} A generalized cusp has holonomy  conjugate to a group
generated by $[A]$ where  either $A$ is diagonal with three distinct positive eigenvalues, or else
is one of
$$\begin{pmatrix} e^a & 0 & 0 \\
0 & 1 &  1 \\
0 & 0 & 1
\end{pmatrix}
\qquad
\begin{pmatrix}
1 & 1 & 0\\
0 & 1 & 1\\
0 & 0 & 1 
\end{pmatrix} \qquad a\ne 0
$$
\end{theorem} 

We regard $\Aff(\RR^2)$ as a subgroup of $\PGL(3,\RR)$.
For  each $\ppsi=(\ppsi_1,\ppsi_2)\in\RR^2$  with $\ppsi_1\ge\ppsi_2\ge0$
 there is a
 one-dimensional subgroup  $T(\ppsi)\subset\Aff(\RR^2)$
$$  T(\ppsi)=\begin{array} {ccc}
\psi_1\ge\psi_2>0 & \psicoef_1>\psi_2=0 & \psi_1=\psi_2=0 \\
\begin{pmatrix} e^{x} & 0 &0 \\
0 & e^{-x.\psicoef_2/\psicoef_1} & 0 \\
0 & 0 & 1
\end{pmatrix} 
& 
\begin{pmatrix} e^{x} & 0 & 0 \\
0 & 1 &  -\psicoef_1  x \\
0 & 0 & 1
\end{pmatrix}
 &
\begin{pmatrix}
1 & x & x^2/2\\
0 & 1 & x\\
0 & 0 & 1 
\end{pmatrix} 
\end{array}\qquad x\in\RR
$$
The holonomy of a generalized cusp is conjugate in $\PGL(3,\RR)$ into one of these groups. 
The orbit of {\red the basepoint (see Definition \ref{basepointdef})} under each of these Lie groups is a convex curve 
$\gamma$  in $\RR^2$ and the convex hull of $\gamma$ is a properly-convex closed set $\Omega=\Omega(\ppsi)\subset \RR^2$ as shown in Figure {\red\ref{dim2aff}},  that is preserved by the group.

The closure of $\Omega$ in $\RP^2$ is $\overline\Omega=\Omega\sqcup\bdy_{\infty}\Omega$ where
$\bdy_{\infty}\Omega\subset\RP_1^{\infty}$,
and $\bdy_{\infty}=[e_1]$ for $T(0,0)$, and  it is the closed line segment 
$\{[te_1+(1-t)e_2]: 0\le t\le 1\}$
with endpoints $[e_1]$ and $[e_2]$ in the remaining cases
as shown in Figure \ref{dim2proj}.

\begin{figure}[h]
\flushleft
\begin{subfigure}[b]{0.25\textwidth}
\centering
\includegraphics[scale=.7]{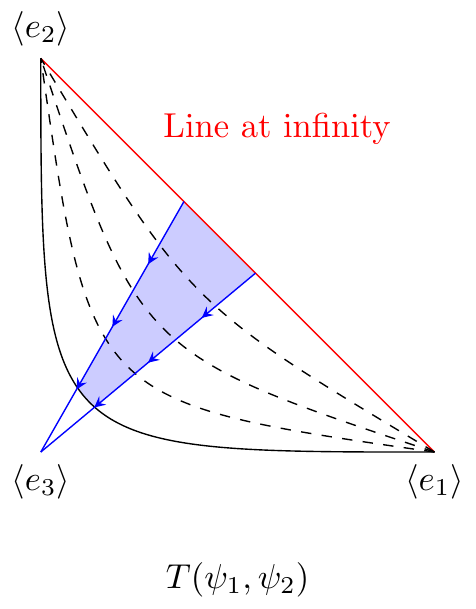}
\end{subfigure}
\begin{subfigure}[b]{0.25\textwidth}
\centering
\includegraphics[scale=.7]{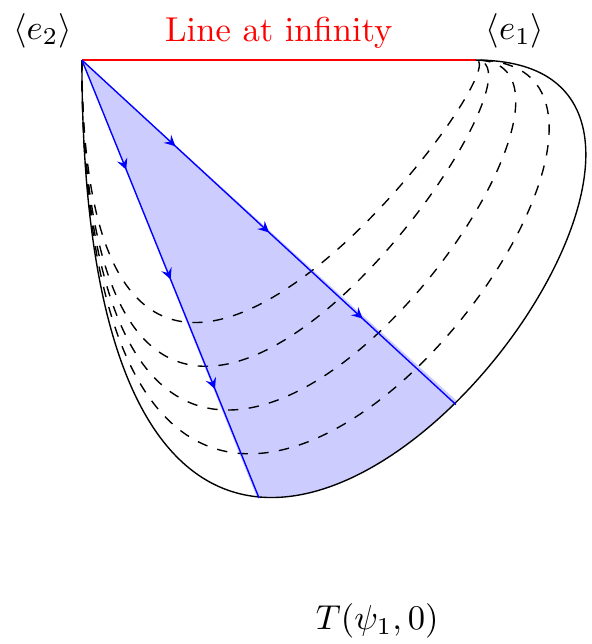}
\end{subfigure}
\hspace{.6 cm}
\begin{subfigure}[b]{0.25\textwidth}
\centering
\includegraphics[scale=.7]{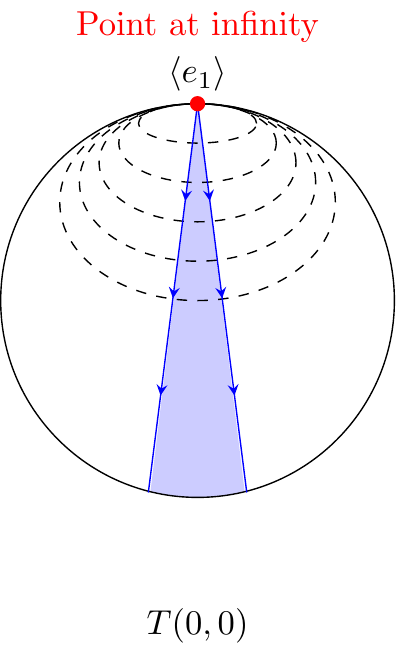}
\end{subfigure}
\caption{Generalized Cusps: Projective View} 
\label{dim2proj} 
\end{figure} 

 Goldman classified convex projective structures on closed surfaces \cite{Goldman2}, and 
 Marquis \cite{Marquis2}, \cite{Marquis1} shows that if $S$ is a finite type surface without boundary,
 then a properly-convex projective structure on $S$ has finite area if and 
 only if the holonomy of each end of $S$ 
is unipotent: conjugate into $T(0,0)$. 

Each domain $\Omega(\ppsi)$ has two foliations that are preserved by $T(\ppsi)$. A {\em horocycle}
is the orbit of a point under $T(\ppsi)$. The {\em radial flow} is a one parameter subgroup
$\Phi^{\ppsi}\subset\PGL(3,\RR)$ that only depends on {\red the type} $\type=\type(\ppsi)$, which is the number of non-zero coordinates of $\ppsi$.  
$$ \begin{array} {cccc}
 & \type=2 & \type=1 & \type=0 \\
\Phi^{\ppsi}(t)= &\begin{pmatrix} 1 & 0 &0 \\
0 & 1 & 0 \\
0 & 0 & e^t
\end{pmatrix} 
& 
\begin{pmatrix} 1 & 0 & 0 \\
0 & 1 &  t \\
0 & 0 & 1
\end{pmatrix}
 &
\begin{pmatrix}
1 & 0 & t\\
0 & 1 & 0\\
0 & 0 & 1 
\end{pmatrix} 
\\
center = & [e_3] & [e_2] & [e_1]
\end{array}$$
This group centralizes $T(\ppsi)$. 
 The $\Phi$-orbit of a (non-stationary) 
point is called a {\em radial flow line} and
is contained in a projective line. All these lines meet at a single point 
called the {\em center of the radial flow}.
The foliation of $\Omega$ by (subarcs of) radial flow lines is transverse to the horocycle foliation. 
The domain $\Omega$ is backwards invariant under 
the radial flow: $\Phi_t(\Omega)\subset\Omega$ for
$t\le 0$.

The group $T(\type):=T(\ppsi)\oplus\Phi^{\ppsi}$ is 
called the {\em enlarged translation group} Equation (\ref{enlargedtranslationgrp})  is
$$  T(\type)=\begin{array} {ccc}
\type=2 &\type=1 & \type=0 \\
\begin{pmatrix} e^{x} & 0 &0 \\
0 & e^{y} & 0 \\
0 & 0 & 1
\end{pmatrix} 
& 
\begin{pmatrix} e^{x} & 0 & 0 \\
0 & 1 & y \\
0 & 0 & 1
\end{pmatrix}
 &
\begin{pmatrix}
1 & x & y\\
0 & 1 & x\\
0 & 0 & 1 
\end{pmatrix} 
\end{array}\qquad x,y\in\RR
$$
and $T(\ppsi)$ is the kernel of a homomorphism $T(\type)\to\RR$  derived from $\ppsi$.

A fundamental domain for a generalized
cusp is obtained by taking an interval  $J\subset\bdy\Omega$ that is a fundamental
domain for the  action there, and taking the backward orbit $\cup_{t\le0}\Phi_t(J)$
 under the radial flow. We now describe
these foliations, 
see Figures \ref{dim2proj} and \ref{dim2aff}.

For $T_0$, the domain $\Omega=\{(x_1,x_2):x_1\ge x_2^2/2\}$, and the horocycles are $x_1=C+x_2^2/2$, and the radial flowlines are $x_2=C$. There is an identification of $\Omega$ with a horoball  $B\subset \HH^2$.  
The action of $T_0$ on $\Omega$
is then conjugated to the action of those parabolic isometries that preserve $B$.
Horocycles in $\Omega$ map to horocycles in $B$ and  radial flow lines in $\Omega$
map to hyperbolic geodesics 
that are  orthogonal to the horocycles. In $\RP^2$ the horocycles for $\Omega$ are ellipses of 
unbounded eccentricity, all tangent at $[e_1]$.

The group $T_2(\psicoef_1, \psicoef_2)$
 preserves the positive quadrant  $\Delta=\{(x_1,x_2): x_1,x_2>0\}$. 
 The domain $\Omega$ is the subset of $\Delta$
with $x_1^{\psicoef_2}x_2^{\psicoef_1}\ge 1$, and is foliated by the horocycles  
$x_1^{\psicoef_2}x_2^{\psicoef_1}=C$. Each horocycle limits
 on the points $[e_1],[e_2]\in\RP^1_{\infty}$
 that are the attracting and repelling fixed points of the holonomy.  The radial flow lines in $\RR^2$
 are straight lines through
 the origin, which is the neutral fixed point of the holonomy.

 For $T_1(\psicoef_1)$ the domain $\Omega=\{(x_1,x_2):x_2\ge  -\psicoef_1\log x_1,\ \ x_1>0\}$.
The horocycles are ${x_2 = -\psicoef_1\log x_1 +C}$. At 
 $[e_2]\in\bdy_\infty\Omega$  the 
 horocycles are transverse to $\bdy_{\infty}\Omega$, but at $[e_1]$ they are tangent to
 $\bdy_{\infty}\Omega$.   The radial flow lines are the straight lines  $x_1=C$.

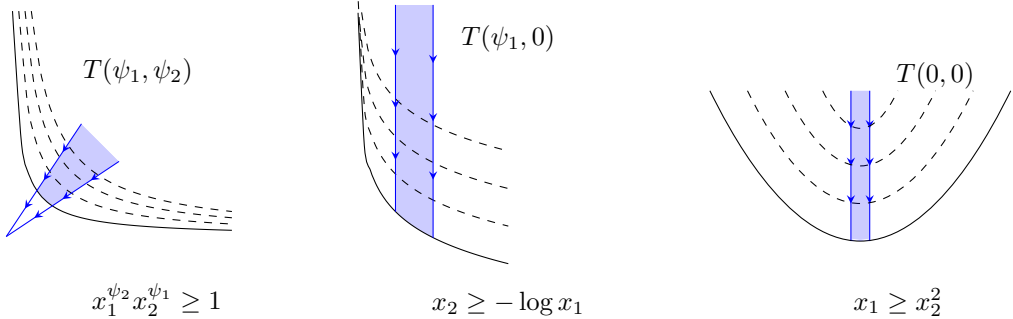
\begin{figure}[h]
\flushleft
\addtocounter{figure}{-1}
\begin{subfigure} [b]{0.3\textwidth}
\begin{tikzpicture}[scale=.5] 
 \draw[ domain=1/6:6, smooth,variable=\x] plot ({\x},{1/(\x)});
  \draw[dashed, domain=2/6:6, smooth,variable=\x] plot ({\x},{2/(\x)});
   \draw[dashed, domain=3/6:6, smooth,variable=\x] plot ({\x},{3/(\x)});
   \draw[dashed, domain=4/6:6, smooth,variable=\x] plot ({\x},{4/(\x)});
  \draw[blue, postaction={on each segment={mid arrow=blue, mid arrow1= blue, mid arrow2=blue}}] (2, 3)--(0,0);
\draw[blue, postaction={on each segment={mid arrow=blue, mid arrow1= blue, mid arrow2=blue}}] (3, 2)--(0,0);
    \draw (3.5,5) node[anchor=north]{$T(\psicoef_1, \psicoef_2)$};
    \draw (4,-1) node[anchor=north]{$x_1^{\psicoef_2}x_2^{ \psicoef_1} \ge 1$};
  \begin{scope}[on background layer] 
\clip (0,0) -- (2, 3) -- (3,2) --cycle; 
\fill[blue, opacity= 0.2] (0,0) -- (2, 3) -- (3,2) --cycle; 
\end{scope} 
  \begin{scope}[on background layer] 
\clip (0,0) -- (3/2, 8/3) -- (5/2,8/5) --cycle; 
 \fill[ white, domain= 0.25:4, variable=\x] (0,0) -- plot ({\x},{1/(\x)}) ; 
\end{scope} 
\end{tikzpicture} 
\end{subfigure} 
\begin{subfigure}[b]{0.3\textwidth}
\begin{tikzpicture}[scale=.5]
 \draw[ domain=0.005:4, smooth,variable=\x] plot ({\x},{-ln(\x)});
  \draw[dashed, domain=0.01:4, smooth,variable=\x] plot ({\x},{-ln(\x)+1});
   \draw[dashed, domain=0.05:4, smooth,variable=\x] plot ({\x},{-ln(\x)+2});
  \draw[dashed, domain=0.1:4, smooth,variable=\x] plot ({\x},{-ln(\x)+3});
   \draw (4,4) node[anchor=south]{$T(\psicoef_1,0)$};
    \draw (4,-3) node[anchor=south]{$x_2 \ge - \log x_1$};
\draw[blue, postaction={on each segment={mid arrow=blue, mid arrow1= blue, mid arrow2=blue}}](1, 5.5)--(1,0) ; 
\draw[blue, postaction={on each segment={mid arrow=blue, mid arrow1= blue, mid arrow2=blue}}]  (2, 5.5)--(2, -.7); 
\begin{scope}[on background layer] 
\clip  (1, 5.5) -- ( 1, -.2)-- (2, -.7)-- (2, 5.5) --cycle; 
\fill[blue, opacity=0.2, domain= 0.005:4, smooth,variable=\x] (2,5.5)-- (1, 5.5) -- plot ({\x},{-ln(\x)})-- cycle;
\end{scope}
\end{tikzpicture}
\end{subfigure}
\begin{subfigure}[b]{0.3\textwidth}
\begin{tikzpicture}[scale=.5] 
\draw[blue, postaction={on each segment={mid arrow=blue, mid arrow1= blue, mid arrow2=blue }}](.25,4)-- ( 0.25,0) ; 
\draw[ blue, postaction={on each segment={mid arrow=blue, mid arrow1= blue, mid arrow2=blue}}](-.25,4)-- ( -0.25,0) ; 
\draw (-4,4)  parabola bend (0,0) (4,4) ;
\draw[dashed] (-3,4)  parabola bend (0,1) (3,4) ;
\draw[dashed] (-2,4)  parabola bend (0,2) (2,4) ;
\draw[ dashed ] (-1,4) parabola bend (0,3) (1,4);
   \fill [blue, opacity = 0.2 ,domain=-.25:.25, variable=\x]
      (-.25, 0)
    -- (-.25, 4) 
    -- (.25, 4) 
      -- (.25, 0)
     -- plot ({\x}, {\x*\x})
      -- cycle;
    \draw (2,5) node[anchor=north]{$T(0,0)$}; 
    \draw (1,-1) node[anchor=north]{$x_1\ge x_2^2$}; 
\end{tikzpicture} 
\end{subfigure} 
\caption{Generalized Cusps: Affine View} 
\label{dim2aff}
\end{figure}
The subgroup $O(\ppsi)\subset\PGL(\Omega(\ppsi))$
is the stabilizer of a point.
This group is trivial unless $\psicoef_1=\psicoef_2$
in which case $O(\ppsi)\cong\mathbb Z_2$. The action of $O(\ppsi)$ is easily described in homogeneous coordinates on $\RP^2$. When $\lambda=(0,0)$ it is generated by the reflection
$[x_1:x_2:x_3]\mapsto[x_1:-x_2:x_3]$ and otherwise by
$[x_1:x_2:x_3]\mapsto[x_2:x_1:x_3]$. In each case  this preserves 
$\Omega(\ppsi)$. If $\psicoef_1\ne\psicoef_2$ then $\PGL(\Omega(\ppsi))=T(\ppsi)$ and
acts freely on 
$\Omega(\ppsi)$.

In all dimensions, a generalized cusp is determined by a lattice in a generalized cusp Lie group. 
For a surface, a lattice is infinite cyclic, and is 
determined by a nontrivial element 
 of some $T(\ppsi)$ up to replacing the element by its inverse. A {\em marked lattice}
 is a lattice with a choice of basis. Thus conjugacy classes of lattices correspond to moduli space
 and conjugacy classes of marked lattices to Teichmuller space. 
 
There is an equivalence relation on {\em marked} generalized cusps {\red generated by projectively embedding one in another}.
Let $\mathcal T$ be the  (Teichmuller) space  of equivalence classes of {\em marked} generalized cusps for surfaces.
There is an identification of  $\mathcal T$  with a subspace of $\SL(3,\RR)$ modulo conjugacy
that sends a marked generalized cusp to
the conjugacy class, $[A]$, of the holonomy of the chosen generator. {\red The eigenvalues 
$\{\exp(x_1),\exp(x_2),\exp(x_3)\}$ of $A$
determine $[A]$ and satisfy $x_1+x_2+x_3=0$. Thus a generalized cusp is determined by $(x_1,x_2,x_3)$
up to permutations.}
 
 Let {\red $X=\RR^2/S_3$} {\red (closed Weyl chamber)} where we identify $\RR^2$  
with the plane $x_1 + x_2 +x_3 =0$ in $\R^3$, and the  quotient is
by the action of the symmetric group $S_3$ on the coordinates.  
Then $X$ can be identified with the fundamental domain for this action: $X=\{(x_1,x_2,x_3):x_1+x_2+x_3=0,\ \  x_1 \geq x_2 \geq x_3\}$; which can be identified with $Y=\{(y_1,y_2):y_2\ge y_1\ge 0\}$ via $y_2=x_1-x_3$ and $y_1=x_2-x_3$.

\begin{figure}[h]
\centering
\begin{tikzpicture}[scale=.8]
\draw[->] (0,0) -- (0,3.5); 
\draw[->] (0,0) -- (3.5,0); 
\draw (0,0) -- (3.5,3.5); 
\fill [blue, opacity =0.2] (0,3.5) -- (0,0) -- (3.5,3.5) ; 
\draw (5,0) node [anchor= north] {$y_1$}; 
\draw (0,3.5) node [anchor= east] {$y_2$}; 
\draw (-1.3,.5) node [anchor= east] 
{$\bpmat 1 & 1 & 1 \\ 0 & 1  &1 \\ 0 & 0 & 1 \epmat$}; 
\draw (-6,1.4) node [anchor= east] {$\alpha>1$};
\draw (-1.3,2.7) node [anchor= east] 
{$\bpmat \alpha ^2 & 1 & 1 \\ 0 & \alpha^{-1}  &1 \\ 0 & 0 & \alpha^{-1} \epmat$}; 
\draw (4.4,2.8) node [anchor= north] { $ \bpmat \alpha  & 1 & 1 \\ 0 & \alpha  &1 \\ 0 & 0 &\alpha^{-2} \epmat  $}; 
\draw [->] (-1.2, 0.5) -- ( -.1,0) ;
\draw [->] (2.7, 2.1) -- ( 2.5,2.5) ; 
\draw [->] (-1.2, 2.8) -- ( 0,2.8) ; 
\end{tikzpicture} 
\caption{Y $\equiv$  {\red Parameter} space  of 2 dimensional cusps}
\label{def2}
\end{figure}
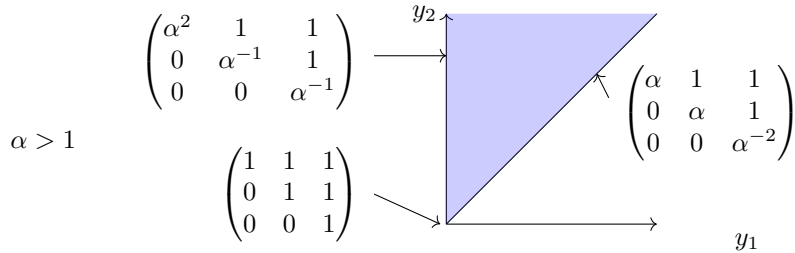

\begin{proposition}\label{defsp2} There is a homeomorphism $f:Y\to\mathcal T$ given by
$$ f(y_1,y_2) = \left[\begin{pmatrix} \exp((2y_2-y_1)/3) &1 & 1\\
0 & \exp((2y_1-y_2)/3) & 1 \\
0 & 0 & \exp((-y_1-y_2)/3) 
\end{pmatrix}\right]
$$
\end{proposition} 
\begin{proof}  By Theorem \ref{dim2_class} the matrix shown determines a generalized cusp. 
Clearly $f$ is continuous.  It is easy to check  that $f$ is surjective. Consideration of eigenvalues
shows $f$ is injective. 

Suppose $A\in \SL(3,\RR)$ and $[A]\in\mathcal T$, then $A$ has real positive eigenvalues.
Let $\lambda_1,\lambda_2,\lambda_3$ be the eigenvalues of $A$ in decreasing order  and 
define $g([A])=(\log \lambda_1,\log \lambda_2,\log \lambda_3)$.
Since the eigenvalues  of a matrix are continuous functions of
the matrix, $g$ is continuous.  But $g$ is the inverse of 
$f$, so $f$ is a homeomorphism.
\end{proof}

The groups $T(\ppsi)$  and $T(\ppsi')$ are conjugate in $\PGL(3,\RR)$ if and only if $\ppsi=t\ppsi'$ for some $t>0.$ It follows
that the space of conjugacy classes of  translation subgroup is the non-Hausdorff space obtained by taking the quotient of $X$ by this
equivalence relation. This is the union of a compact Euclidean interval $[0,1]$ and one extra point which only has one neighborhood.

\section{Dimension 3}\label{3mfd}

Let $C=\Omega/\Gamma$ be an
orientable 3-dimensional generalized cusp,  
then $C$ is diffeomorphic to $T^2 \times [ 0, \infty)$.
Given $\ppsi=(\ppsi_1,\ppsi_2,\ppsi_3)$ with $\ppsi _1 \geq \psicoef_2   \geq \psicoef_3 \geq 0$ there is a  Lie subgroup $G(\ppsi)=T(\ppsi)\rtimes O(\ppsi)$ of $\PGL(4,\RR)$, 
where $T(\ppsi)\cong \RR^2$
is called the {\em translation group}, and $O(\ppsi)$ is compact.
Then $\Gamma$ is conjugate to a lattice in some $T(\ppsi)$, and $\ppsi$ is unique up to 
multiplication by a positive scalar.

The Lie groups $T(\ppsi)$ fall into 4 families, depending on the {\em type} $\type=\type_\ppsi$, which is the number 
 of non-zero components of $\ppsi$.  

  $$
\begin{array}{cc} 
\type=0 & \type=1\\
 \begin{pmatrix}
1 & y_1 & y_2 & \frac{1}{2}( y_1 ^2 + y_2 ^2) \\
0 & 1 &0 & y_1 \\
0 & 0 & 1 & y_2 \\
0 & 0 & 0 & 1
\end{pmatrix} &
\begin{pmatrix} 
e^{x_1} & 0 & 0 & 0\\
0 & 1 & y_1 & \frac{1}{2} y_1 ^2 - \psicoef_1 x_1 \\
0 & 0 & 1 & y_1 \\
0 & 0 & 0 & 1
\end{pmatrix}\\
\\
\type=2 & \type=3\\
\begin{pmatrix} 
e^{x_1} & 0 & 0 & 0 \\
0 & e^{x_2} & 0 & 0 \\
0 & 0 & 1 & - \psicoef_1 x_1 - \psicoef_2 x_2 \\
0 & 0 & 0 & 1
\end{pmatrix} &
\begin{pmatrix} 
e^{x_1} & 0 & 0 & 0\\
0 & e^{x_2} & 0 & 0 \\
0 & 0 & e^{(-\psicoef_1 x_1 - \psicoef_2 x_2)/\psicoef_3} &0 \\
0 & 0 & 0 & 1
\end{pmatrix}
\end{array}$$
  The group $T(\ppsi)$ preserves a properly convex domain $\Omega(\ppsi)\subset\RR^3$ that
 is  the convex hull of the $T(\ppsi)$-orbit of   the {\em basepoint}, (see Definition \ref{basepointdef}). 
It has a foliation
  by convex surfaces called {\em horospheres}, that are $T(\ppsi)$-orbits.
Moreover $\Omega(\ppsi)$ is the epigraph of a convex function, see  Equation (\ref {eq:Logpsif}), and is shown
in Figure \ref{projective3d}.

\begin{center}
 \begin{figure}[ht]
  \includegraphics[scale=.4]{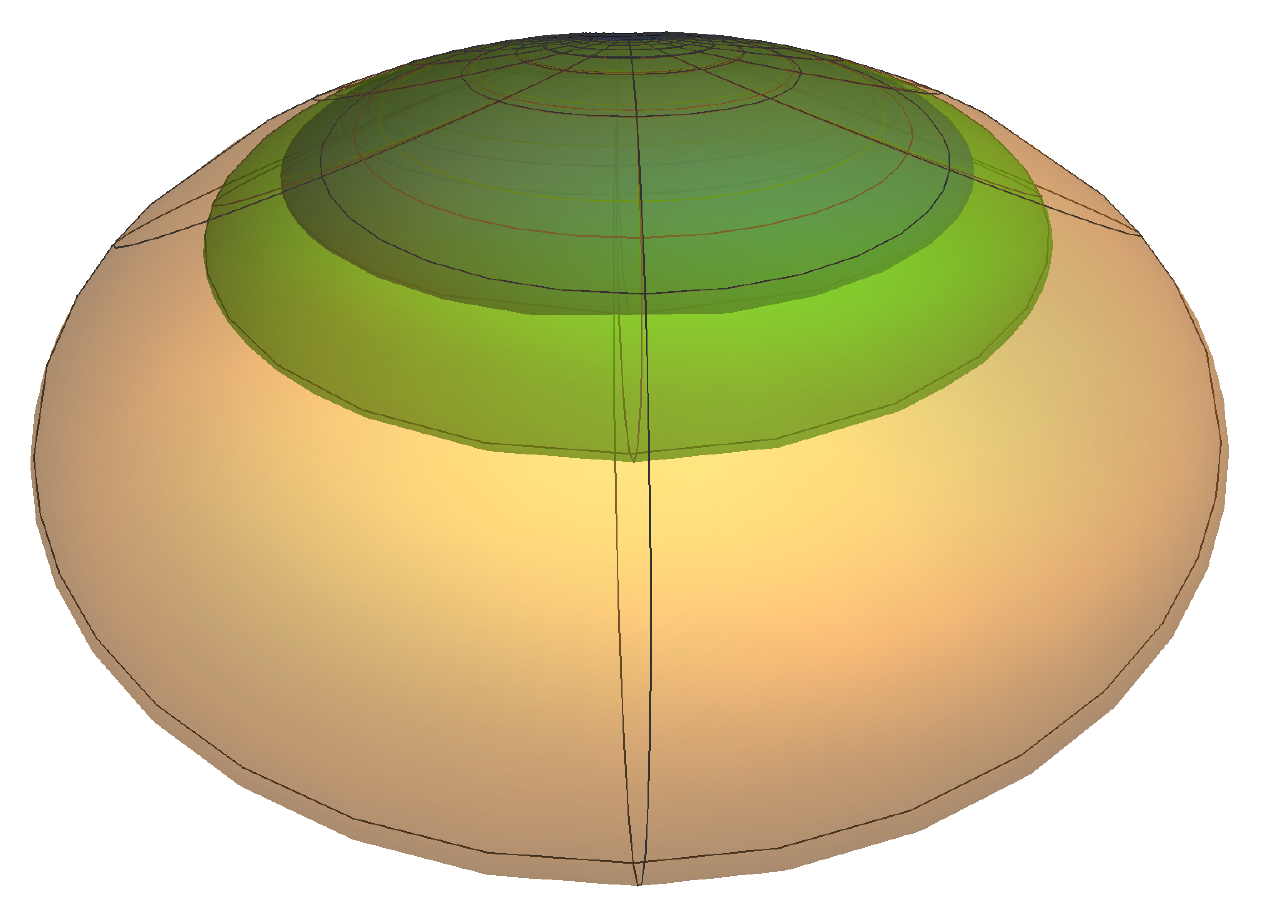}
  \includegraphics[scale=.4]{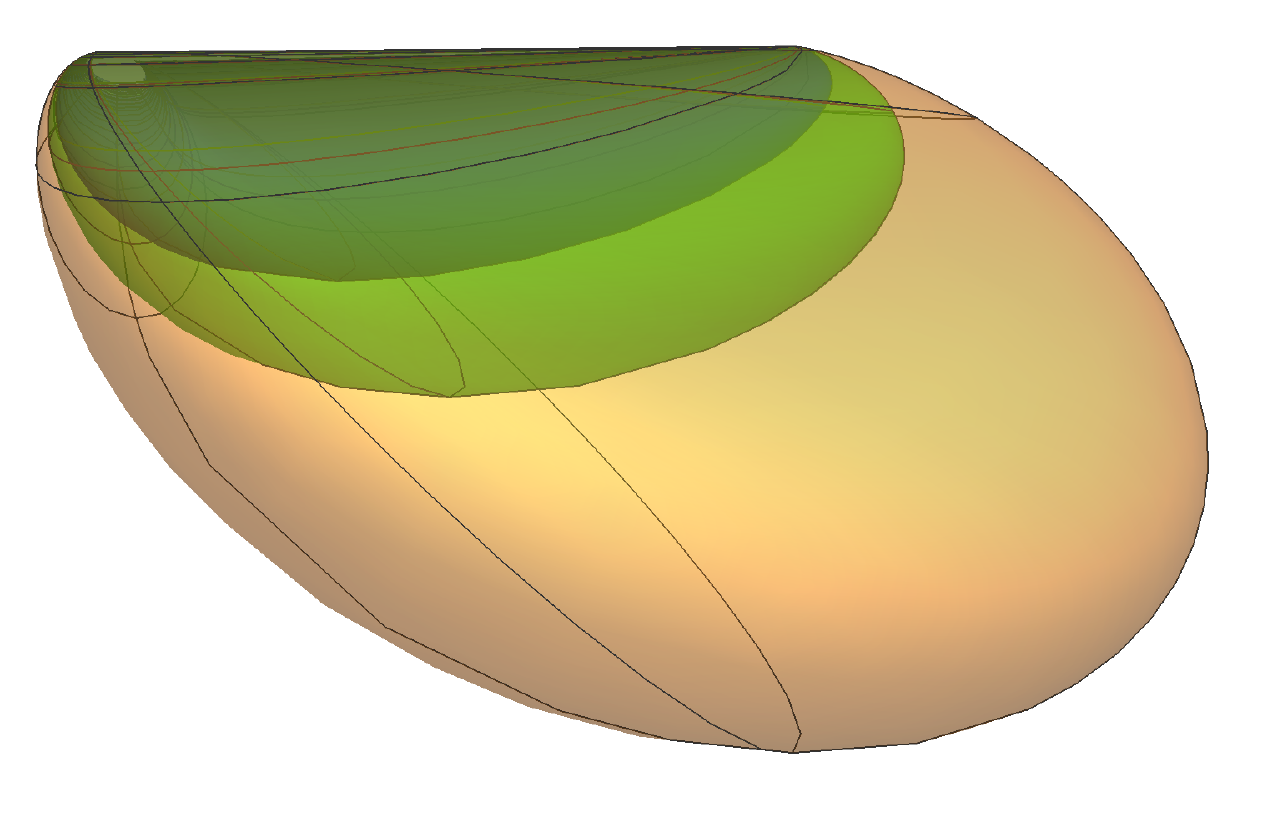}
   \includegraphics[scale=.4]{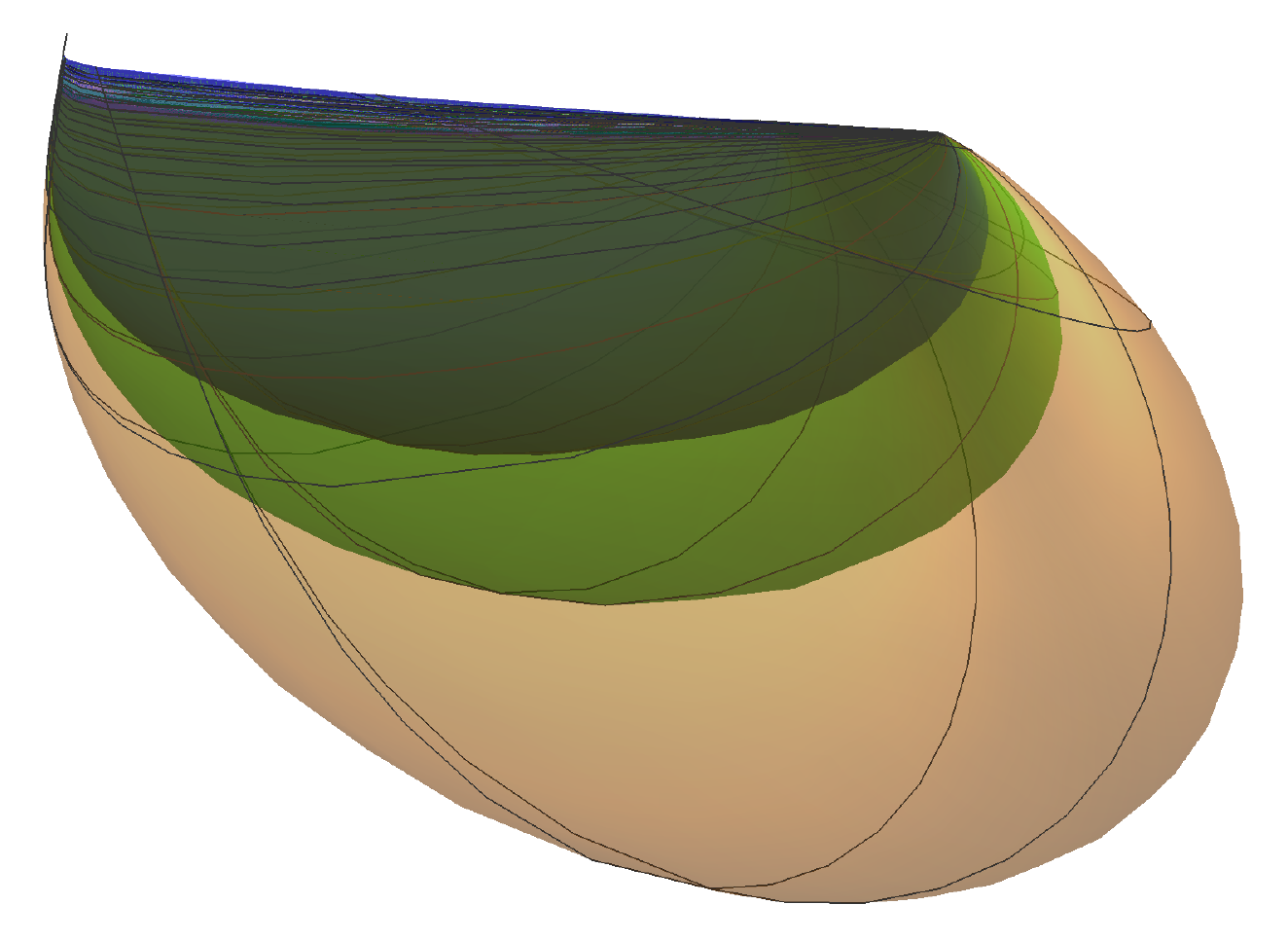}
    \includegraphics[scale=.3]{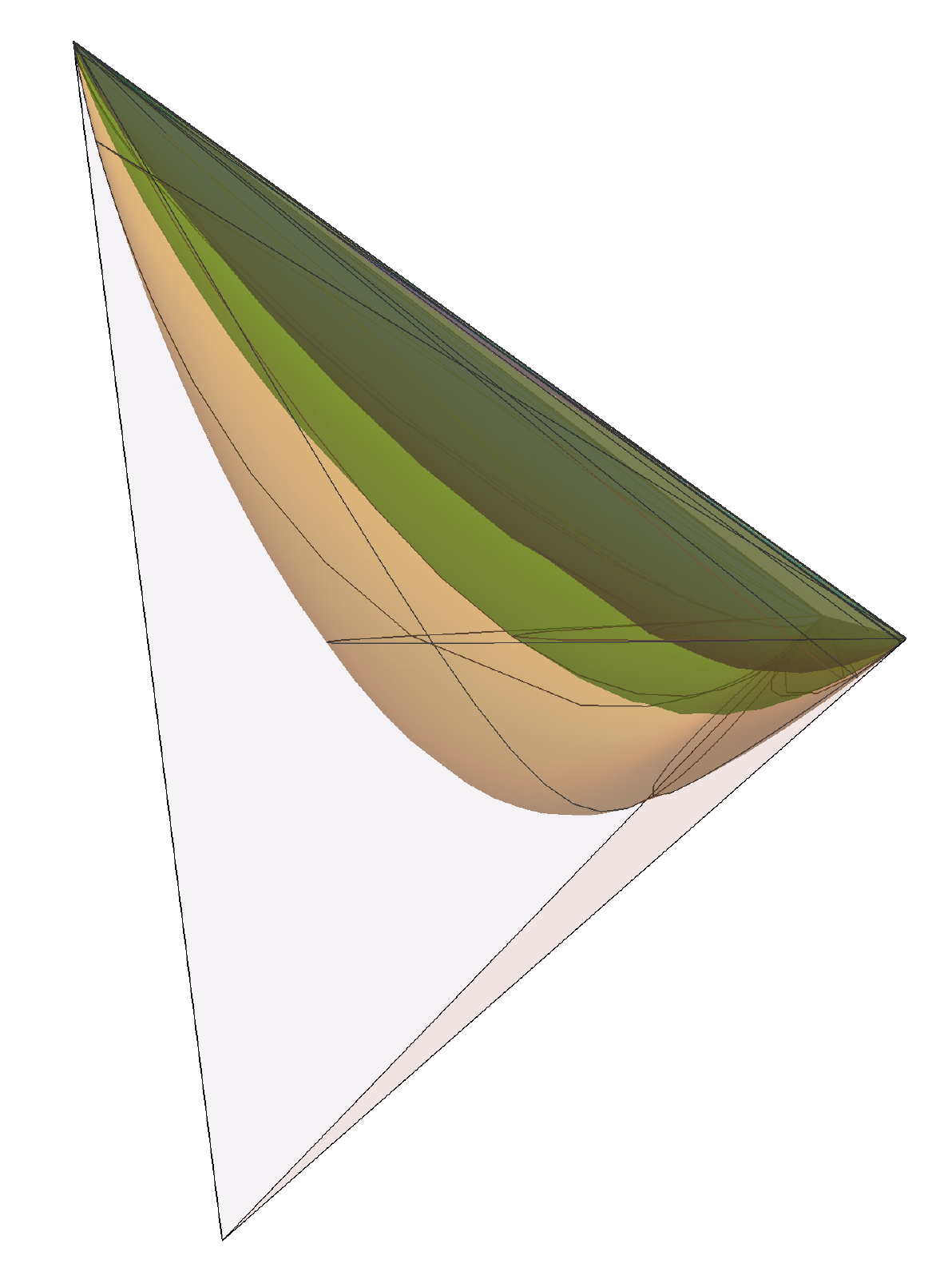}
    \caption{\label{projective3d}3-dimensional generalized cusp domains and their foliation by horospheres in projective space. From left to right, top to bottom the domains are  $\Omega(0,0,0)$, $\Omega(1,0,0)$, $\Omega(1,1,0)$, finally $\Omega(1,1,1)$ is shown inside a simplex }
 \end{figure}
 \end{center}

  The radial flow Equation (\ref{eq:radialflow}) is a one-parameter affine group $\Phi^{\ppsi}$ that
  centralizes $T(\ppsi)$, and $\Phi^{\ppsi}$-orbits give a foliation by a pencil of lines transverse to the horospheres.
  The {\em enlarged translation group}, Equation (\ref{enlargedtranslationgrp}), is $T_\type=T(\ppsi)\oplus\Phi^{\ppsi}\cong\RR^3$. 
  It is obtained by
  replacing the most complicated term in the matrix for $T(\ppsi)$ by $z$. There are 4 such groups, depending only on $\type$. Then
  $T_\ppsi$ is the kernel of a homomorphism $T_\type\to\RR$ obtained from $\ppsi$. The group $T(\type)$
  acts simply transitively on $\RR_+^{\red\type} \times\RR^{3-\red\type}$,  and the latter contains $\Omega(\ppsi)$.

  The group, $O(\ppsi)$, is the subgroup of $G(\ppsi)$ that fixes the basepoint (see Definition \eqref{basepointdef}). It is computed in Proposition \ref{olambda}, and
$O(0,0,0)\cong O(2)$, 
and $O(\ppsi_1,0,0)\cong O(1)$ when $\ppsi_1\ne 0$. 
For the remaining cases,
 $O(\ppsi)$  is the group of  coordinate permutations of $\RR^3$ that preserve $\ppsi$. 
 In particular $O(\ppsi)$ is finite unless $\ppsi=0$.  
 
 There is a 6 parameter family of marked, 3-dimensional generalized cusps.
 As described in Theorem \ref{Classification Theorem} they are parameterized by a triple 
 $(\ppsi,\Gamma,A\cdot O(\ppsi))$
 with $\ppsi$ as above, and $\Gamma$ is a marked lattice of co-area $1$ in $\RR^2$, and $A\cdot O(\ppsi)\in O(2)/O(\ppsi)$ is
  a left coset.
  
   In \cite{Leitner}, the third author showed that in dimension 3, every translation group, 
 as defined in Definition \ref{translationgrpdef}, is conjugate into one of these 4 families. This, together with \cite{BallasFig8},
  provided the {\em impetus} for the present paper.

We now describe some geometric properties of these domains and discuss relevant examples from the literature. The interior of $\Omega(0,0,0)$ is projectively equivalent to $\HH^3$. If $C\cong \Omega(0,0,0)/\Gamma$ then $\Gamma$ is conjugate into $\PO(3,1)$. Cusps of finite volume hyperbolic 3-manifolds give rise to generalized cusps of this type. The ideal boundary, see Equation (\ref{eq:bdyinfinity}), of $\Omega(0,0,0)$ consists of a single point which is stabilized by $G(\ppsi)$, and  $C$ admits a compactification by a singular projective manifold obtained by adjoining this ideal boundary point. 

For a generalized cusp $C=\Omega/\Gamma_C$ modelled on $\Omega=\Omega(1,0,0)$ 
the ideal boundary, $\bdy_{\infty}\Omega$ is a projective line segment $J$. 
The action of $\Gamma_C$ on $I=\interior(J)$ is discrete
iff $\Gamma$ contains a parabolic. In this case $C$ has a compactification 
$\overline C=(\Omega\cup I)/\Gamma_C$
that is a projective manifold that is singular along the circle $S^1=I/\Gamma_C$.

In \cite{BallasFig8}, the first author found, for $t\in [0,\infty)$, a continuous family of properly convex manifolds projectively
equivalent to
$M_t=\Omega_t/\Lambda_t$, and 
diffeomorphic to the figure-8 knot complement, $X=S^3\setminus K$, and $M_0$ is  the complete hyperbolic structure.
Moreover the end of $M_t$ is projectively equivalent to $\Omega(t,0,0)/\Gamma_t$, where $\Gamma_t\subset T(t,0,0)$ is a lattice containing parabolics. As a result, for $t>0$, there is a compactification 
$\overline M(t)=\Omega^+_t/\Lambda_t$ that
is a projective structure on $S^3$ that is singular along $K$, and  $M_t=\overline M(t)\setminus K$
is a properly convex structure on $X$. Here $\Omega^+_t\supset\Omega_t$ and also contains
the $\Gamma_t$-orbit of an open segment in $\bdy_{\infty}\Omega(t,0,0)$.
 The cusp of the hyperbolic manifold $M_0$ deforms to a  generalized cusp of a different type. 
 As the deformation proceeds, an ideal boundary point of $\HH^3$ {\it opens up} into an ideal boundary segment. 
 This is an example of a \emph{geometric transition}; the hyperbolic 
 cusp $\Omega(0,0,0)/\Gamma_0$ geometrically transitions to the non-hyperbolic cusp 
 $\Omega(t,0,0)/\Gamma_t$ as $t$ moves away from zero,  cf. \cite{DancigerThesis} and \cite{CDW}. Higher dimensional examples of hyperbolic manifolds deforming to properly convex manifolds with type 1 cusps can be found in \cite{BallasMarquis}. Furthermore, in subsequent work, the authors will show that every generalized cusp arises as a deformation of a hyperbolic cusp in this way.

The domains of the form $\Omega(\ppsi_1,\ppsi_2,0)$ have  ideal boundary  a 2-simplex, $\Delta$. The interior of one of the edges of  $\Delta$ consists of $C^1$ points, and the remainder of the 1-skeleton of $\Delta$ consists of non-$C^1$ points. In particular, the fixed point of the radial flow 
is the intersection of the two edges of non-$C^1$ points of $\Delta$, see Lemma \ref{characterizecenterRF}. Any lattice in $T(\ppsi_1,\ppsi_2,0)$ acts properly discontinuously on $\Delta$. Thus $C=\Omega(\ppsi_1,\ppsi_2,0)/\Gamma$ has a manifold compactification by adjoining $\Delta/\Gamma$. {\red Recently, Martin Bobb produced the first examples of hyperbolic 3-manifolds with type 2 cusps \cite{Bobbtype2}. Roughly speaking, his examples are constructed by starting with a certain arithmetic hyperbolic 3-manifold and successively bending along a pair of orthogonal totally geodesic hypersurfaces. The first author has also been able to show, using different techniques, that there are infinitely many hyperbolic 1 cusped hyperbolic 3-manifolds that admit properly convex structures with type 2 cusps (see \cite{Ballastype2}).}


Finally, the domains of the form $\Omega(\ppsi_1,\ppsi_2,\ppsi_3)$ also have ideal boundary consisting of a 2-simplex $\Delta$. However, in this case each point of the 1-skeleton of $\Delta$ is a non-$C^1$ point. As in the previous case, if $\Gamma$ is a lattice in $T(\ppsi_1,\ppsi_2,\ppsi_3)$ then $\Gamma$ acts properly discontinuously on $\Delta$ and there is
a compactification of $C$ by adjoining $\Delta/\Gamma$. There are examples of properly convex deformations of the complete hyperbolic structure on finite volume hyperbolic 3-manifolds whose topological ends are of the form $\Omega(\ppsi_1,\ppsi_2,\ppsi_3)/\Gamma$, where $\Gamma\leq T(\ppsi_1,\ppsi_2,\ppsi_3)$. {\red The first such examples were constructed by Benoist \cite{Benoist} using Coxeter orbifolds. These ideas were extended and generalized by Marquis in \cite{Marquismirrors} allowing him to construct further examples. Other} examples were constructed for the figure-eight knot complement and the figure-eight sister by Gye-Seon Lee \cite{Lee}. His examples are constructed by gluing together two projective {\em ideal} tetrahedra using the combinatorial pattern that produces the {\red figure}-eight knot complement (see  Chapter 3 of \cite{Thurston} for details). Subsequent work of the first author, J. Danciger and G-S.\ Lee showed that any finite volume hyperbolic 3-manifold that satisfies a mild cohomological condition  (that is known to be satisfied by infinitely many hyperbolic 3-manifolds, ({\red for example, by applying \cite{HP}, Theorem 1.4 to the Whitehead link}) also admits deformations all of whose ends are projectively equivalent to $\Omega(1,1,1)/\Gamma$, where $\Gamma\leq T(1,1,1)$, thus producing many additional examples.  

Furthermore, as explained in Section \ref{extendedomainssec}, the lack of $C^1$ points in the 1-skeleton of the ideal 
boundary allows properly convex manifolds with ends projectively equivalent to 
quotients of $\Omega(\ppsi_1,\ppsi_2,\ppsi_3)$ to sometimes be glued together to 
produce new properly convex manifold. This idea is explored in detail in \cite{BDL} and 
using these techniques it is possible to find properly convex projective structures on non-hyperbolic 3-manifolds.
 This was first done by Benoist \cite{Benoist} using Coxeter orbifolds.



\end{document}